\documentclass{article}
\usepackage{amsmath}
\usepackage{alltt}
\usepackage{exscale}
\usepackage{amsfonts}
\usepackage{array}
\usepackage{latexsym}
\usepackage{enumerate}
\usepackage{graphicx}
  \usepackage{paralist}
  \usepackage[T1]{fontenc}
  \usepackage{mathtools}
  \usepackage{graphics} %% add this and next lines if pictures should be in esp format
  \usepackage{epsfig} %For pictures: screened artwork should be set up with an 85 or 100 line screen
\usepackage{graphicx} 
 \usepackage{epstopdf}%This is to transfer .eps figure to .pdf figure; please compile your paper using PDFLeTex or PDFTeXify.
 \usepackage[colorlinks=true]{hyperref}
   \usepackage[english]{babel}
\usepackage[latin1]{inputenc}
\usepackage{amsthm,amsfonts,amssymb}
\usepackage{pgf,pgfarrows,pgfautomata,pgfheaps,pgfnodes,pgfshade}
\usepackage{times}
\usepackage{tikz}
\usepackage{color} % a enlever ensuite
\usepackage{subfigure}
\hypersetup{urlcolor=blue, citecolor=red}

\newtheorem{theorem}{Theorem}[section]
\newtheorem{corollary}{Corollary}

\theoremstyle{definition}

\newtheorem{remark}{Remark}

\usepackage{mathtools}
\newcommand{\fracj}{\displaystyle \frac}

\newcommand{\limj}{\displaystyle \lim}

\newcommand{\supj}{\displaystyle \sup}
\newcommand{\intj}{\displaystyle \int}

\newcommand{\minj}{\displaystyle \min}

\newcommand{\beq}{\begin{equation}}
\newcommand{\eeq}{\end{equation}}

\def\a#1{{\mathbb #1}}

\title{ About penalty-duality methods in fluid-structure interactions}

\author{Philippe Destuynder$^*$ and Erwan Liberge$^*$}
\begin{document}
\maketitle

% Enter the first author's name and address:

%\medskip
{\footnotesize
% please put the address of the first author
 \centerline{* Laboratoire LaSIE}
   \centerline{Universit\'e de La Rochelle}
   \centerline{ Avenue Michel Cr\'epeau,  17042 La Rochelle France}
   \centerline{ philippe.destuynder@univ-lr.fr,\;\;\;erwan.liberge@univ-lr.fr}
} % Do not forget to end the {\footnotesize by the sign }

%\medskip

% The name of the associate editor will be entered by an editorial staff
% "Communicated by the associate editor name" is not needed for special issue.
% \centerline{(Communicated by the associate editor name)}

%The abstract of your paper
\begin{abstract} In  fluid-structure interaction problems, some people use a penalty method for  positioning the structure inside the fluid. This is usually performed by considering that the fluid is very stiff or/and very heavy at the place occupied by the structure. These methods are very convenient for the programming point of view but lead to ill conditioned operators. This is a  drawback in the numerical solution methods. In particular the  forces applied to the structure by the surrounding flow are not accurately estimated because the penalty parameter -which is a very large number- appears in their expressions. We suggest in this paper a mathematical analysis of the difficulties encountered and we discuss how the penalty-duality method of D. Bertsekas can be an interesting alternative to overcome them. 
\end{abstract}
\section{Introduction} Let us consider -for sake of brevity- a two dimensional bounded and connected open set denoted by $\Omega$ with  boundary $\Gamma$. The unit normal to $\Gamma$ outward $\Omega$  is denoted by $\nu$. Inside $\Omega$ there is an open subset $S$ with  boundary $\partial S$. The complementary open set of $\overline S$ in $\Omega$ is $\Omega_S=\Omega\setminus {\overline S}$. The set $\overline S$ is occupied by a structure that we assume to be rigid or/and heavy for sake of simplicity. But there is no difficulty for extending our discussion to a flexible structure. The open set $\Omega_S$ is occupied by a viscous fluid that we assume to be incompressible in order to close the set of equations modeling the flow. The open set $S$ is moving with the displacement of the structure. Hence a new discretization of $\Omega_S$ would be necessary in case where one chooses to model separately the two media. Furthermore the transport of the physical field requires some special transformation as a local (or global) Euler-Lagrange parametrization. In addition the approximation  and the transport of the forces  (or stresses) interacting between the two media is a tough problem in terms of precision. \\
\begin{figure}[!htbp]
\center
\begin{tikzpicture}
\draw [white] (-4,0) grid (5,2.5);
\begin{scope}[scale=0.5]

\draw  (-3,0.6) .. controls +(1,0) and +(-1,0) .. (0,1.8) 
             .. controls +(1,0) and +(0,-3) .. (5,3.2)
             .. controls +(0,2) and +(2,0)  .. (0,5.2)
             .. controls +(-1,0) and +(0,3) .. (-4.5,2.2)
             .. controls +(0,-1) and +(-1,0).. (-3,0.6);     
  
\draw [fill,color=gray] (-0.5,3) .. controls +(0,1) and +(-1,.) .. (1,4)
							..  controls +(1,0) and +(0,1) .. (2.5,3)
					        ..  controls +(-1,0) and +(1,0) .. (1,2)
					        ..  controls +(-1,0) and +(0,-1) .. (-0.5,3);
\draw [-latex] (4.9,4) -- (6,4.4);
%\draw [-latex] (2,2) -- (2,3);	
\node [right] at (5,3.9) {$\nu$}	;				        
\node at (-2,2.8) {$\Omega_s$}	;				        
\node at (1,3.5) {$S$}	;
\node at (7,2) {$\Gamma=\Gamma_0\cup\Gamma_1$}	;
\node at (2.2,2) {$\partial S$};
\draw [-latex] (1.8,2.7) -- (1.,3.1);
\node [left] at (1.7,2.4) {$\nu$}	;
\end{scope}
\end{tikzpicture}
\caption{Schematic description of fluid structure interaction problem}
\end{figure}
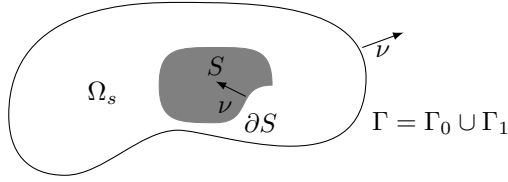

An alternative which has been introduced by several authors  \cite{penalty2} \cite{penalty3}, consists in considering that the two media obey to the same equilibrium laws but differ by different constitutive relationships. More precisely, the coefficients of the constitutive equation are very large  in order to prescribe a rigid body motion in the open set occupied by the structure. Another possibility consists in forcing the velocity field of the fluid to be equal to a rigid body one (the one of the structure) by adding a penalty term using the $L^2(S)$-norm \cite{liberge1} \cite{liberge2} \cite{liberge3}. The advantage of these approaches is to work on a single mesh (but requires some interpolations as far as the structure is moving on this mesh and its boundary has no reason to respect the vertices of the mesh. The second one is a very promising method for particles methods and lattice Boltzmann modeling \cite{liberge1} \cite{liberge2} \cite{liberge3}. Vortex methods are also well adapted to this strategy \cite{author8} \cite{author9} \cite{author10}. 

Several authors \cite{author4} \cite{author5} \cite{penalty1} have suggested to take into account the penalty term only on the boundary of the structure. But the question of the ill conditioning and the one of the mapping between the global mesh and this boundary should be handled in order to improve the strategy, mainly in a multiprocessor programming. Our purpose in this paper is only to discuss mathematically the penalty method with these different possibilities concerning the penalty term.  We choose to remain at the theoretical level in order to point out the mathematical aspects of the problems. Furthermore,  our analysis is performed for a given position of the structure and with a given rigid body velocity of it. In a dynamical model one could consider that it is the case when the classical algorithm (transport prediction-diffusion correction) is used (see P.A. Raviart and V. Girault \cite{PARVG}. Concerning the movement of the structure it is assumed to be sufficiently slow in order to adopt an explicit scheme for its displacements (small reduced frequencies approximation \cite{FUNG}-\cite{PhD1}). 

Let us first state few notations used in the following.
The velocity of the fluid particles is denoted by ${\bold u}=\{u_i\},\;i=1,2$ (the extension to 3D is not a difficulty for the theoretical approach but a real one for the implementation). The strain ratio is $\gamma({\bold u})=\{\gamma_{ij}({\bold u})=\frac{1}{2}(\partial_iu_j+\partial_ju_i)\}$ and the stress ratio is ($\hbox{div}({\bold u})=0$): $$\sigma_{ij}=2\mu\gamma_{ij}({\bold u}),$$ $\mu$ being the viscosity of the fluid. The unit normal to $\partial S$ inward $S$ is also denoted by $\nu$ and the movement of the solid is represented by its rigid body velocity (${\bold k}$ is the unit vector normal to the plan containing $\Omega$; $r$ is the velocity rotation around  ${\bold k}$ and ${\bold a}$ the velocity of point $O$). Hence the the velocity of the structure at point ${\bold x}$ is:
 \beq\label{e1}{\bold d}=\{d_i\}={\bold a}+r{\bold k}\wedge \bold {Ox}=\left\{ \begin{array}{l}a_1-rx_2\\\\a_2+rx_1\end{array}\right \},\;\;{\bold x}=(x_1,x_2)\;\hbox{ and }{\bold a}=(a_1,a_2).\eeq

 The kinematical continuity between the fluid and the structure is traduced by the following relation:
 \beq\label{eq2} {\bold u}={\bold a}+r{\bold k}\wedge\bold {O x}\hbox{ on }\partial S.\eeq

 The flow model inside $\Omega_S$ consists in finding ${\bold u}$ such that ($\Gamma_0$ and $\Gamma_1$ being the two distinct components of $\Gamma$ where respectively the flow velocity -respectively the normal stress- is given and $p$ is the static pressure):
 \beq\label{eq3}\left\{ \begin{array}{l}
\sigma({\bold u})=\mu \gamma({\bold u}),\;\bold {div}(\sigma)-\bold {grad}(p)=0\;\hbox{ in }\;\Omega	_S,\\\\
{\bold {div}}({\bold u})=0\;\hbox{ in }\;\Omega	_S,\\\\
 {\bold u}={\bold u}_0\hbox{ on }\;\Gamma_0\subset \Gamma,\;\sigma.\nu-p\nu={\bold g}\hbox{ on }\Gamma_1\subset \Gamma,\\\\
 \hbox{ and finally the condition stated in equation (\ref{eq2}) on $\partial S$.}
 \end{array}
 \right.
 \eeq

This a quasi-static flow model (following the terminology of aeroelasticity and hydroelasticity \cite{FUNG}-\cite{PhD1}) because the inertia forces aren't taken into account in the flow. 
 
 The equations of the model for the structure are the following ones (let us notice that: $J_O=J_G+M\vert\vert \bold{OG}\vert\vert^2$ where $J_G$ is the inertia around the center of mass $G$ and $J_O$ the one around the point $O$ and finally the dot stands for the scalar product or more generally for the matrix product):
 \beq\label{eq4}\left \{\begin{array}{l}M\dot {\bold a}+M\dot r {\bold k}\wedge {\bold {OG}}+Mr {\bold k}\wedge {\bold a}-Mr^2 \bold{ OG}=\intj_{\partial S}p \nu-\sigma.\nu,\\\\
J_O\dot r+M(\dot {\bold a},{\bold k},\bold {OG})+Mr({\bold a}.\bold {OG})+2Mr({\bold k}\wedge \bold{OG})(\bold {k}\wedge \bold{a})\\\\\hskip 4cm=\intj_{\partial S}(p \nu-\sigma.\nu)\wedge \bold{Ox},\\\\
 + \hbox{initial conditions on ${\bold a}$ and $r$.}
 \end{array}\right.
 \eeq
 \begin{remark}\label{rem1}In fact, we have considered in the previous modeling, that the movement of the structure is quite slow compared to the velocity necessary for establishing the permanent flow. In other words, it is the quasic-static approximation corresponding to low reduce frequencies \cite{FUNG}-\cite{PhD1}. Therefore the problem to be solved here is a simple first order differential equation in ${\bold a}$ and $r$ coupled with a static linear partial differential equation modeling the flow around the moving structure. Therefore we focus on the solution method for this last problem explicited at equation (\ref{eq3}) assuming that $\partial S$, ${\bold u}_0$ and ${\bold g}$ are given.
 \hfill $\Box$ \end{remark}

\section{The initial formulation for the flow} First of all let us introduce the problem (\ref{eq3}) as an optimization one. We introduce the functional space for the velocity field of the fluid:
\beq\label{eq5}V_0=\{{\bold v}=\{v_i\},\;i=1,2,\;v_i\in H^1(\Omega_S),\;\bold {div}({\bold v})=0\;\hbox{ in }\Omega_S,\;{\bold v}=0\;\hbox{on}\;\Gamma_0\}. \eeq
Then we set ($\sigma.\gamma=\sum_{i,j\in\{1,2\}}\sigma_{ij}\gamma_{ij}$):
\beq\label{eq6}\forall {\bold v}\in V_0,\;J({\bold v})=\fracj{1}{2}\intj_{\Omega_S}\sigma({\bold v}).\gamma({\bold v})-\intj_{\Gamma_1}{\bold g}.{\bold v}\eeq
It is classical that  equations (\ref{eq3}) are equivalent to:
\beq\label{eq7}
\min_{{\bold v}\in V_0\;{\bold v}={\bold a}+r {\bold k}\wedge {\bold Ox}\;\hbox{ on }\partial S} J({\bold v}).
\eeq

Existence and uniqueness of the solution -say ${\bold u}$- are also well known (see for instance \cite{PARVG}). 
\begin{remark}\label{rem2} The most popular method for solving (\ref{eq3}) is based on the mixed formulation where both the velocity and the pressure are unknowns. The results are very good and an extension with the acceleration term: 
$$\displaystyle \varrho  \left[ \fracj{\partial {\bold u}}{\partial t}+ {\bold u}\nabla {\bold u}\right],$$
 works perfectly. But, this formulation is restricted to incompressible (or slightly compressible) flows. The details are given for instance in \cite{PARVG}.
\hfill $\Box$
\end{remark}
\section{The extension to the full domain using penalty}\label{penalty} 
\subsection{The definitions of the penalty models}
The velocity ${\bold u}_0$ on the structure is defined by a rigid  body description. Such velocity fields satisfy the relation:
\beq\label{eq8}\gamma_{ij}({\bold u}_0)=0\;\;\forall i,j\in \{1,2\}.\eeq
The idea of the penalty method in this case, consists in replacing problem (\ref{eq7}) by  the next one ($\alpha\geq 0,\;\beta\geq 0,\;\gamma\geq 0\;\varepsilon >0$):
\beq\label{eq9}\hskip-.3cm\left\{\begin{array}{l}\;\;\;\;\;\;\;\;\;\;\;\;\;\;\;\;\;\;\;\;\;\;\;\;\;\;\;\;\;\;\;\;\;\;\;\;\;\;\;\minj_{{\bold v}\in W_0}J^\varepsilon({\bold v})\\\\\hbox{where:}\\\\
\hskip-.2cmJ^\varepsilon({\bold v})\hskip-.05cm=\displaystyle \hskip-.05cm\fracj{1}{2} \left[ \hskip-.05cm\intj_{\Omega}\hskip-.1cm\sigma({\bold v}).\gamma({\bold v})\hskip-.05cm+\hskip-.05cm\fracj{\alpha}{\varepsilon}\hskip-.05cm\intj_{\partial S}\hskip-.15cm\vert\vert {\bold v}\hskip-.05cm-\hskip-.05cm{\bold u}_0\vert\vert^2\hskip-.05cm+\hskip-.05cm\fracj{\beta}{\varepsilon}\hskip-.05cm\intj_{S}\hskip-.05cm\sigma({\bold v}){.\gamma({\bold v})}\hskip-.05cm+\hskip-.05cm\fracj{\gamma}{\varepsilon}\hskip-.05cm\intj_{S}\hskip-.05cm\vert\vert {\bold v}\hskip-.05cm-\hskip-.05cm{\bold u}_0\vert\vert^2 \right]
\\\\
\hbox{and:}\\\\
W_0=\left\lbrace {\bold v}=\{v_i\},\;i=1,2,\;v_i\in H^1(\Omega),\;\bold {div}({\bold v})=0\hbox{ on }\Omega,\;{\bold v}=0\hbox{ on } \Gamma_0 \right\rbrace .
\end{array}\right.
\eeq

Here again the existence and uniqueness of a solution are standard. The solution of (\ref{eq9}) is a function of $\alpha,\;\beta,\;\gamma$ and mainly of $\varepsilon$. Our purpose in this section, is to analyze the behaviour of this solution denoted by $\bold {u}^\varepsilon$ for several choices of $\alpha,\beta$ and $\gamma$, when $\varepsilon$  tends to zero. Let us  make explicit the variational formulation of (\ref{eq9}). The solution ${\bold u}^\varepsilon$ is characterized by:
\beq\label{eq3bis}\left \{
\begin{array}{l}
{\bold u}^\varepsilon \in W_0\;\hbox{such that }\forall {\bold v}\in W_0:\\\\
\intj_{\Omega}\sigma({\bold u}^\varepsilon).\gamma({\bold v})+\fracj{\alpha}{\varepsilon}\intj_{\partial S}({\bold u}^\varepsilon-{\bold u}_0).\bold{v}+\fracj{\beta}{\varepsilon}\intj_S\sigma({\bold u}^\varepsilon).\gamma({\bold v})\\ \\ \hspace{5cm}+\fracj{\gamma}{\varepsilon}\intj_{ S}({\bold u}^\varepsilon-{\bold u}_0).\bold{v}=\intj_{\Gamma_1} {\bold g}.{\bold v}
\end{array}\right.
\eeq
\begin{remark}\label{rem3}
In the papers by Benamour-Liberge-B\'eghein \cite{liberge1} \cite{liberge2} \cite{liberge3}, the authors set $\alpha=0$ and also $\beta=0$ but $\gamma\neq 0$. This is a very convenient strategy because one can use the mass matrix which is required for the dynamic analysis (the mass density for this new term is artificially equal to $\gamma/\varepsilon$ in $S$). The solution ${\bold u}^\varepsilon$ and also the normal stress associated should be continuous across the boundary $\partial S$ and this last condition (on the normal stresses)  is not well satisfied localy for $\varepsilon\simeq0$.  The reason is that the bilinear form involved in the limit model for $\alpha=\beta=0$, which couples ${\bold u}^0$ on $\Omega_S$  and ${\bold u}^1$ on $S$,  is only coercive on the space:
\beq\label{eq10}W_\gamma=\{{\bold v}\hbox{ s.t.}\; \hbox{div}({\bold v})=0\hbox{ in }\Omega\hbox{ and }{\bold v}_{\vert \Omega}\in {\Huge [}H^1(\Omega_S){\Huge ]}^2,\;{\bold v}_{\vert S}\in {\Huge [}L^2(S){\Huge]}^2\}.\eeq
It is a larger space than $W_0$ and such that only the continuity of the normal velocity ${\bold v}.\nu$ can be verified in the space $H^{-1/2}(\partial S)$ using the trace Theorem (see for instance \cite{JLLEM}) which is a weak result, mainly concerning the evaluation of the stress due to the fluid and applied to the structure.  Furthermore, the tangential component of ${\bold v}_{\vert S}$ on $\partial S$ doesn't exist in such a space. Nevertheless this strategy is very tricky for the numerical implementation and can lead to nice and cheap evaluation of the velocity field. Therefore we discuss it in the following in order to check which kind of information can be lost  concerning the tangential component of the normal stress along $\partial S$. 

If $\beta\neq 0$ whatever would be $\alpha$ and $\gamma$, one can can ensure the continuity of the two components of the velocity field ${\bold u}^\varepsilon$ solution of (\ref{eq9}) when $\varepsilon\rightarrow 0$. Hence only the rigid body motion of the open set $S$ is taken into account. Let us point out that if $\alpha=\gamma=0$ nothing guarantees that it will be the one prescribed by the equations of the mechanics. In fact ${\bold u}_0$ has disappeared from the penalty model and this statement is propped by the numerical tests of section \ref{test1}. 

 A discussion is carried out at subsection \ref{asymptotic} in order to explain the drawbacks of this nice strategy which is nevertheless used  for convenience in the implementation in existing softwares by many of authors.
\hfill $\Box$
\end{remark}
\begin{remark}\label{rem4}
The implementation of the penalty term on $S$ for $\beta>0$ is easy to do as far it only consists in choosing a new expression for the viscosity coefficient of the fluid. We set $\mu$ in $\Omega_S$ and $\mu(1+\frac{\beta}{\varepsilon})$ in $S$.
\hfill $\Box$
\end{remark}
\begin{remark}\label{eq11}
In a numerical implementations, one needs to use  two mappings because the movement of the structure does respect the sides of the mesh used in the approximation of the fluid movement. One -say $R^h$-  maps the position of $\partial S$ on the mesh and the other -say $P^h$- maps the mesh of $\Omega$ on $\partial S$.  But  this embedding problem is fully and accurately solved in a practical formulation by C. Farhat and his coworkers \cite{Farhat1}-\cite{Farhat2}. A mathematical formulation of this arbitrary Euler-Lagrange parametrization is briefly discussed in  \cite{PHDEL2} using domain derivative tools for improving the normal stresses approximation between the fluid and the structure.
\hfill $\Box$
\end{remark}
\subsection{The asymptotic analysis}\label{asymptotic}
Let set {\it a priori}:
\beq\label{eq13}
{\bold u}^\varepsilon={\bold u}^0+\varepsilon {\bold u}^1+\ldots
\eeq

By introducing this expression into the variational formulation (\ref{eq3bis}) and by equating the terms of same power in $\varepsilon$ one obtains the following necessary conditions:
\beq\label{eq14}\left \{
\begin{array}{l}
\hbox{ i) terms of order }\varepsilon^{-1}:\;\forall {\bold v}\in W_0,\;(\hbox{ let us notice that $\sigma({\bold u}_0)=0$})\\\\
\alpha\intj_{\partial S}({\bold u}^0-{\bold u}_0).{\bold v}+\beta\intj_{S}\sigma({\bold u^0}).\gamma({\bold v})+\gamma\intj_{S}({\bold u}^0-{\bold u}_0).{\bold v}=0,
\\\\
\hbox{ ii) terms of order zero: } \forall {\bold v}\in W_0,\\\\
\intj_{\Omega} \sigma({\bold u^0}).\gamma({\bold v})+\alpha\intj_{\partial S} {\bold u}^1.{\bold v}+\beta\intj_{S}\sigma({\bold u^1}).\gamma({\bold v})+\gamma\intj_S{\bold u}^1.{\bold v}={\intj_{\Gamma_1} {\bold g}.{\bold v}},
\\\\
\\\\
\hbox{ iii) terms of order }\varepsilon:\;\forall {\bold v}\in W_0,\;\\\\
\intj_{\Omega} \sigma({\bold u^1}).\gamma({\bold v})+\alpha\intj_{\partial S} {\bold u}^2.{\bold v}+\beta\intj_{S}\sigma({\bold u^2}).\gamma({\bold v})+\gamma\intj_S{\bold u}^2.{\bold v}=0,\\\\
\ldots
\end{array}\right.
\eeq

Setting on $S$:
 $${\bold u}^0={\bold u}_0+{\bold u}^{00},\;{\bold u}^{00}\in W_0,$$
  with the condition ${\bold u}^{00}=0$ on $\partial S$, equation i) of (\ref{eq14}) implies that:
\beq\label{eq15}
\left \{\begin{array}{l} \hbox{ if }\alpha,\beta>0 \hbox{ but }\gamma\geq 0: \\\\\hskip 1cm =>{\bold u}^0={\bold u}_0\hbox { on }\partial S \hbox{ and }\gamma_{ij}({\bold u}^{00})=0\hbox{ in }S=>{\bold u}^{00}=0 \hbox{ on }S,\\\\
\hbox{ if }\beta=0 \hbox{ but }\gamma>0 \hbox{ one has for any $\alpha\geq0$}:\\\\\hskip 1cm=> {\bold u}^{00}=0 .\end{array}\right.
\eeq

Therefore ${\bold u}^0$ is a rigid body velocity on $S$ and equal to ${\bold u}_0$ as far as $\alpha$ or $\gamma$ is (are) strictly positive. But  it should be noticed that this last statement is valid only because $\alpha>0$ or $\gamma>0$. For $\alpha=\gamma=0$ one could only claim that ${\bold u}^0$ is a rigid body velocity on $S$ but not necessarily  equal to ${\bold u}_0$. At this step, the case $\alpha=\beta=0$ seems to be the most convenient, but the following  contradicts this first conclusion.

We now introduce a new functional space by:
\beq\label{eq16}
K_0=\{{\bold v}\in W_0,\;{\bold v}=0\hbox{ in }S\}.
\eeq
From equation ii) of (\ref{eq14}), one deduces that:
\beq\label{eq17}
\forall {\bold v}\in K_0,\;\intj_{\Omega_S}\sigma.({\bold u}^0).\gamma({\bold v})=\intj_{\Gamma_1}{\bold g}.{\bold v},\;\hbox{ and if $\alpha^2+\gamma^2>0$ }=>{\bold u}^0={\bold u}_0\hbox{ on }\partial S.
\eeq
The element ${\bold u}^0\in K_0$ is now perfectly defined and is identical on $\Omega_S$ to the initial solution ${\bold u}$ characterized at (\ref{eq3}).
\begin{remark}\label{rem5}
If $\alpha=\gamma=0$ and $\beta>0$ one has only:
\beq\label{eq17bis}\gamma_{ij} ({\bold u}^0)=0 \hbox{ in }S=>{\bold u}^0=a+r{\bold k}\wedge \bold {Ox}\hbox {  in }S.
\eeq
The term ${\bold u}^0$ belongs to the following closed subspace of $W_0$ defined by:
\beq\label{eq20}RBM=\{{\bold v}\in W_0,\;\gamma({\bold v})=0\hbox{ in }S\}.\eeq
Hence in this case ($\alpha=\gamma=0$) the term ${\bold u}^0\in RBM$ is the unique solution of:
\beq \label{eq17ters}\forall {\bold v}\in RBM,\;\intj_{\Omega}\sigma({\bold u}^0).\gamma({\bold v})=\intj_{\Gamma_1} {\bold g}.{\bold v}\;\;\;\;.
\eeq

But the rigid body motion on $S$ is derived from the solution ${\bold u}^0$ and is no more the one that has been prescribed from the movement of the structure. This is the fundamental difference that suggests to avoid the choice $\alpha=0$ or $\gamma=0$. In fact the choice $\beta=0$ is another possibility which will appear as a smart possibility in the discussion contained in the next remark \ref{rem9}, even if it has its own drawbacks concerning the regularity if the boundary $\partial S$ is not smooth enough as we mention in the following.
\hfill $\Box$
\end{remark}

\begin{remark}\label{rem9}
Let us now consider the case where $\alpha>0$ and $\beta=\gamma=0$. This is very interesting situation from the practical point of view, because we only have to manage the position of the boundary $\partial S$ inside the mesh which can be unchanged during the displacement of the structure (full eulerian representation of the fluid). Let us also underline that this method is just for the quasi-static case. If inertia terms are introduced in the fluid then it  doesn't work in this case. But even for quasi-static cases (no inertia terms),  this assertion could be false for a multiprocessor programming where each processor could be idling when waiting for the informations concerning this boundary $\partial S$. At the order  minus one in $\varepsilon$ in the asymptotic expansion, one only gets the condition ${\bold u}^0={\bold u}_0$ on $\partial S$. But at the order zero in $\varepsilon$ one has:

\beq\label{eq100}
\forall {\bold v}\in W_0,\;\intj_\Omega \sigma({\bold u}^0).\gamma({\bold v})+\alpha\intj_{\partial S}{\bold u}^1.{\bold v}=\intj_{\Gamma_1}{\bold g}.{\bold v}.
\eeq

Let us define three closed subspaces of $W_0$ by:
\beq\label{eq101}\left \{ \begin{array}{l}
W_{00}=\{ v\in W_0,\;{\bold v}=0\hbox{ on }\partial S\},\\\\
V_{00\Omega_S}=\{{\bold v}\in V_0\hbox{ and }{\bold v}=0\hbox{ on }\partial S,\;({\bold v}=0\hbox{ on } S)\};\\\\
V_{00S}= {\Huge[}H^1_0(S){\Huge ]}^2.
\end{array}\right.
\eeq
It is worth to notice that one can identify $W_{00}$ with the space $V_{00\Omega}\times V_{00S}$.

By restricting the virtual velocity ${\bold v}$ to the space $W_{00}$ (${\bold u}^0= {\bold u}_0$ (the rigid body velocity) on $\partial S$) one obtains two distinct sub-problems; one is set on $\Omega_S$ and the other on $S$. They are given  hereafter:
\beq\label{eq102}\left \{\begin{array}{l}
{\bold u}^0_{\vert \Omega_S} \in V_0,\forall {\bold v} \in V_{00\Omega_S}:\;
\intj_{\Omega_S}\sigma({\bold u}^0).\gamma({\bold v})=\intj_{\Gamma_1}{\bold g}.{\bold v},\;\;{\bold u}^0={\bold u}_0\hbox{ on }\partial S,\\\\
{\bold u}^0_{\vert S}\in V_{00S},\;
\forall {\bold v}\in V_{00S},\;\intj_S\sigma({\bold u}^0).\gamma({\bold v})=0,\;{\bold u}^0={\bold u}_0\hbox{ on }\partial S.
\end{array}\right.
\eeq

The system of equations (\ref{eq102}) leads to two independent problems for which the solutions are uniquely defined respectively on $\Omega_S$ and on $S$. The solution ${\bold u}^0$ on $\Omega_S$ is precisely the one of the initial model we started from (\ref{eq3}). The solution ${\bold u}^0$ on $S$ can easily be characterized. Because the solution is unique and because the rigid body ${\bold u}_0$ is a solution, it is the right one. Hence ${\bold u}^0$ is exactly the solution in the full domain $\Omega$ of the initial problem. 

Nevertheless, it should be underlined that in this case ($\beta=\gamma=0$ and $\alpha>0$) the penalty term is reduced to a boundary term implying the $L^2(\partial S)$ norm and therefore the continuous model is perfectly defined. The numerical aspects are easy to handle if we use a separate modeling of $\partial S$ (we do not need to divide the elements at the boundary between $S$ and $\Omega_S$). But  the condition number of the discretized model remains an important question in the solution method compared to the case where $\alpha>0$ and $\beta>0$.
\hfill $\Box$
\end{remark}
\begin{remark}\label{rem9} In this remark we discuss the case where $\alpha=\beta=0$ and $\gamma>0$. The first equation (\ref{eq14}) leads directly to ${\bold u}^0={\bold u}_{0}$ in $S$. Therefore, from the second equation (\ref{eq14}) ${\bold u}^0$ should connected to an element ${\bold u}^1$ such that:
\beq\label{eq131}\forall {\bold v}\in W_0,\;\intj_\Omega \sigma({\bold u}^0).\gamma({\bold v})+\gamma\intj_S{\bold u}^1.{\bold v}=\intj_{\Gamma_1}{\bold g}.{\bold v}.\eeq

Because  we proved that ${\bold u}^0={\bold u}_0$ in $S$, the characterization of ${\bold u}^0$ on $\Omega_S$ is obtained by solving the equation (which has a unique solution):
\beq\label{eq132}{\bold u}^0_{\vert \Omega_S} \in V_0,\forall {\bold v} \in V_{00\Omega_S},\;
\intj_{\Omega_S}\sigma({\bold u}^0).\gamma({\bold v})=\intj_{\Gamma_1}{\bold g}.{\bold v},\;\;{\bold u}^0={\bold u}_0\hbox{ on }\partial S. \eeq

Consequently,  the term ${\bold u}^1$ should satisfy in $S$, the relation:
\beq\label{eq133}{\bold u}^1=0,\eeq
which is not compatible with equation (\ref{eq131}) which leads to ($\nu $ is here the unit normal along $\partial S$ inward $S$):
\beq \label{eq134}\forall {\bold v}\in W_0,\;\intj_{\partial S}(\sigma({\bold u}^0).\nu).{\bold v}+\gamma\intj_S{\bold u}^1.{\bold v}=0.
\eeq

In fact it is impossible to compute ${\bold u}^1$ which would appear as a first order derivative at the origin of ${\bold u}^\varepsilon$ with respect to $\varepsilon$. This phenomenon is due to a stiff boundary layer near $\partial S$ which is fully analyzed in the linear case in the book of J.L. Lions on singular perturbations \cite{Lionssingular}. Using interpolation technics between Hilbert spaces, he proved that one could only hope a fractional derivative order  and furthermore in a larger space than the energy one. This result is extended to Navier-Stokes equations in the paper by P. Angot \& all \cite{penalty2}. Nevertheless the convergence of ${\bold u}^\varepsilon$ to ${\bold u}^0$ is true in energy but for the weak topology  for $\alpha=\beta=0$ and $\gamma>0$ as we proved at Theorem \ref{th0} at the end of this remark. But the speed of convergence in $\varepsilon$ is not obvious and furthermore in a larger space than the energy one. This justifies to use this penalty method in this framework when the programming advantages are clear (multiprocessor programming for instance). A numerical discussion concerning  these theoretical results on a very simple model is suggested in section \ref{test1}.
\hfill $\Box$
\end{remark}
\begin{theorem}\label{th0} Let ${\bold u}^\varepsilon\in W_0$ be solution of (\ref{eq3bis}) and ${\bold u}^0$ be the solution of (\ref{eq132}) in $\Omega_S$ such that ${\bold u}^0={\bold u}_0$ in $S$. One has:
\beq\label{eq200}\limj_{\varepsilon\rightarrow 0}\vert\vert {\bold u}^\varepsilon-{\bold u}^0\vert\vert_{0,\Omega}=0\hbox{  and  }\vert\vert {\bold u}^\varepsilon-{\bold u}_0\vert\vert_{0,S}\leq c\sqrt{\varepsilon}.\eeq
\hfill $\Box$\end{theorem}
\begin{proof} Let set ${\bold v}={\bold u}^\varepsilon$ in equation (\ref{eq3bis}). Using Korn inequality \cite{duvautlions}, one can claim that there exists a constant independent on $\varepsilon$ -say $c_0>0$- such that:
\beq\label{eq500}c_0\vert\vert{\bold u}^\varepsilon\vert\vert^2_{1,\Omega}+\fracj{\gamma}{\varepsilon}\vert\vert {\bold u}^\varepsilon-{\bold u}_0\vert\vert^2_{0,S}\leq \fracj{\gamma}{\varepsilon}\intj_S{\bold u}_0{\bold u}^\varepsilon+\intj_{\Gamma_1}{\bold g}.{\bold u}^\varepsilon.\eeq
Using the triangular Cauchy-Schwarz\footnote{$\forall a,b,\zeta \geq 0,\;2 ab\leq \zeta a^2+\frac{1}{\zeta}b^2$} inequality and the continuity of the trace operator from $H^1(\Omega)$ into $L^2(\Gamma_1)$, on deduces the existence of another constant independent on $\varepsilon$ -say $c_1$- such that:
\beq\label{eq600}c_0\vert\vert {\bold u}^\varepsilon\vert\vert_{1,\Omega}^2+\fracj{\gamma}{\varepsilon}\vert\vert {\bold u}^\varepsilon-{\bold u}_0\vert\vert_{0,S}^2 \leq  c_1. \eeq
Therefore one can extract from the sequence ${\bold u}^\varepsilon$ a subsequence denoted by ${\bold u}^{\varepsilon'}$ and such that:
\beq\label{eq501}\left \{ \begin{array}{l}
i)\hbox{ in } S:\;\limj_{\varepsilon
\rightarrow 0} \vert\vert {\bold u}^{\varepsilon}-{\bold u}_0\vert\vert_{0,S}=0 \hbox{ (strong convergence with the rate }\sqrt{\varepsilon}),\\\\i)\;\hbox{ in }\Omega:\;\limj_{\varepsilon'
\rightarrow 0}{\bold u}^{\varepsilon'}={\bold u}^* \hbox{  in ${\Huge [}H^1(\Omega){\Huge ]}^2$-weakly; hence ${\bold u}^*={\bold u}_0$ in $S$}.\end{array}\right.\eeq
Taking the limit in (\ref{eq3bis}), one characterizes ${\bold u}^*$ as the unique solution of:
\beq\label{eq502}{\bold u}^*={\bold u}_0\hbox{ on }\partial S\hbox{ and }\forall {\bold v}\in K_0,\;\intj_{\Omega_S}\sigma({\bold u}^*).\gamma({\bold v})=\intj_{\Gamma_1}{\bold g}.{\bold v}.\eeq
In other words ${\bold u}^*={\bold u}^0$ in all $\Omega$. Because of the uniqueness of ${\bold u}^0$, all the sequence ${\bold u}^\varepsilon$ converges to ${\bold u}^0$ in $H^1(\Omega)$-weakly and strongly in $L^2(S)$. But the embeding from $H^1(\Omega)$ into $L^2(\Omega)$ is compact, hence the convergence of the sequence ${\bold u}^\varepsilon$ to ${\bold u}^0$ is strong in ${\Huge [}L^2(\Omega){\Huge ]}^2$.
\end{proof}

In the following of this section we assume that $\alpha>0$ and we start with $\beta>0$ in order to reduce singular perturbations problems in the vicinity of $\partial S$ which would damage the computation of the forces applied to the structure by the fluid. Let us go on with the relation ii) of equation (\ref{eq14}) but now for arbitrary element ${\bold v}$ in $W_0$. Because ${\bold u}^0$ is perfectly determined, it remains to find an element (not necessarily unique) ${\bold u}^1$ in $W_0$ such that (let us emphasize on the fact that the right-hand side is a linear and continuous form on $W_0$ without any additional regularity required on ${\bold u}^0$):
\beq\label{eq18}\forall {\bold v}\in W_0,\;\alpha\intj_{\partial S}{\bold u}^1.{\bold v}+\beta\intj_{S}\sigma({\bold u}^1).\gamma({\bold v})=-\intj_{\Omega_S}\sigma({\bold u}^0).\gamma({\bold v})+\intj_{\Gamma_1}{\bold g}.{\bold v}\;\;.\eeq
Or else from an integration by parts (the velocity fields are divergence free):
\beq\label{eq19}\forall {\bold v}\in W_0,\;\alpha\intj_{\partial S}{\bold u}^1.{\bold v}+\beta\intj_{S}\sigma({\bold u}^1).\gamma({\bold v})=-\intj_{\partial S}(\sigma({\bold u}^0).\nu).{\bold v}\;\;.\eeq
Let us check that ${\bold u}^1$ is perfectly determined by the previous relation on $S$ (but not yet on $\Omega_S$).  We denote by ${\bold u}^{1p}$ this particular term defined on $S$. The bilinear form $a^1$ defined by:
\beq\label{eq20}
{\bold u},{\bold v}\in {\Huge [}H^1(S){\Huge ]}^2\rightarrow a^1({\bold u},{\bold v})=\alpha\intj_{\partial S}{\bold u}.{\bold v}+\beta\intj_{S}\sigma({\bold u}).\gamma({\bold v}),
\eeq
is ${\Huge [}H^1(S){\Huge ]}^2$-coercive (and symmetrical) and therefore, from Lax-Milgram Theorem,  one can claim that ${\bold u}^{1}$ is perfectly defined on $S$ as a solution of (\ref{eq19}) (using the expression given at (\ref{eq18}) the right-hand side is clearly a linear and continuous form on the functional space ${\Huge [}H^1(S){\Huge ]}^2$). It is denoted by ${\bold u}^{1p}$ and it also satisfies the incompressibility condition on $S$. 

In order to complete the definition of ${\bold u}^1$ on the whole open set $\Omega$, we go to the equations obtained at the order one. Choosing an arbitrary element ${\bold v}\in K_0$ (see (\ref{eq16})), one obtains (let us point out that the normal stress due to ${\bold u}^1$ is not continuous across $\partial S$):
\beq\label{eq21}\left\{\begin{array}{l}\intj_{\Omega_S}\sigma({\bold u}^1).\gamma({\bold v})=0,\\\\
{\bold u}^1={\bold u}^{1p}\hbox { on }\partial S \hbox{ and }{\bold u}^1_{\vert \Omega_S}\in V_0.
\end{array}\right.
\eeq
Finally, the term ${\bold u}^1$ is defined on the whole open set $\Omega$ and because of the continuity of the trace on $\partial S$, it belongs to the space ${\Huge [}H^1(\Omega){\Huge]}^2$. Its restriction to $S$ is ${\bold u}^{1p}$ characterized at (\ref{eq19}) and its restriction to $\Omega_S$ is solution of (\ref{eq21}). Globally ${\bold u}^1$ belongs to the space $W_0$ because it is divergence free on each open subset  of $\Omega$. It should be underlined that no additional regularity is required in the computation of ${\bold u}^1$. But this will be different when  $\beta=0$.

Let us finish this partial analysis of an assumed asymptotic expansion of ${\bold u}^\varepsilon$, by  a partial characterization of ${\bold u}^2$ on $S$ denoted by ${\bold u}^{2p}$. Once ${\bold u}^1$ is known one should have (order 1 in $\varepsilon$ and $[\vert.\vert ]$ is the jump function across $\partial S$ in the direction $\nu$):
\beq\label{eq21bis}\left \{\begin{array}{l}\forall 
 {\bold v}\in W_0,\\\\
\alpha\intj_{\partial S}{\bold u}^2.{\bold v}+\beta\intj_{S}\sigma({\bold u}^2).\gamma({\bold v})=-\intj_{\Omega}\sigma({\bold u}^1).\gamma({\bold v})=\intj_{\partial S}{[\vert}(\sigma({\bold u}^1){\vert ]}.\nu.{\bold v}.\end{array}\right.
\eeq
The restriction of ${\bold u}^2$ to $\Omega_S$ is denoted by ${\bold u}^{2p}$. The existence and uniqueness of this term ${\bold u}^{2p}\in {\Huge [}H^1(S){\Huge ]}^2,\,\hbox{div(${\bold u}^{2p}$)$=0$}$ on $S$, are obtained exactly  as we did for ${\bold u}^{1p}$. Then the computation of ${\bold u}^2$ on $\Omega_S$ is also handled as we did for ${\bold u}^1$. Therefore ${\bold u}^2$ is now defined on the whole domain $\Omega$ and is an element of the space $W_0$. Here again, no additional regularity is required for the computation of ${\bold u}^2$ but we shall have a different conclusion for $\beta=0$.
\begin{remark}\label{rem10}
If $\beta=\gamma=0$ and $\alpha>0$ the term ${\bold u}^1$ should be solution of (terms of order zero in the assumed asymptotic expansion in $\varepsilon$):
\beq\label{eq104}
\forall {\bold v}\in W_0,\;\intj_{\Omega} \sigma({\bold u}^0). \gamma({\bold v})+\alpha\intj_{\partial S}{\bold u}^1.{\bold v}=\intj_{\Gamma_1} {\bold g}.{\bold v}.
\eeq
This implies that necessarily one should have (here $\sigma({\bold u}^0).\nu$ is the value from $\Omega_S$ and it is zero for the contribution from $S$ because ${\bold u}^0$ is a rigid body velocity on $S$):
\beq\label{eq105}{\bold u}^1=-\sigma({\bold u}^0).\nu\;\hbox{ on }\partial S\;(\nu \hbox{ inward }S).\eeq

It is necessary to prove that the normal component of the stress field $\sigma({\bold u}^0).\nu$ is more regular than the classical result (see for instance J.L. Lions-E. Magenes \cite{JLLEM}) which states that it belongs to the space ${\Huge [}H^{-1/2}(\partial S){\Huge ]}^2$. In fact a hidden regularity result that one can derive from the J. Hadamard domain derivation method which has been widely used in mathematical control analysis \cite{lionscont}, enables one to prove that $\sigma({\bold u}^0).\nu\in {\Huge [}L^2(\partial S){\Huge ]}^2$. But this is insufficient because ${\bold u}^1$ should belong to the space ${\Huge [}H^{1/2}(\partial S){\Huge ]}^2$ in order to have ${\bold u}^1$ in the space ${\Huge [}H^1(S){\Huge ]}^2$. Therefore one can point out a weakness of this case which is the lost of some regularity which will have some nasty consequences on the convergence of the penalty method when $\varepsilon\rightarrow 0$. This is the case for instance if the boundary of the structure has sharp corners  (see P. Grisvard for details \cite{grisvard}). But the penalty-duality method will overcome this difficulty and could be seen as an improvement of this penalty method when $\alpha>0$ and $\beta=\gamma=0$  as we show in the following.

Let us now turn to the characterization of ${\bold u}^1$ inside $\Omega_S$ and $S$ (always with $\alpha>0$ and $\beta=\gamma=0$). Hence from the equation at the order one in $\varepsilon$, one obtains:
\beq\label{eq16}
\forall {\bold v}\in W_0,\;\intj_{\Omega}\sigma({\bold u}^1).\gamma({\bold v})+\alpha\intj_{\partial S}{\bold u}^2.{\bold v}=0\;\hbox{ and }{\bold u}^{1p}=-\sigma({\bold u}^0).\nu\;\hbox{ on }\partial S.
\eeq
Here again, as in Remark \ref{rem9}, the system (\ref{eq104}) can be split into two independent problems. which characterizes ${\bold u}^1$ separately on $\Omega_S$ and on $S$.
\beq\label{eq107}\left \{ \begin{array}{l}\hbox{ on } \Omega_S:
\\\\
\forall {\bold v}\in V_0,\;\intj_{\Omega_S}\sigma({\bold u}^1).\gamma({\bold v})=0,\;\;{\bold u}^1={\bold u}^{1p}\hbox{ on }\partial S,
\\\\
\hbox{ on } S:
\\\\
\forall {\bold v}\in {\Huge [}H^1_0(S){\bold ]}^2,\;\intj_{S}\sigma({\bold u}^1).\gamma({\bold v})=0\;\hbox{ and }{\bold u}^1={\bold u}^{1p}\hbox{ on }\partial S.
\end{array}\right.
\eeq
These two non-homogeneous boundary values problems (Dirichlet condition on $\partial S$) have each one a unique solution because of the classical properties of the linear elasticity operator. But in this case there are two important features which should be emphasized.
\begin{enumerate}
\item  The regularity of  the boundary of the immersed structure ($\partial S$) is required in order to ensure that $\sigma({\bold u}^0).\nu\in {\Huge [}H^{1/2}(\partial S){\Huge ]}^2.$ If it isn't one could make use singular functions due to the corner on $\partial S$ but this is beyond our scope in this paper.
\item The solution ${\bold u}^1$ on $S$ is no more a rigid body velocity (in general).
%g
\end{enumerate}
Nevertheless this method ($\alpha>0$ and $\beta=\gamma=0$) will be the most adapted to the penalty-duality method that we discuss in section \ref{penaltyduality} as far as one only considers quasi-static modeling.\hfill $\Box$
\end{remark}
%===========================
\subsection{Convergence results (case $\alpha>0\:\beta>0$ and $\gamma=0$).} We introduce the error term:
\beq\label{eq22}
\underline{\bold u}^\varepsilon={\bold u}^\varepsilon-{\bold u}^0-\varepsilon {\bold u}^1-\varepsilon^2{\bold u}^{2}\in W_0.
\eeq

From the definition of the different terms and from the triangular inequality\footnote{$\forall a,b\in \a R,\;\xi\in \a R^+\hbox{ one has }ab\leq \frac{2\xi}a^2+\frac{1}{2\xi}b^2$} and $c$ denoting the continuity constant of the bilinear form of the elasticity operator:
\beq\label{eq23}\hskip-.2cm\left \{\begin{array}{l}\forall \xi>0:\;\;
\intj_{\Omega}\hskip-.05cm\sigma(\underline{\bold u}^\varepsilon).\gamma(\underline{\bold u}^\varepsilon)+\fracj{\alpha}{\varepsilon}\intj_{\partial S}\hskip-.05cm\vert\vert \underline{\bold u}^\varepsilon\vert\vert^2+\fracj{\beta}{\varepsilon}\intj_{S}\hskip-.05cm\sigma(\underline{\bold u}^\varepsilon).\gamma(\underline{\bold u}^\varepsilon)=\\\\-\varepsilon^2\intj_{\Omega}\hskip-.05cm\sigma({\bold u}^2).\gamma(\underline{\bold u}^\varepsilon)
\leq c \displaystyle\left\lbrace  \fracj{\xi}{2}\varepsilon^4 \vert\vert {\bold u}^2\vert\vert_{1,\Omega}^2+\fracj{1}{2\xi}\vert\vert\underline {\bold u}^\varepsilon\vert\vert_{1,\Omega}^2 \right\rbrace . \end{array}\right. 
\eeq

Finally, from  classical coerciveness properties deduced from Korn inequality \cite{duvautlions}, one has the following error estimates where $c_0$ and $c_1$ are two positive constants independent on $\varepsilon$ (but dependent on the data of the problem):
\beq\label{eq24}
c_0\vert\vert \underline{\bold u}^\varepsilon\vert\vert_{1,\Omega}^2+\fracj{\alpha}{2 \varepsilon}\vert\vert \underline{\bold u}^\varepsilon\vert\vert_{0,\partial S}^2+\fracj{\beta}{2 \varepsilon}\vert\vert \underline{\bold u}^\varepsilon\vert\vert_{1,S}^2\leq c_1\varepsilon^4.
\eeq

Let us summarize the previous results in the following statement.
\begin{theorem}\label{th1}
Let ${\bold u}^\varepsilon$ be the solution of (\ref{eq3bis}) and ${\bold u}^0,{\bold u}^1$ the terms characterized in the asymptotic expansion at (\ref{eq19}) and (\ref{eq21}). We also assume in this Theorem that $\alpha>0$ and $\beta>0$. Then there exists a constant -say $c_2>0$- independent on $\varepsilon$ and such that:
\beq\label{eq25}\left \{\begin{array}{l}
\vert\vert {\bold u}^\varepsilon-{\bold u}^0-\varepsilon{\bold u}^1\vert\vert_{1,\Omega_S}\leq c_2\varepsilon^{2},\\\\
\vert\vert  {\bold u}^\varepsilon-{\bold u}^0-\varepsilon{\bold u}^1\vert\vert_{0,\partial S}\leq c_2 \varepsilon^{5/2},\;\;\vert\vert  {\bold u}^\varepsilon-{\bold u}^0-\varepsilon{\bold u}^1\vert\vert_{1,S}\leq c_2 \varepsilon ^{5/2}.
\end{array}\right.
\eeq

Furthermore the term ${\bold u}^0$ is exactly the solution of the model (\ref{eq3}) we started from where the displacement of the structure is the prescribed to the boundary $\partial S$ of the fluid.
\end{theorem}
\begin{corollary}\label{corollaire1}
From Theorem \ref{th1} and making use of the triangular inequality, one obtains the following error bounds between ${\bold u}^\varepsilon$ and ${\bold u}^0$ where $c_3$ is a constant independent on $\varepsilon$:
\beq\label{eq17}
\left \{\begin{array}{l}
\vert\vert {\bold u}^\varepsilon-{\bold u}^0\vert\vert_{1,\Omega_S}\leq c_3\varepsilon,\\\\
\vert\vert  {\bold u}^\varepsilon-{\bold u}^0\vert\vert_{0,\partial S}\leq c_3 \varepsilon,\;\;\vert\vert  {\bold u}^\varepsilon-{\bold u}^0\vert\vert_{1,S}\leq c_3 \varepsilon.
\end{array}\right.
\eeq
\end{corollary}
\begin{remark}\label{rem7} In the eventuality where $\alpha=0$ (no penalty term on the boundary of $S$) the similar results as those of Theorem \ref{th1} are still true but ${\bold u}^0$ is the not the solution of the initial problem. Even if it is a rigid body velocity on $S$, it is not the one of the structure deduced from a correct writing of the principle of the mechanics. For instance the mass of the structure is not appearing in the model. Furthermore we restrict the analysis to quasi-static fluid-structure interaction (small reduce frequencies). Hence this possibility should be forgotten.
\hfill $\Box$
\end{remark}
\subsection{Convergence result for $\beta=\gamma=0$ and $\alpha>0$.} The results obtained in Theorem \ref{th1} are still true up to few modifications implying an additional regularity assumption on the term ${\bold u}^0$. We just give the statement without details because of the narrow similarity with the one of Theorem \ref{th1}. Similarly to the error bound derived for $\alpha>0$ and $\beta>0$, and as far as we just upper bound the error between ${\bold u}^\varepsilon$ and ${\bold u}^0$, it is not necessary to imply the term ${\bold u}^2$.   At the opposite of the case $\beta>0$, it would require also an additional regularity on $\sigma({\bold u}^1).\nu$ and therefore on ${\bold u}^0$ too.
\begin{theorem}\label{th2} The notations are those introduced in the previous subsection but ${\bold u}^0$ is now solution of (\ref{eq102}). Furthermore, it is assumed that the solution ${\bold u}^0$ of the equations (\ref{eq102}) satisfies $\sigma({\bold u}^0).\nu \in {\Huge [}H^{1/2}(\partial S){\Huge ]}^2$.   Then there exists a constant $c_4>0$ independent on $\varepsilon$ such that:
\beq\label{eq110}\left \{ \begin{array}{l}
\vert \vert{\bold u}^\varepsilon-{\bold u}^0\vert\vert_{1,\Omega}\leq c_4\varepsilon,\\\\
\vert\vert {\bold u}^\varepsilon-{\bold u}_0\vert\vert_{0,\partial S}\leq c_4\varepsilon.
\end{array}\right.
\eeq
\end{theorem}
\begin{remark}\label{rem9bis}
One should notice that the convergence on $S$ to the rigid body velocity is in $\varepsilon$ as in the case where $\beta>0$. The correctors ${\bold u}^1$ have in fact no real interest because the term that we need to estimate is ${\bold u}^0$. Let us also emphasize on the interest of this strategy which will clearly appear in the penalty-duality method.
\hfill $\Box$
\end{remark}
%
%\begin{remark}\label{rem80}
 %Using interpolation between Hilbert spaces, it is possible to complete the results of Theorem \ref{th3}. Hence if one has only $\sigma({\bold u}^0).\nu\{\Huge [}L^2(\partial S){\Huge ]}^2$ one can prove that the error bound in $L^2(\Omega)$-norm between ${\bold u}^\varepsilon$ and ${\bold u}^0$ is in $\sqrt{\varepsilon}$. But this is not a very useful result. Details of this interpolation technics can be found in J.L. Lions and E. Magenes \cite{lionsmagenes}.\hfill $ Box$
%\end{remark}
%
 \section{A simple numerical discussion between the various penalty methods}\label{test1} In order to compare the different penalty terms we discuss a very simple example which shows the various behavior of the three penalty terms considered in this paper. 
 \subsection{The test case} The model is the following one:
 \beq\label{eq600}\hskip-.5cm\left \{\begin{array}{l}V= \{v\in H^1(]0,L[),\;v(0)=0\},\;\hbox{ find }u\in V\;\hbox{ that: }
\displaystyle \forall v\in V,\\\\\intj_0^L\hskip-.05cm\fracj{du}{dx}\fracj{dv}{dx}\hskip-.05cm+\hskip-.05cm\fracj{\alpha}{\varepsilon}\left(u\left(\frac{\tiny L}{\tiny 2}\right)\hskip-.05cm-\hskip-.05cmu_0\right)v\left(\frac{L}{2}\right)\hskip-.05cm+\hskip-.05cm\intj_0^L\hskip-.05cm \left\lbrace \fracj{\beta}{\varepsilon}\fracj{du}{dx}\fracj{dv}{dx}\hskip-.05cm+\hskip-.05cm\fracj{\gamma}{\varepsilon}\left(u-u_0\right)v \right\rbrace \hskip-.05cm=\hskip-.05cm0.\end{array}\right.\eeq
 \\
 
 The term $u_0$ is a given constant and $\varepsilon$ is the penalty parameter in order to prescribe approximately $u=u_0$ on $]L/2,L[$. The three parameters $\alpha,\beta$ and $\gamma$ are chosen equal to $0$ or $1$ in order to compare their efficiency in the penalty strategy. The solution method used is a first degree finite element method (1000 points). We have plotted the results obtained for two choices of the penalty parameter $\varepsilon$ on Figures \ref{fig1}-\ref{fig2}. The space derivative of the solution is plotted on Figure \ref{fig3} for $\varepsilon=10^{-2}$

 \subsection{Comments on the results} First of all it is worth noting that the method with $\alpha=\gamma=0$ and $\beta>0$ is not adapted to prescribe a given movement of the structure. As it is explained in remark \ref{rem7}, the solution is constant on $]L/2,L[$ but it is not the one that one hopes. Nevertheless, if we consider the two terms $\alpha>0$ and $\beta\geq 0$ or $\beta>0$ and $\gamma>0$ the results are are much better. But in the second case the converge is not as good as than in the first one. If one restricts the penalty term to $\alpha=\beta=0$ and $\gamma>0$ the convergence occurs but is slower than in the case where $\alpha>0$, even if this is the only term which is kept.
 
 Concerning the derivatives of the solution with respect to the space coordinate one can see on Figure \ref{fig3} that the convergence is not satisfying if $\alpha=0$. This the reason why we develop in the next section \ref{penaltyduality} a penalty-duality method which implies only this boundary term. Furthermore, the dual variable (Lagrange multiplier) is easier to handle when only this boundary term is taken into account. In case of internal conditions the duality is a little bit more complex. We refer to the so-called Arlequin method developed by H. Ben Dhia \cite{hachmi1}-\cite{hachmi2} and in slightly different context by T. Belytschko \& all  \cite{ted}.
 %=======
 \begin{figure}[htbp] %  figure placement: here, top, bottom, or page
    \centering
    \includegraphics[width=13cm,height=18cm]{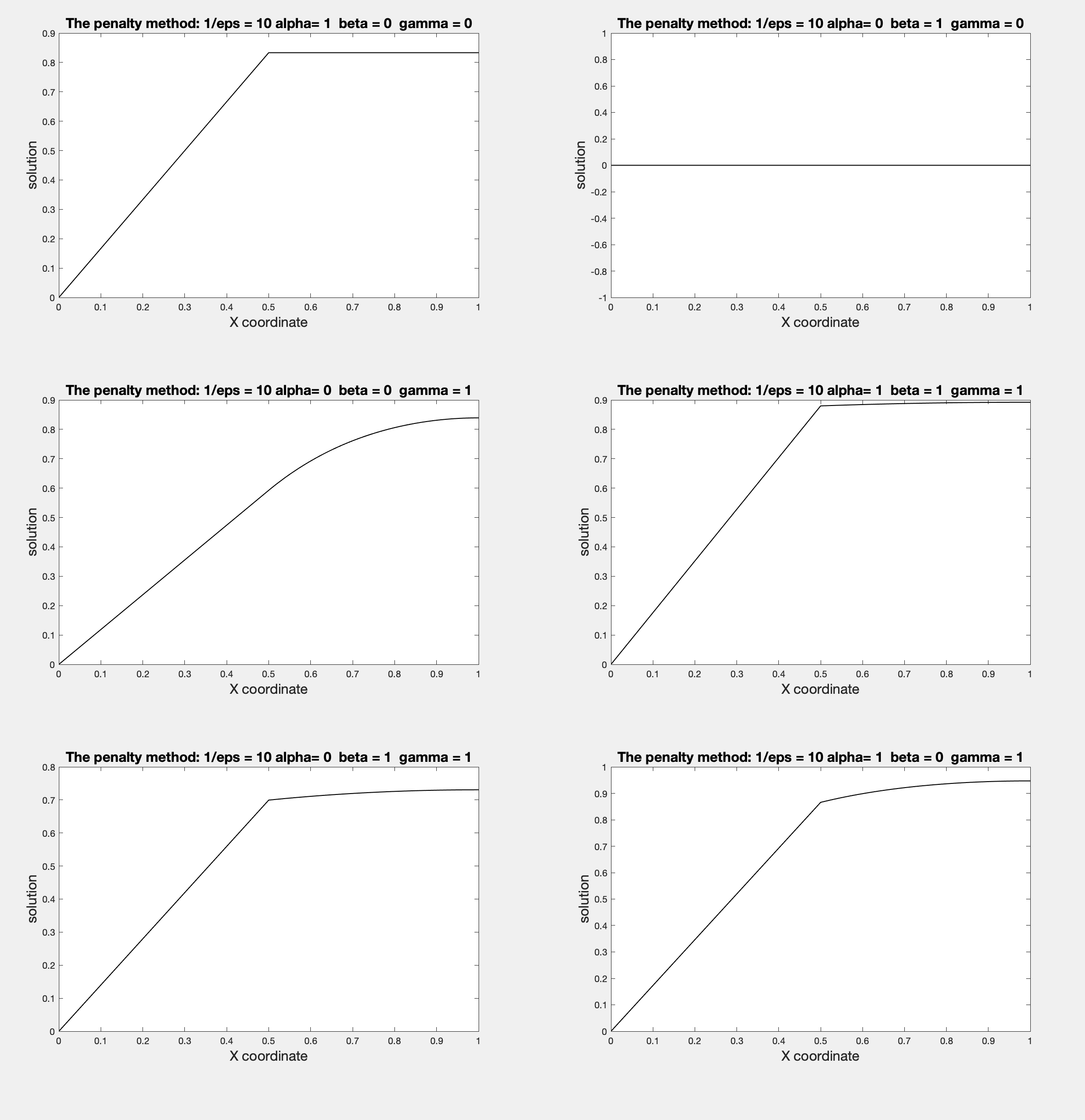} 
    \caption{Penalty solution for $\varepsilon=0.1$}
    \label{fig1}
 \end{figure}
 %=======
 \begin{figure}[htbp] %  figure placement: here, top, bottom, or page
    \centering
    \includegraphics[width=13cm,height=18cm]{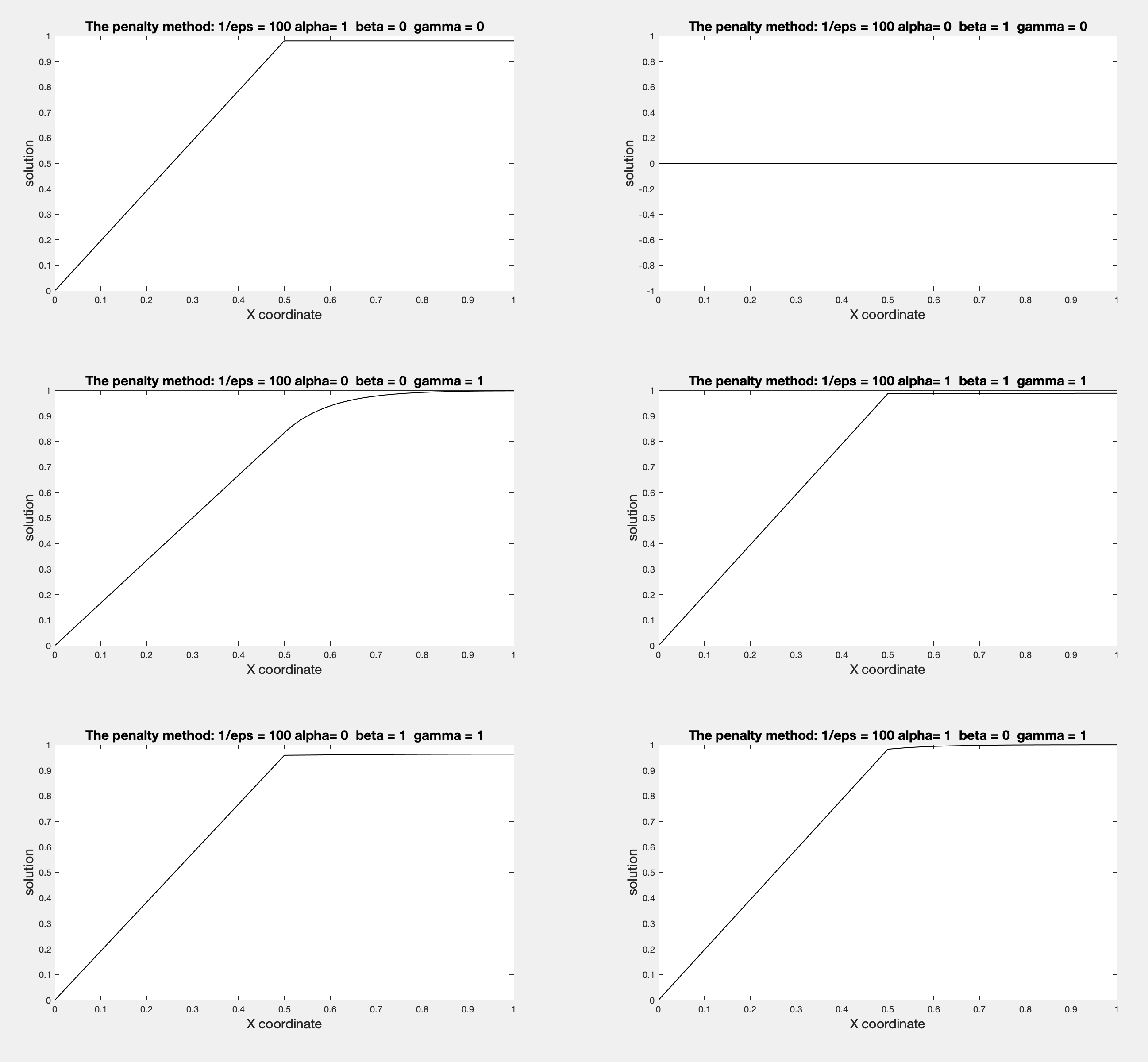} 
    \caption{Penalty solution for $\varepsilon=0.01$}
    \label{fig2}
 \end{figure}
 %
 %=========
 \begin{figure}[htbp] %  figure placement: here, top, bottom, or page
    \centering
    \includegraphics[width=13cm,height=18cm]{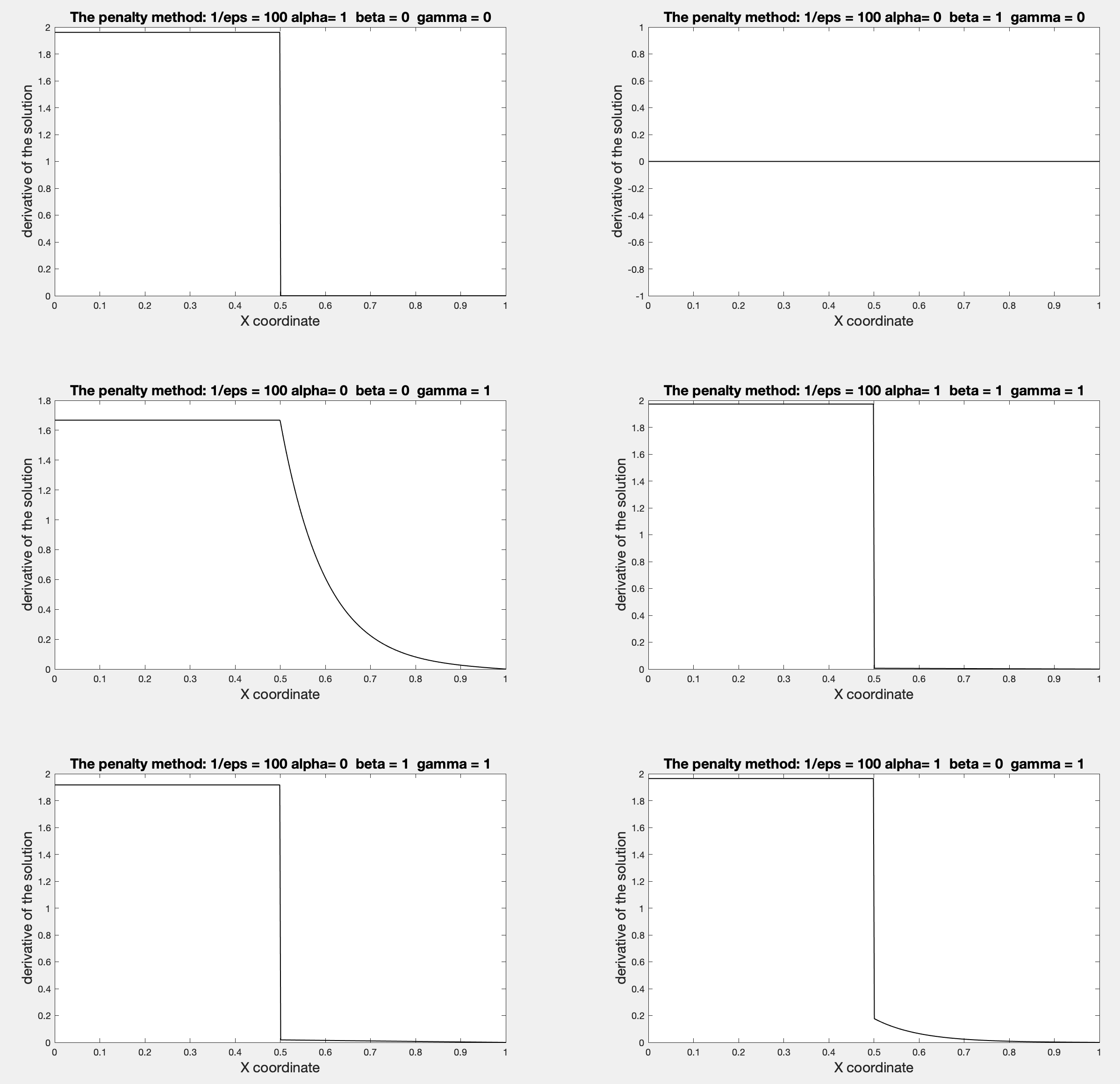} 
    \caption{Derivative in $x$ of the penalty solution for $\varepsilon=0.01$}
    \label{fig3}
 \end{figure}
 %========
 %
\section{The penalty-duality method}\label{penaltyduality} In this section, we consider the case where $\alpha>0$ and $\beta=\gamma=0$. 
Because the penalty model leads to ill conditioning, one can suggest to use a variant of this strategy introduced fifty years ago by D. Bertsekas \cite{bertsekas} and known as the penalty-duality method. The point is to satisfy exactly the constraints which are considered in the penalty term by a dual treatment. Hence, for any $r>0$ ($r$  stands for $1/\varepsilon$ but will be moderate at the opposite of what happens in the penalty strategy),  we introduce the Lagrangian defined by:
\beq\label{eq26}
L({\bold v},\lambda,\mu)=\fracj{1}{2}\intj_{\Omega}\sigma({\bold v}).\gamma({\bold v})+\fracj{\alpha r}{2}\intj_{\partial S}\vert\vert {\bold v}-{\bold u}_0\vert\vert^2+\intj_{\partial S}\mu.({\bold v}-{\bold u}_0)-\intj_{\Gamma_1} {\bold g}.{\bold v},\eeq
where ${\bold v}\in W_0,\;\mu\in {\Huge [}H^{-1/2}(\partial S){\Huge ]}^2$ (we use the notation with the integral for the duality\footnote{the dual of $H^{1/2}(\partial S)$ is $H^{-1/2}(\partial S)$ in  $1D$ which is the case here.} but we know that it is an abuse; see remark \ref{rem12}). One can add the penalty term inside $S$ (see \cite{liberge1}-\cite{liberge3}-\cite{liberge2}):
 $$\fracj{\gamma}{\varepsilon}\intj_S\vert {\bold u}-{\bold u}_0\vert^2.$$

It is not really useful for the static case that we consider here but it is a necessity  for the dynamic case as those treated in the numerical tests of subsections \ref{test3} and \ref{test4}.\\

A saddle point of $L$ is an element $({\bold u}^r,\lambda)\in W_0\times {\Huge [}H^{-1/2}(\partial S){\Huge ]}^2$ such that:
\beq\label{eq27}\left \{\begin{array}{l}
\forall {\bold v}\in W_0,\;\intj_\Omega\sigma({\bold u}^r).\gamma({\bold v})+{\alpha r}\intj_{\partial S}({\bold u}^r-{\bold u}_0).{\bold v}+
\intj_{\partial S}\lambda.{\bold v}=\intj_{\Gamma_1}{\bold g}.{\bold v},\\\\
\forall \mu\in {\Huge [}H^{-1/2}(\partial S){\Huge ]}^2,\;\intj_{\partial S}\mu.({\bold u}^r-{\bold u}_0)=0, 

\end{array}\right.\eeq
\begin{remark}\label{rem12} The writing with an integral of the duality between the spaces ${\Huge [}H^{-1/2}(\partial S){\Huge ]}^2$ and ${\Huge [}H^{-1/2}(\partial S){\Huge ]}^2$ is a familiarity which is not correct because the functions involved  are not in the space $L^2(\partial S)$. It just a commodity (in our mind). 
\hfill $\Box$
\end{remark}
The existence and uniqueness are relevant of the general A.K. Aziz and I. Babuska Theorem \cite{babuska} that we recall hereafter for sake of convenience.
\begin{theorem}\label{th3} [Aziz-Babuska] Let $H_1$ and $H_2$ be two Hilbert spaces, $b$ a bilinear and continuous form on $H_1\times H_2$ and $j$ a linear and continuous form on $H_2$. It is asuumed that the following properties hold:
\begin{enumerate}[\#]
\item 1 if $q\in H_2$ is such that: $\forall \mu\in H_1,\;b(\mu,q)=0$ then $q=0$;
\item 2 there exists a constant $\delta>0$  such that $\forall  \mu\in H_1,\;\supj_{q\in H_2}\fracj{b(\mu,q)}{\vert \vert q\vert\vert_{H_2}}\geq \delta \vert\vert \mu\vert\vert_{H_1}.$
\end{enumerate}
Then there exists a unique element  $\lambda\in H_1$ such that:
\beq\label{eq120}\forall q\in H_2,\;b(\lambda,q)=j(q).\eeq
\end{theorem}
Let us now apply this Theorem to our case.
\begin{theorem}\label{th4}{Existence and uniqueness of a solution to (\ref{eq27}). } Let $r>0$ and $\alpha>0$. The function ${\bold g}$ is assumed as previously to be in the space ${\Huge [}L^2(\Gamma_1){\Huge ]}^2$. Hence the system (\ref{eq110}) has a unique solution  $({\bold u}^r,\lambda)$in the space $W_0\times \left[ H^{-1/2}(\partial S)\right]^2$. Furthermore ${\bold u}^r={\bold u}$ solution of the initial problem (\ref{eq3}).
\end{theorem}

\begin{proof} 
If $({\bold u}^r,\lambda)\in W_0\times \left[ H^{-1/2}(\partial S)\right]^2$ is a solution it should satisfy the following relation:
\beq\label{eq111}
\forall {\bold v}\in W_{00},\;\intj_{\Omega}\sigma({\bold u}^r).\gamma({\bold v})=\intj_{\Gamma_1}{\bold g}.{\bold v},\;
\hbox{ and necessarily:\;}{\bold u}^r={\bold u}_0\hbox{ on }\partial S.
\eeq

Therefore ${\bold u}^r$ is the initial solution of equation (\ref{eq3}) which is also the term ${\bold u}^0$ computed in the asymptotic method(s) applied to the penalty method at section \ref{penalty} (for both $\alpha>0,\;\beta=0$ or $\alpha>0,\beta>0$). It also implies that it is unique and equal to the rigid body velocity ${\bold u}_0$ on the whole domain $S$ and not only on its boundary $\partial S$.
\\

Finally, in order to complete the proof of Theorem \ref{th4} it is sufficient to prove the existence and uniqueness of $\lambda$ using Aziz-Babuska Theorem \ref{th3}.

Let us set:
\beq\label{eq112}\left \{ \begin{array}{l}
H_1=\left[ H^{-1/2}(\partial S)\right]^2,\;H_2=\left[ H^{-1/2}(\partial S)\right]^2, 
\\
\\
\forall \mu\in H_1,\;\forall {\bold v}\in H_2,\;b(\mu,{\bold v})=\alpha\intj_{\partial S}\mu.{\bold v},\\\\
\forall {\bold v}\in H_2,\;j({\bold v})=\intj_{\Omega_S}\sigma({\bold u}^r).\gamma({\bold v})-\intj_{\Gamma_1}{\bold g}.{\bold v}.
\end{array}\right.
\eeq
The first requirement  of the Aziz-Babuska Theorem is obvious because of the duality between $H_1$ and $H_2$. The second one  is satisfied with $\delta=1$ because of the definition of the norm in $H_1$.

The only point to be verified is the definition and the continuity of $j$ on $H_2$. 
First of all $j$ is clearly linear. It depends on the value of ${\bold v}$ on the boundary $\partial S$ because of the definition of ${\bold u}^r={\bold u}^0$. Hence, for a given element ${\bold v}$ in $H_2$ one can associate any element in the space $V_0$ denoted also by ${\bold v}$ and which gives the same value of $j$. Let us recall that there exists a linear and continuous operator -say $R$- from $H_2$ into $V_0$ \cite{JLLEM} such that:
\beq\label{eq115}\exists c_9>0,\hbox{ independent on }{\bold v}\in H_1\hbox{ suh that: }
\vert\vert R({\bold v})\vert\vert_{1,\Omega_S}\leq c_9\vert\vert {\bold v}\vert\vert_{1/2,\partial S}.\eeq
Hence:
\beq \label{eq116}
\vert j({\bold v})\vert=\displaystyle \vert\intj_{\Omega_S}\sigma(R{\bold v}).\gamma(R{\bold v})-\intj_{\Gamma_1}{\bold g}.R({\bold v})\vert \leq c_{10}\vert\vert {\bold v}\vert\vert_{1/2,\partial S}, 
\eeq
Hence the assumptions of Theorem \ref{th3} are satisfied and we can conclude the proof of Theorem \ref{th4}. 
\end{proof}
\begin{remark}\label{rem90} At the opposite of what occurs in the penalty method, the solution of the penalty-duality model is exactly the one of the initial problem.\hfill $\Box$
\end{remark}
\begin{remark}\label{rem20}It is a basic point to notice that the solution ${\bold u}^r$ is independent on $r$. Therefore it is not necessary to choose a large value as in the penalty model ($r$ is the equivalent of $1/{\varepsilon}$) . Furthermore one could point out that the strategy could also be applied without the penalty term but in this case the algorithm studied in the next section is more difficult to use as we underline in the following (difficulty in adjusting the gradient step in the Uzawa  solution method \cite{bertsekas}, \cite{glo}). If one uses a conjugate gradient on the dual problem (in $\lambda$), the operator has a better condition number if one add  the penalty term. \hfill $\Box$\end{remark}
\begin{remark}\label{rem11}
A possibility that we do not recommend consists in considering the two penalty terms ($\alpha>0$ and $\beta>0$). Nevertheless the dual space for the Lagrange multiplier of the the constraint $\gamma_{ij}({\bold u})=0$ on $S$ is a little bit complex and not very convenient in a practical application. Furthermore the numerical  implementation  seems to be some more complicated. \hfill $\Box$
\end{remark}
%
%============
\section{The numerical algorithm and its solution method} In this section we focus on the Uzawa algorithm \cite{glo}\cite{cea} for solving (\ref{eq27}). Let us start from a given value for $\lambda^{n_0}\in  {\Huge [}H^{-1/2}(\partial S){\Huge ]}^2$ (the most classical choice is to start from $\lambda^{n_0}=0$). For each $n\geq n_0$, we define ${\bold u}^n\in W_0$ solution of:
\beq\label{eq121}
\forall {\bold v}\in W_0,\;\intj_{\Omega}\sigma({\bold u}^n).\gamma({\bold v})+{r}\intj_{\partial S}({\bold u}^n-{\bold u}_0).{\bold v}=\intj_{\Gamma_1}{\bold g}.{\bold v}-\intj_{\partial S}\lambda^n.{\bold v}.\eeq
Then we upgrade $\lambda^n$ by setting:
\beq\label{eq122}
\lambda^{n+1}=\lambda^n+r({\bold u}^n-{\bold u}_0)\;\hbox{ on }\partial S.
\eeq
An important point is that if the first multiplier $\lambda^n$ belongs to the space ${\Huge[}L^2(\partial S){\Huge]}^2$ all the sequence $\lambda^n$ belongs to this space. But in general the solution $\lambda$ of (\ref{eq27}) doesn't belongs to this space unless one has a regularity property. In fact, one has (interpretation of (\ref{eq27})), 
$\lambda=-\sigma({\bold u}^r).\nu$ on $\partial S$. This regularity assumption is exactly the one used in the asymptotic analysis of the penalty model for $\alpha>0$ and $\beta=0$. It is convenient (but not necessary) in the following to adopt this hypothesis (the norm $L^2(\partial S)$ is more convenient in the writings than the one of $H^{-1/2}(\partial S)$).
\vskip.2cm

The convergence of the  algorithm is classical \cite{bertsekas}. But for sake of clarity let us summarize the proof hereafter with few remarks concerning the regularity of the multiplier $\lambda$ which are only valid for our case.
\begin{theorem}\label{th5}Let $({\bold u}^n,\lambda^n)\in W_0\times {\Huge [}H^{-1/2}(\partial S){\Huge ]}^2$ be the sequence defined by the algorithm (\ref{eq121})-(\ref{eq122}). We assume, just for sake of convenience in the writings, that: $$\lambda=-\sigma({\bold u}^r).\nu\in {\Huge [}L^2(\partial S){\Huge ]}^2.$$ Then for any $r>0$:
\beq\label{eq130}\left \{ \begin{array}{l}
\limj_{n\rightarrow \infty}\;{\bold u}^n={\bold u}^r={\bold u}\hbox{ solution of (\ref{eq3})},\\\\
\limj_{n\rightarrow \infty}\;\lambda^n=\lambda\hbox{  solution with ${\bold u}^r$ of (\ref{eq27})}.
\end{array}\right.
\eeq
\end{theorem}
\begin{proof} 
Let us introduce the {\it gap} variables:
$$\underline{\bold u}^n={\bold u}^n-{\bold u}^r,\;\;\underline{\lambda}^n=\lambda^n-\lambda.$$
One has the following equalities:
$$\left \{ \begin{array}{l}\forall {\bold v}\in W_0,\;\intj_{\Omega}\sigma(\underline{\bold u}^n).\gamma(\underline{\bold u}^n)+{r}\intj_{\partial S}\vert \underline{\bold u}^n\vert ^2+\intj_{\partial S}\underline{\bold u}^n\underline\lambda^n=0,
\\\\
\vert\underline{\lambda}^{n+1}\vert^2=\vert\underline {\lambda}^{n}\vert^2+{r^2}\vert \underline{\bold u}^n\vert ^2+{2r}\underline{\lambda}^n.\underline{\bold u}^n.
\end{array}\right.
$$
We deduce that:
$$\intj_{\partial S}\vert\underline{\lambda}^{n}\vert^2-\intj_{\partial S}\vert\underline{\lambda}^{n+1}\vert^2=r{\Huge [}-2\intj_{\partial S}\underline{\bold u}^n\underline\lambda^n-r\intj_{\partial S}\vert \underline{\bold u}^n\vert ^2{\Huge ]},$$
or else:
$$\begin{array}{l}\intj_{\partial S}\vert\underline{\lambda}^{n}\vert^2-\intj_{\partial S}\vert\underline{\lambda}^{n+1}\vert^2=r{\Huge [}2\intj_{\Omega} \sigma(\underline{\bold u}).\gamma(\underline{\bold u}^n)+2r\intj_{\partial S}\vert\underline{\bold u}^n\vert^2-r\intj_{\partial S}\vert \underline{\bold u}^n\vert^2{\Huge ]}\\\\
=r{\Huge [}2\intj_{\Omega} \sigma_{ij}(\underline{\bold u}).\gamma(\underline{\bold u}^n)+r\intj_{\partial S}\vert\underline{\bold u}^n\vert^2{\Huge ]}\geq 0.\end{array}$$
Because the real and positive sequence of numbers $q^n=\intj_{\partial S}\vert \underline {\lambda}^n\vert^2$ is decreasing, it is convergent and therefore it is a Cauchy sequence. Hence $\limj_{n\rightarrow \infty} q^n-q^{n+1}=0.$ As a consequence (using the $H^1(\Omega)$-coerciveness of the bilinear form $\intj_{\Omega}\sigma({\bold u}).\gamma({\bold v})$):
$$\left \{ \begin{array}{l}\limj_{n\rightarrow \infty}\vert\vert{\bold u}^n-{\bold u}_0\vert\vert_{0,\partial S}=0,
\\\\
\limj_{n\rightarrow \infty}\vert\vert {\bold u}^n-{\bold u}^r\vert\vert_{1,\Omega}=0.\end{array}\right. $$
The second convergence result reinforces the first one because of the continuity of the trace mapping from $H^1(\Omega)$ into $H^{1/2}(\partial S)$:
$$\limj_{n\rightarrow \infty}\vert\vert{\bold u}^n-{\bold u}_0\vert\vert_{1/2,\partial S}=0.$$
The final step consists in proving the convergence of the multiplier $\lambda^n$. From:
$$ \intj_{\partial S}\underline {\lambda}^n.{\bold v}=-\intj_{\Omega}\sigma( \underline{\bold u}^n).\gamma({\bold v})-r\intj_{\partial S}\underline {\bold u}^n.{\bold v}-\intj_{\Gamma_1}{\bold g}.{\bold v},$$
and because of the property $\#$ 2 mentioned in Theorem \ref{th3}, one has:
$$\limj_{n\rightarrow \infty}\vert\vert \lambda^n-\lambda\vert\vert_{-1/2,\partial S}=0.$$
 Even if we assumed the regularity $\lambda\in {\huge [}L^2(\partial S){\Huge ]}^2$ the standard convergence is only in the space ${\huge [}H^{-1/2}(\partial S){\Huge ]}^2$.
\end{proof}
\begin{remark}\label{rem30}
The main advantage in the choice of the penalty term reduced on the boundary of the structure is that it enables to decouple completely the numerical software and just requires a geometrical prolongation mapping from the boundary $\partial S$ of the structure into the open set $\Omega$ and a  geometrical restriction mapping from $\Omega$ on to $\partial S$. This advantage was underlined by many authors and a nice presentation by J. Hovnanian is given in \cite{penalty1}. \hfill $\Box$
\end{remark}
%
%============
\section{Few elementary tests for the penalty-duality strategy}\label{test2} Just in order to make a link with the elementary test introduced in section \ref{test1}, we apply the penalty-duality algorithm to the same model. The results are plotted on Figures \ref{fig4}-\ref{fig5}. They show the efficiency of this method regarding the precision even for $\varepsilon$ not so small. The numerical implementation in a complex model is certainly another discussion which will be carried out in A. Falaise and E. Liberge \cite{elaf}. In particular, the mapping connecting the boundary $\partial S$ with the global mesh is a cornerstone problem regarding the multiprocessor programming. 
%Another point is the use of a preconditioning of the dual problem in the penalty-duality strategy in order to increase the performances.
%
\begin{figure}[htbp] %  figure placement: here, top, bottom, or page
   \centering
   \includegraphics[width=11cm,height=7cm]{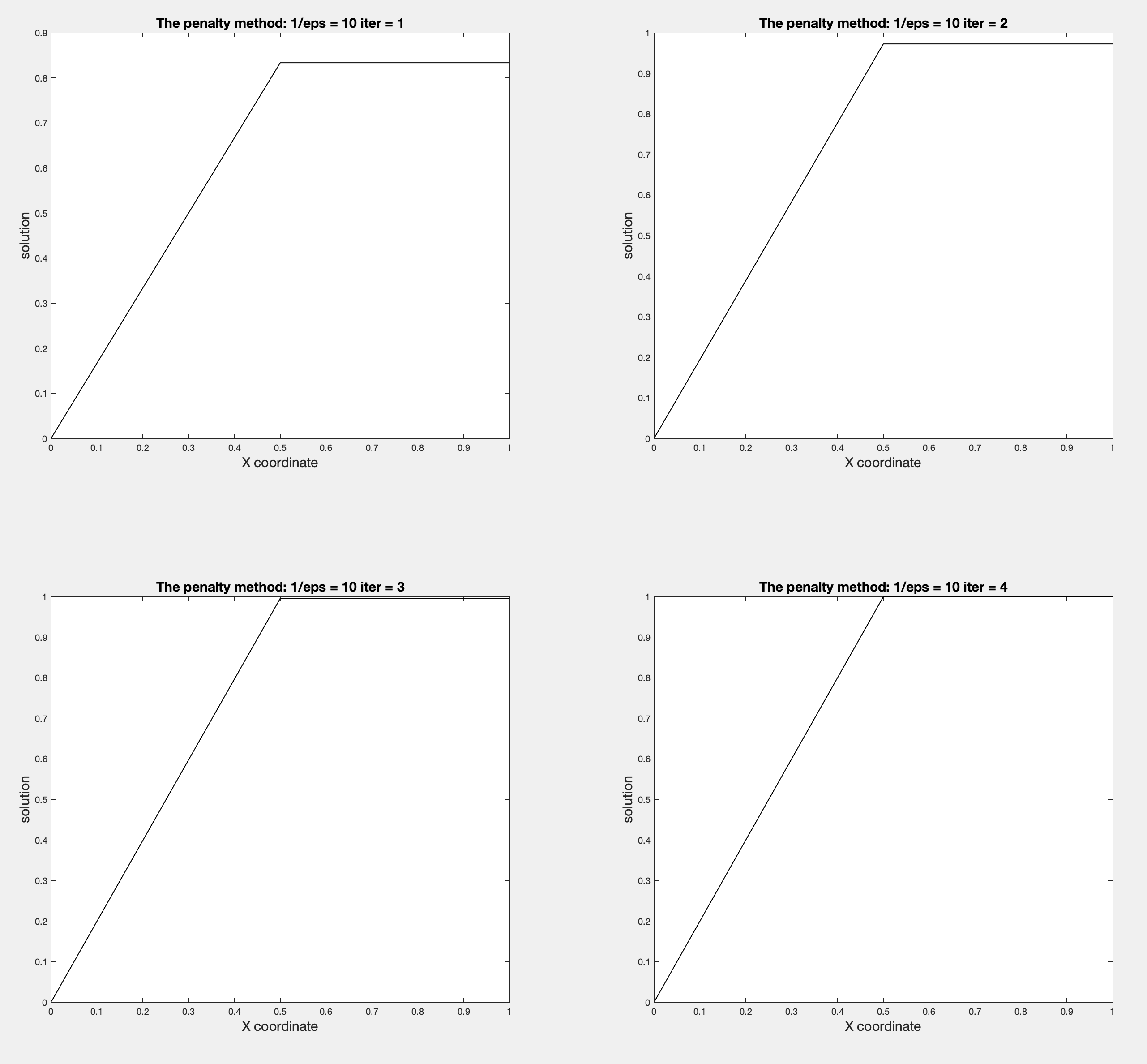} 
   \caption{Solution obtained with the penalty-duality strategy with $\varepsilon=.1$}
   \label{fig4}
\end{figure}
\begin{figure}[htbp] %  figure placement: here, top, bottom, or page
   \centering
   \includegraphics[width=11cm,height=7cm]{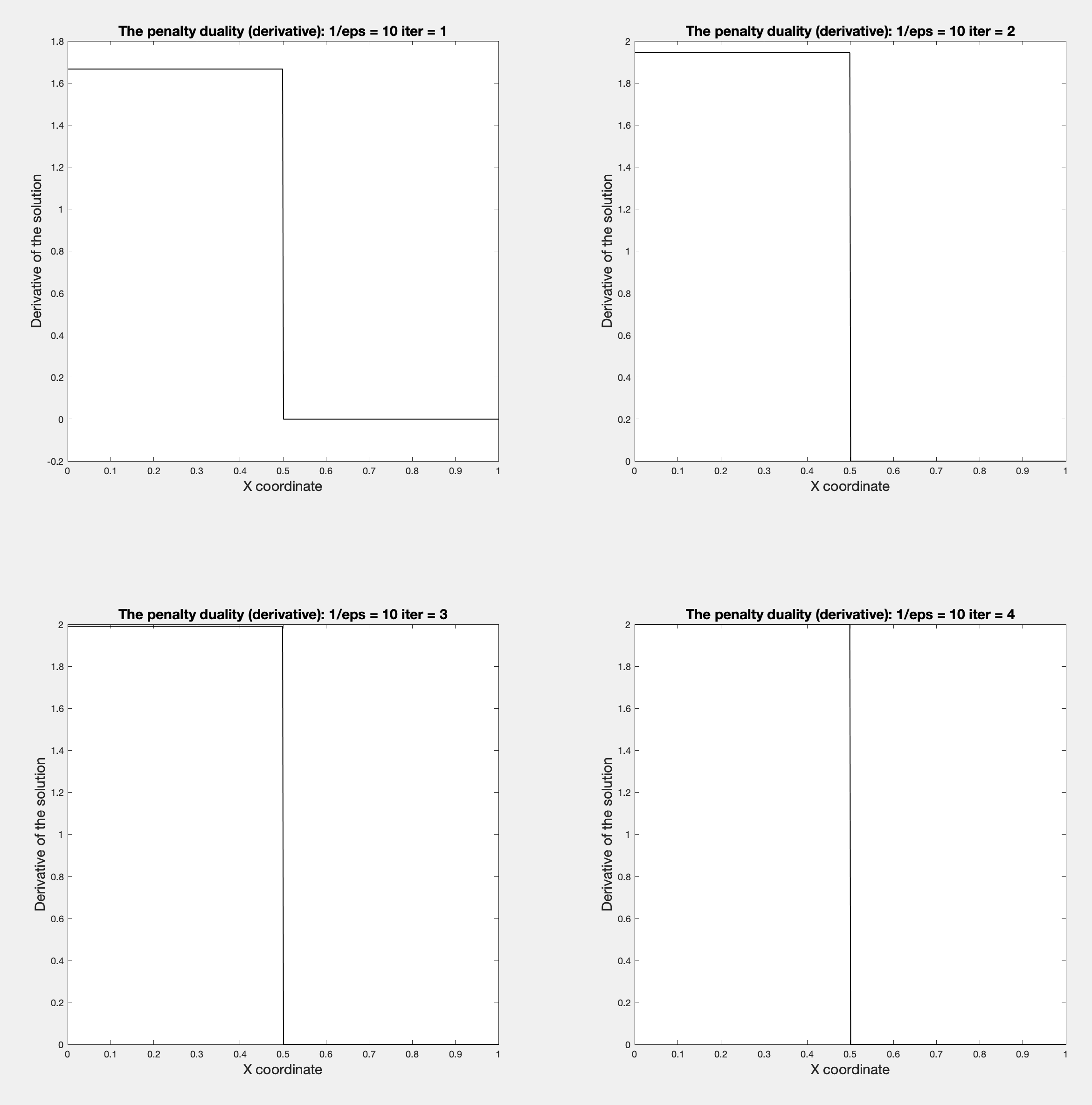} 
   \caption{Derivative of the solution obtained with the penalty-duality strategy with $\varepsilon=.1$ during the four first iterations}
   \label{fig5}
\end{figure}
%============
\section{Application to convection-diffusion equation}\label{advection} We consider two examples: the first one is the linear advection-diffusion model and the second one is the Burgers equation. Nevertheless it is worth to point out that the boundary of the structure are not updated..
\subsection{The linear advection diffusion model}\label{modeladvection} We consider the following $1D$ advection-diffusion equation where $u\in {\mathcal C}^0([0,T]\times[0,L])$ is the unknown (where $c>0$ $u_0(t)$ and $d(x)$ are given):
\beq\label{eq200bis}\left \{ \begin{array}{l}x\in]0,a[\cup]b,L[,\;t\geq 0:\;\fracj{\partial u}{\partial t}+c\fracj{\partial u}{\partial x}-\nu\fracj{\partial^2 u}{\partial x^2}=0,\\\\
t\geq0\;u(0,t)=u(L,t),\;\;x\in ]0,a[\cup]b,L[:\;u(x,0)=u_0(x),\\\\t\geq 0,\;x\in]a,b[:\;u(x,t)=d_0(t)\;(\hbox{ sine function}).\end{array}\right.\eeq
We denote by $\Delta x$ (respectively $\Delta t$) the space step (respectively the time step). The lower index indicate the space dicretization and the upper one the time discretization. The numerical scheme that we consider for this equation in $(]0,a[\cup]b,L[)\times ]0,T[$ leads to:
\beq\label{eq201}\left\{ \begin{array}{l}\forall i\in\{2,i_a\}\cup\{i_b,N-1\},\\\\\fracj{u^{n+1}_i-u^n_i}{\Delta t}+c(\fracj{ u_i^n-u^n_{i-1}}{\Delta x})-\nu(\fracj{ u^{n+1}_{i+1}-2u^{n+1}_{i}+u^{n+1}_{i-1}}{\Delta x^2})=0,\\\\
u_i^0=u_0(i\Delta x) \hbox{ initial conditions,}\\\\
\forall n\geq 0:\;u_1^n=u_N^n,\hbox{ boundary conditions, }\\\\\;u_{i_a}^n=u_{i_b}^n=d_0^n\;\hbox{ boundary of the structure},\\\\\hbox{ where }i_a\hbox{ and }i_b\;\hbox{ are the indices corresponding to $a$ and $b$}.
\end{array}\right.\eeq
The penalty-duality formulation is the following one ($\delta_{i,i_c}$ is the Kronecker symbol $=1$ if $i=i_c$ and $0$ else and we introduce a penalty term with the parameter $\gamma$ as before):
\beq\label{eq300}\left\{ \begin{array}{l}\forall i\in\{2,N-1\},\\\\\fracj{u^{n+1}_i-u^n_i}{\Delta t}+\fracj{\gamma}{\varepsilon}(u_i^{n+1}-d_0^{n+1})+c(\fracj{ u_i^n-u^n_{i-1}}{\Delta x})+r[(u_i^{n+1}-d_0^{n+1})\delta_{i,i_a}+\\\\(u_i^{n+1}-d_0^{n+1})\delta_{i,i_b}]-\nu(\fracj{ u^{n+1}_{i+1}-2u^{n+1}_{i}+u^{n+1}_{i-1}}{\Delta x^2})+\\\\\lambda^{n+1}(i_a)\delta_{i,i_a}+\lambda^{n+1}(i_b)\delta_{i,i_b}=0,\;\;
u_i^0=u_0(i\Delta x) \hbox{ initial conditions,}\\\\
\forall n\geq 0:\;u_1^n=u_N^n\hbox{ and }u_{-1}^n=u_{N-1}^n\hbox{ periodic boundary conditions, }\\\\\forall n\geq 0:\;u_{i_a}^n=u_{i_b}^n=d_0^n\;\hbox{ boundary of the structure}.
\end{array}\right.\eeq
For each $n$ the duality algorithm (the iterations are indexed by $p$) is defined as follows:
\begin{itemize}[\#]
\item $p=0,\;\lambda_{i_a}^0=\lambda_{i_b}^0=0$;
\item compute $u^{n+1}_i$ solution of (\ref{eq300}) $\lambda^p=(\lambda^p_{i_a},\lambda^p_{i_b})$ being fixed;
\item set: $\lambda^{p+1}=\lambda^p+r(u_{i_a}^n-d_0^n,u_{i_b}^n-d_0^n)$
\item convergence test => stop or $p=p+1$
\end{itemize}
\subsection{The numerical tests for the advection-diffusion (linear)}\label{test3} We have tested the two possibilities: a) no duality and b) with duality. The results are plotted on Figures \ref{fig10} and \ref{fig11} for the first case and on Figures \ref{fig12} and \ref{fig13} for the case with duality. We choose the following data set: 
\begin{itemize}[\#]
\item the kinematical viscosity is $\nu=0.001$; 
\item the coefficient of the Bertsekas algorithm is $r=10$; 
\item the length of the space interval is $L=1$ and the length of the structure is  $L/10$ and the coordinates of the two extremities are denoted by $x_a=.45L,\;x_b=.55L$; 
\item the time dependence of the prescribed velocity of the structure is ${u}_{st}(t)=\sin(2\pi t)$ and $c=1$ for the advection-diffusion case. The space dependence in the structure is (just for an example):
$$u_{sx}(x)= 0.4+2(2x-x_a-x_b)/L.$$
 Hence the velocity prescribed in the structure is $u_{st}(t)u_{sx}(x)$;
\item the initial condition outside the structure is (still for example):

$$\begin{array}{l}u(x,0)=4.4(x_a-x-L/5) \hbox{ if } x\in [0,x_a-L/5],\\u(x,0)=4.4(x-x_b-L/5)\hbox{ if }x\in [x_b+L/5,L]\;\hbox{ and $u(x,0)=0$ elsewhere;}\end{array}$$
 \item the penalty parameter used for the periodic conditions at $x=0$ and $x=L$ but also for the penalty term inside the structure is $\varepsilon=10^{-3}$;
 \item  the time delay is $T=2$ and there are $1000$ time steps and 500 space steps.
 \end{itemize}
  There is a meaningful difference between the two cases and mainly on the derivatives with respect to $x$ if $\varepsilon>10^{-8}$.  For $\varepsilon<10^{-8}$ the penalized solution is quite the one obtained with the penalty-duality method. In other words the duality algorithm is not useful. Nevertheless,  for such values of the penalty parameter the solver requires a large enough viscosity which is not so drastic for larger values of $\varepsilon$. Furthermore the one dimensional test aren't meaningful concerning the condition number of two or three dimensional models. Our goal in this paper is only to discuss the advantages and the drawbacks of this added duality algorithm for fluid-structure models from a theoretical point of view.
{Figures \ref{fig1011} show the solution (ad) (Figure \ref{fig10}) and  the derivative  versus $x$ of the solution (ad) (Figure \ref{fig11} obtained without neither penalty inside the structure nor duality, but with a penalty term on the boundary between the structure and the fluid. The structure is moving at a given frequency (2 periods have  been plotted). This induces waves in the fluid. One should notice on Figure \ref{fig11} that the propagation of stress waves is very damped in the fluid. The solutions obtained with the duality method are plotted on Figure \ref{fig1213}. The normal stress at the two boundaries versus time, with and without duality are plotted Figure \ref{fig1112bis}.

\begin{figure}[!htbp]
\subfigure[\label{fig10} Solution ]{\includegraphics[width=5.5cm,height=5.5cm]{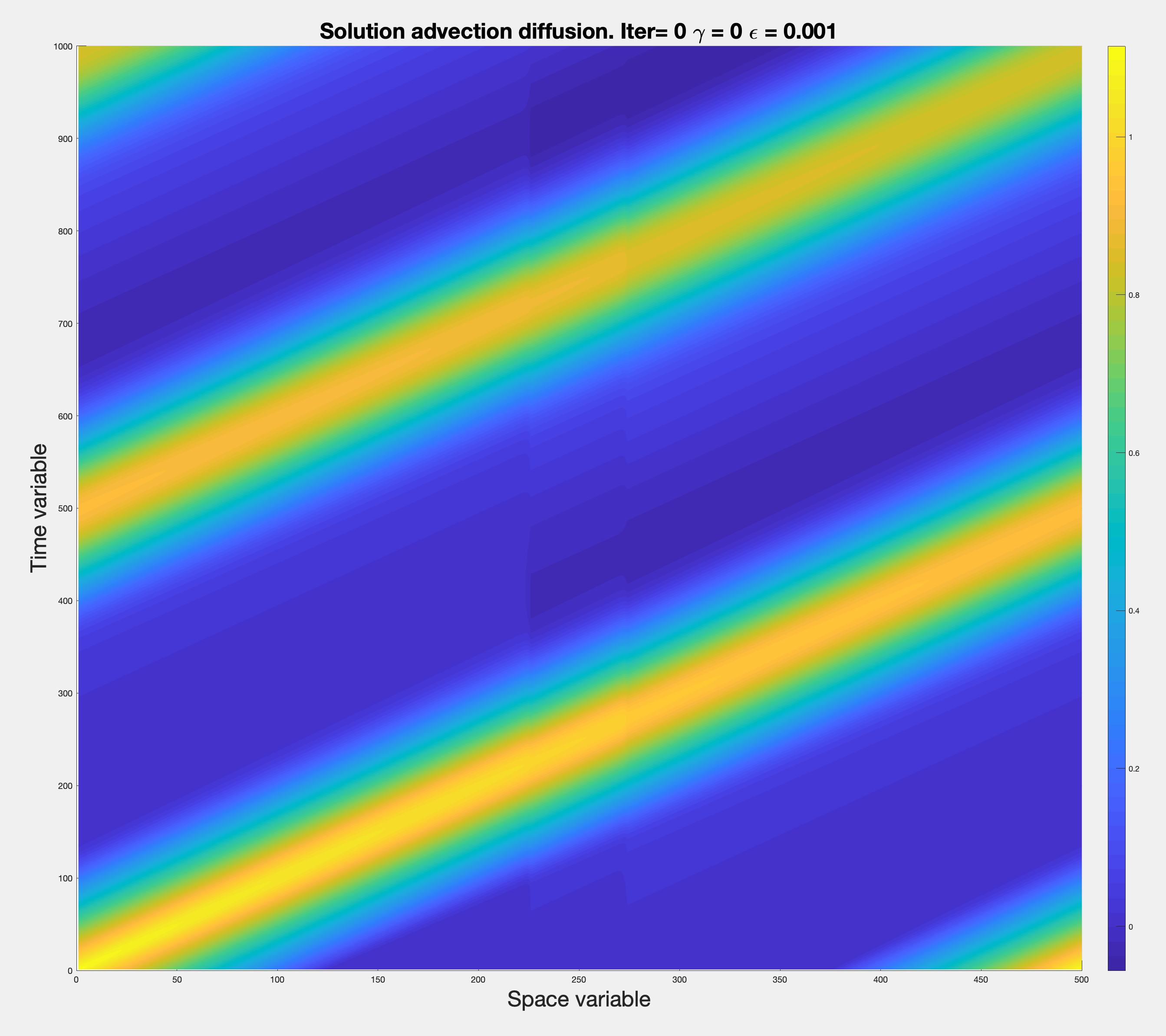} }
\subfigure[\label{fig11} The derivative  versus $x$ of the solution (ad)  ]{\includegraphics[width=5.5cm,height=5.5cm]{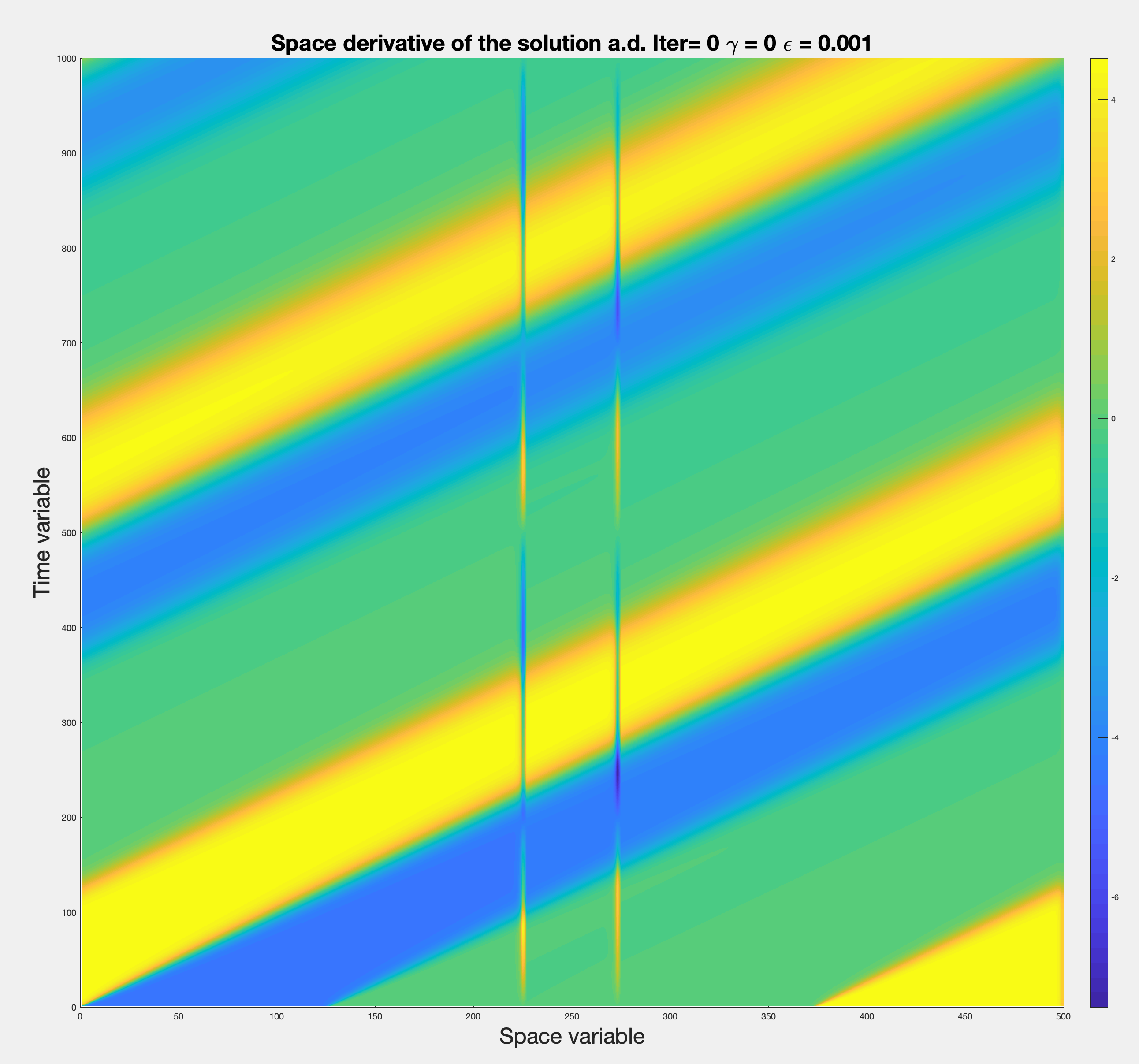}}
\caption{Solution (ad) obtained without neither penalty inside the structure nor duality. But there is a penalty term on the boundary between the structure and the fluid.\label{fig1011}}
\end{figure}

%\begin{figure}[htbp] %  figure placement: here, top, bottom, or page
%   \centering
%   \includegraphics[width=10cm,height=8cm]{ad1.png} 
%   \caption{Solution (ad) obtained without neither penalty inside the structure nor duality. But there is a penalty term on the boundary between the structure and the fluid. The structure is moving at a given frequency (2 periods have  been plotted. This induces waves in the fluid. }
%   \label{fig10}
%\end{figure}
%
%\begin{figure}[htbp] %  figure placement: here, top, bottom, or page
%   \centering
%   \includegraphics[width=10cm,height=8cm]{ad2.png} 
%   \caption{The derivative  versus $x$ of the solution (ad) without neither duality nor penalty inside the structure  (same case as on Figure \ref{fig10}. But there is a penalty term on the boundary between the structure and the fluid. One should notice that the propagation of stress waves is very damped in the fluid.}
%   \label{fig11}
%\end{figure}
\begin{figure}[!htbp]
\subfigure[\label{fig12} Solution (ad)]{\includegraphics[width=5.5cm,height=5.5cm]{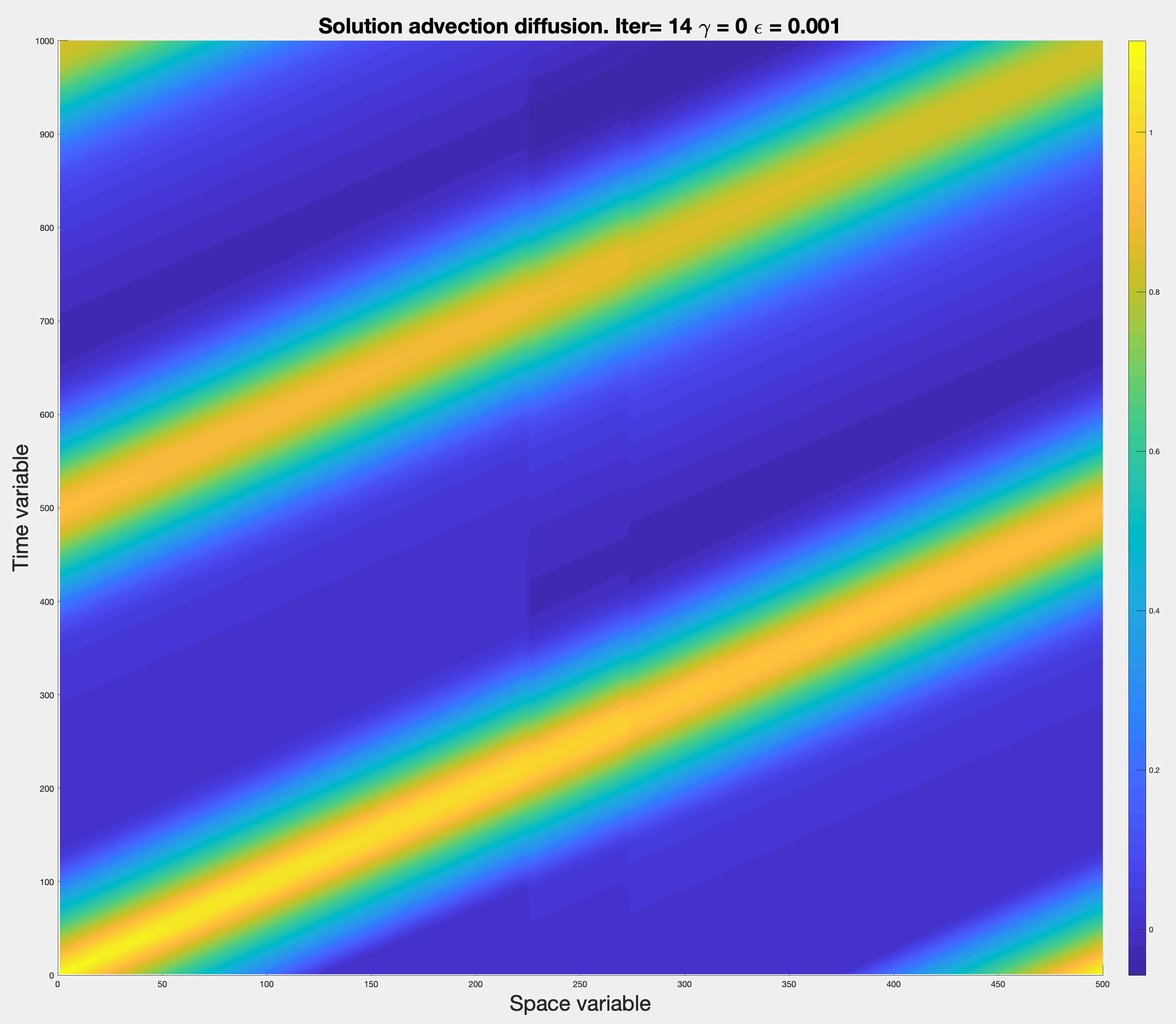} }
\subfigure[\label{fig13} The derivative  versus $x$ of the solution (ad)]{\includegraphics[width=5.5cm,height=5.5cm]{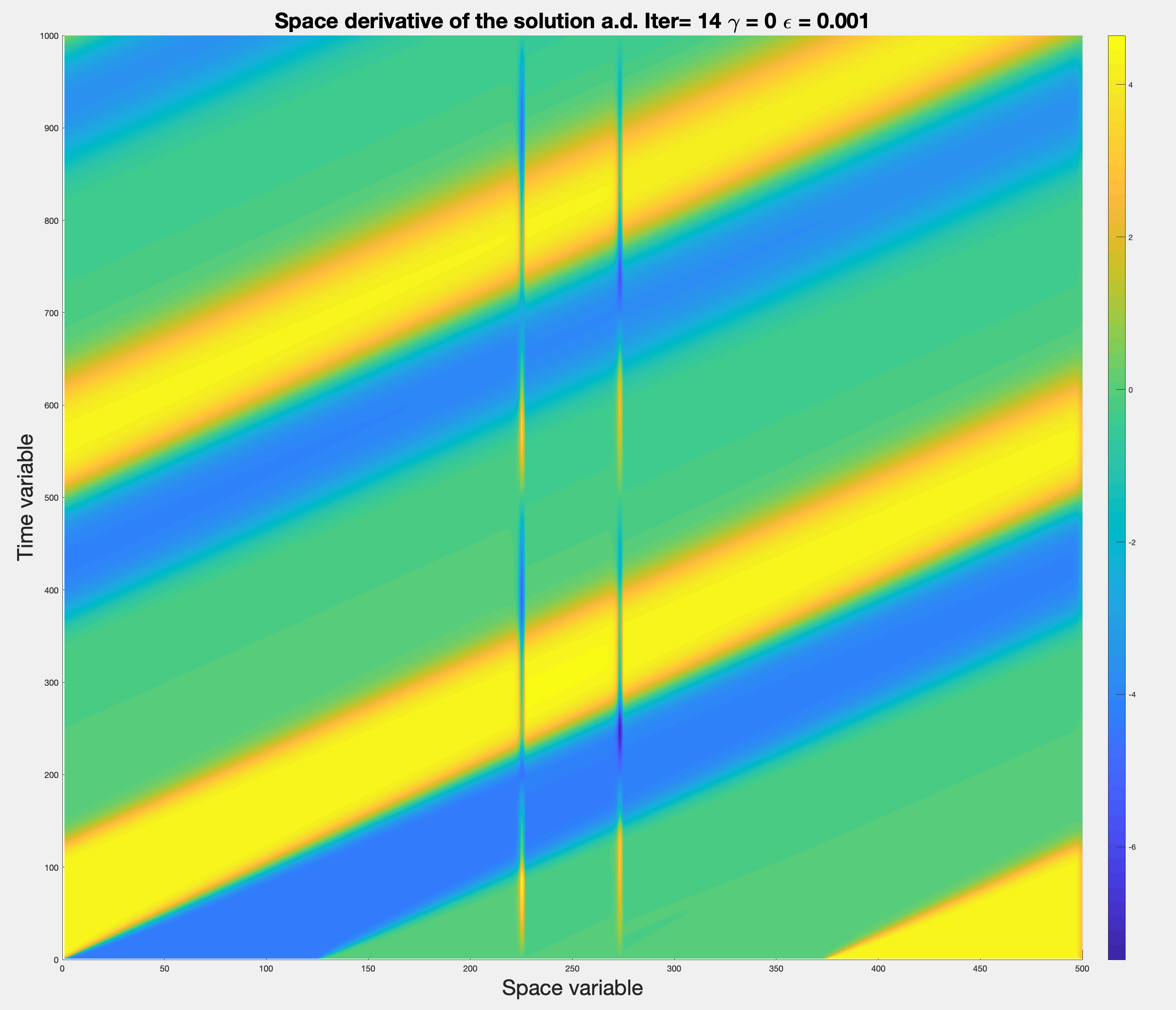} }
\caption{ Solution (ad) obtained with the duality method on the boundary of the structure but no penalty term inside the structure. \label{fig1213} }
\end{figure}

%\begin{figure}[htbp] %  figure placement: here, top, bottom, or page
%   \centering
%   \includegraphics[width=10cm,height=8cm]{ad4.png} 
%   \caption{Solution (ad) obtained with the duality method on the boundary of the structure but no penalty term inside the structure. This picture should be compared to Figure \ref{fig10}.}
%      \label{fig12}
%\end{figure}
%
%\begin{figure}[htbp] %  figure placement: here, top, bottom, or page
%   \centering
%   \includegraphics[width=10cm,height=8cm]{ad5.png} 
%   \caption{Derivative versus $x$ of the solution (ad) obtained with duality but no penalty term inside the structure.Should be compared to Figure \ref{fig11}.}
%   \label{fig13}
%\end{figure}

%
\begin{figure}[!htbp]
\subfigure[\label{fig11bis}without duality ]{\includegraphics[width=5.5cm,height=5.5cm]{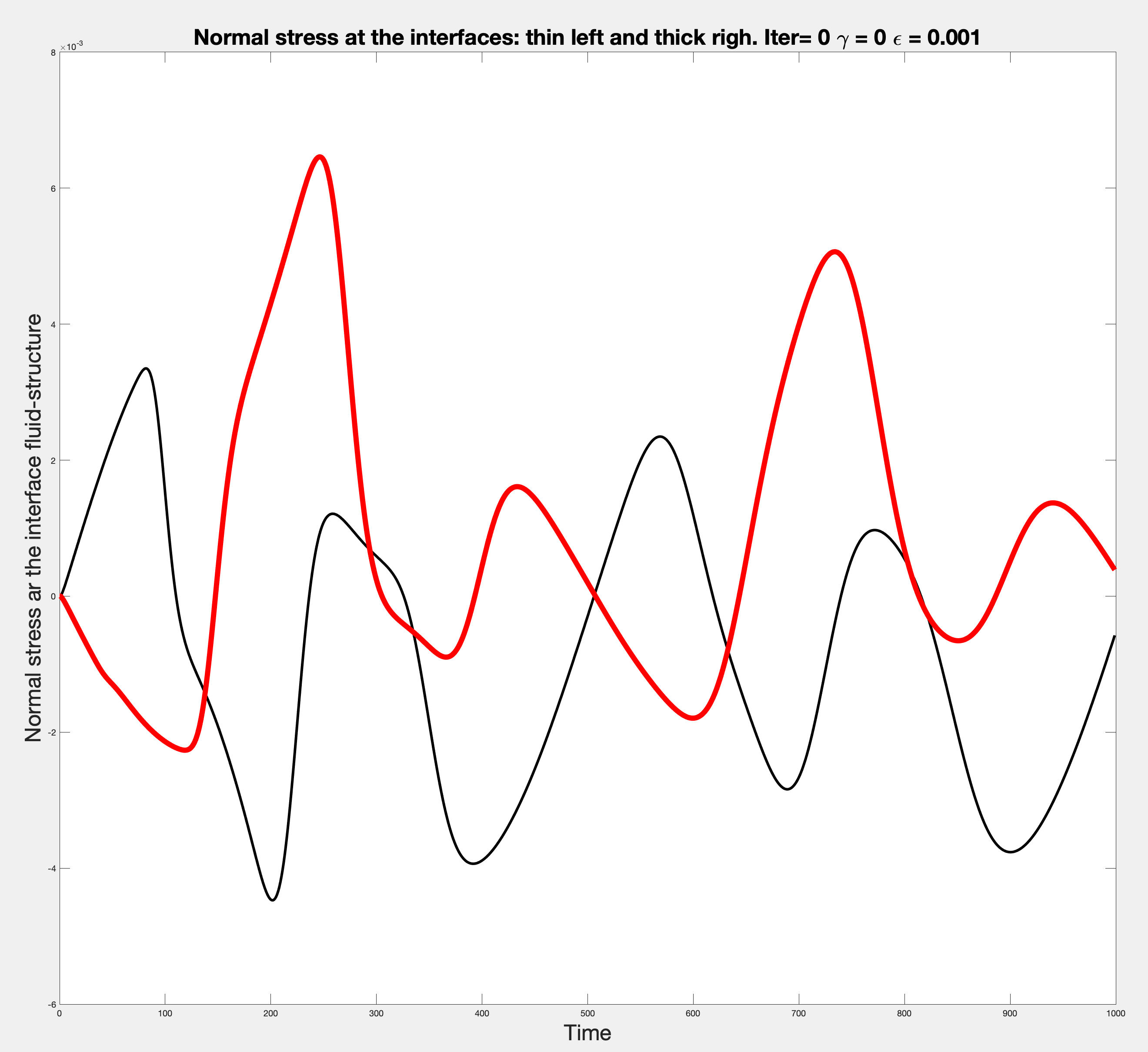}}
\subfigure[\label{fig12bis} with duality]{\includegraphics[width=5.5cm,height=5.5cm]{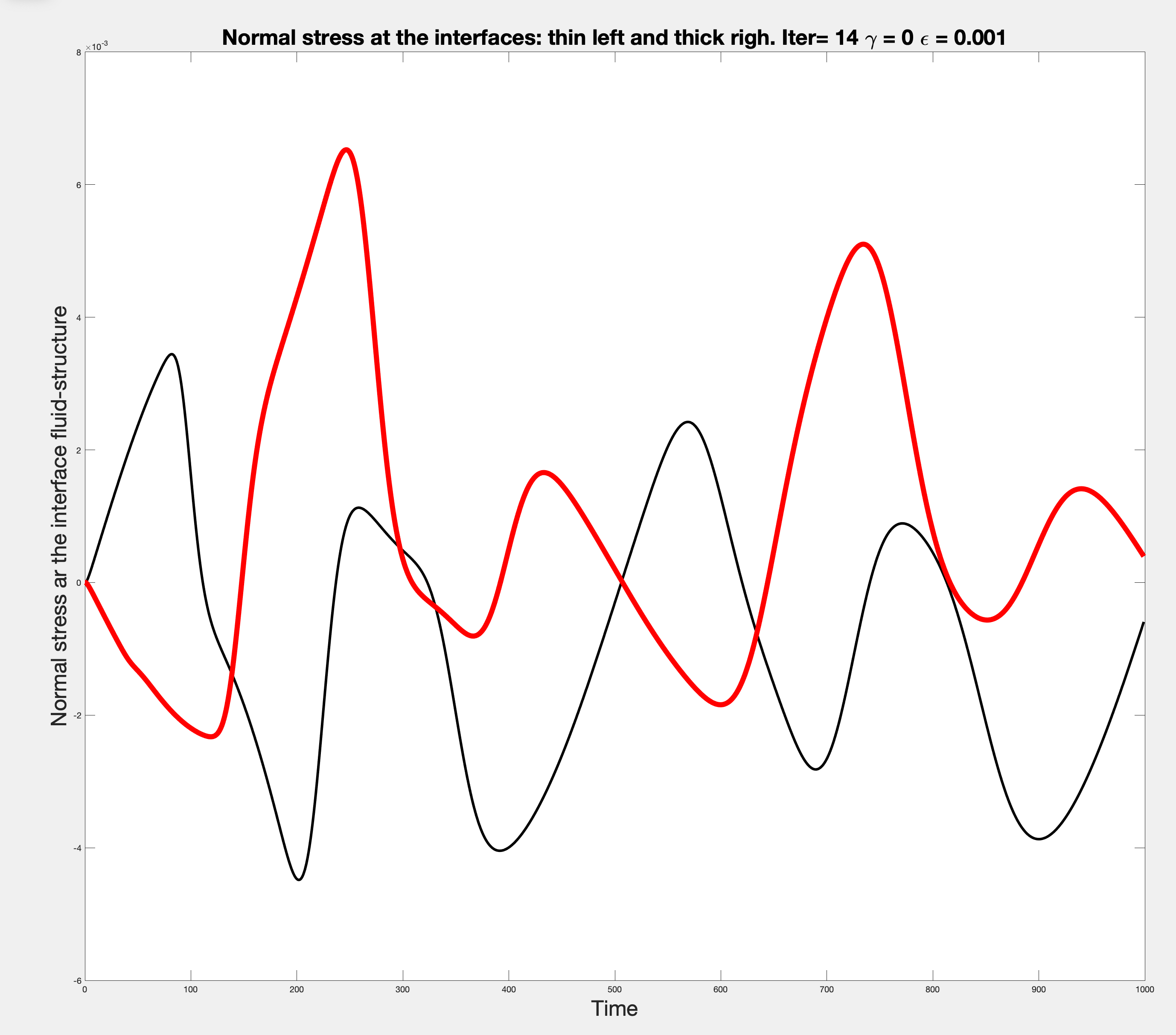}}
 \caption{The normal stress at the two boundaries of the structure versus the time \label{fig1112bis}}
\end{figure}

%\begin{figure}[htbp] %  figure placement: here, top, bottom, or page
%   \centering
%   \includegraphics[width=10cm,height=8cm]{ad3.png} 
%   \caption{The normal stress at the two boundaries of the structure versus the time are plotted on this Figure fo the case without neither duality nor penalty inside the structure. But there is a penalty term on the boundary between the structure and the fluid.}
%   \label{fig11bis}
%\end{figure}
%

%
%\begin{figure}[htbp] %  figure placement: here, top, bottom, or page
%   \centering
%   \includegraphics[width=10cm,height=8cm]{ad6.png} 
%   \caption{This Figure represents the normal stress at the interfaces between the structure and the fluid. The duality has been used and It should be compared with Figure \ref{fig11bis}}
%   \label{fig12bis}
%\end{figure}
%

{Next, the penalty term inside the structure is added. Result with and without duality are plotted on Figures \ref{fig1314} and \ref{fig1516}. The duality algorithm does not affect the normal stress at the fluid-solid interface as far as $\varepsilon$ is very small  (Figure \ref{fig1718}) but the use of the penalty term makes an important difference.}
\begin{figure}[!htbp]
\subfigure[\label{fig13bis} Solution ]{\includegraphics[width=5.5cm,height=5.5cm]{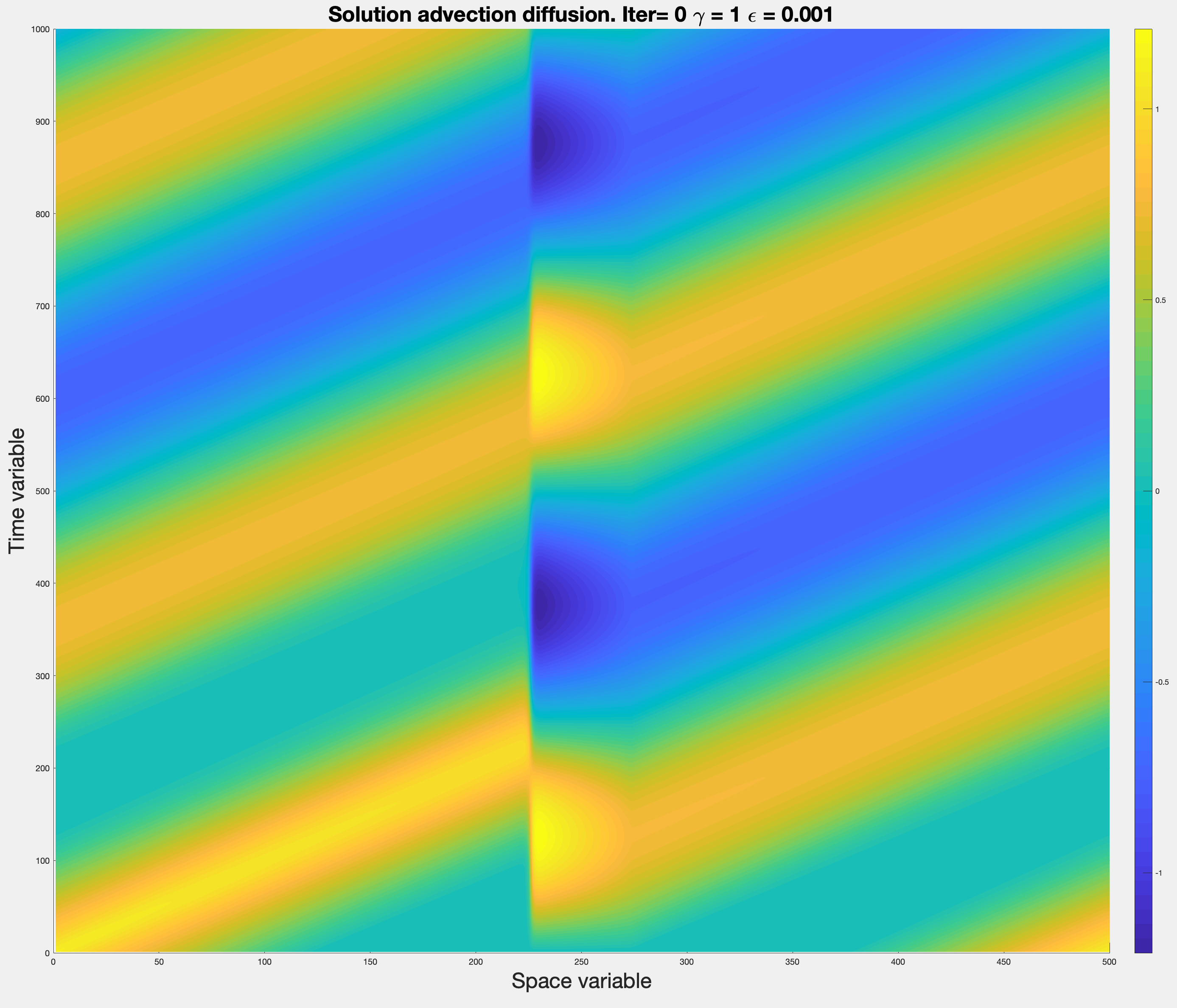} }
\subfigure[\label{fig14} Space derivative of the solution of the advection-convection equation]{\includegraphics[width=5.5cm,height=5.5cm]{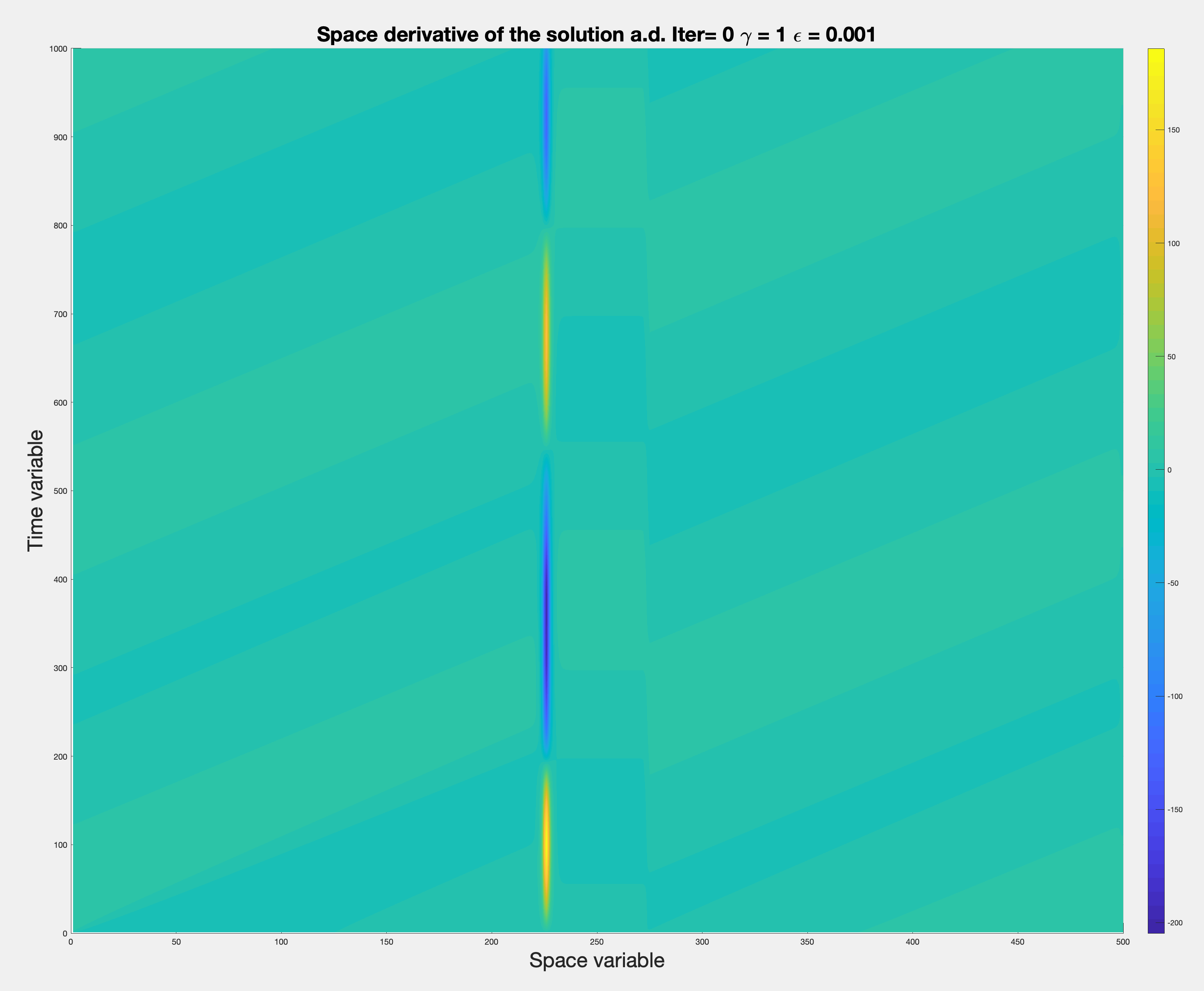} }
\caption{Solution of the advection-convection equation with two penalty terms: one on the boundary between the fluid and the structure and the other inside the structure. But no duality iteration have been used. \label{fig1314} }
\end{figure}

\begin{figure}[!htbp]
\subfigure[\label{fig15} Solution ]{\includegraphics[width=5.5cm,height=5.5cm]{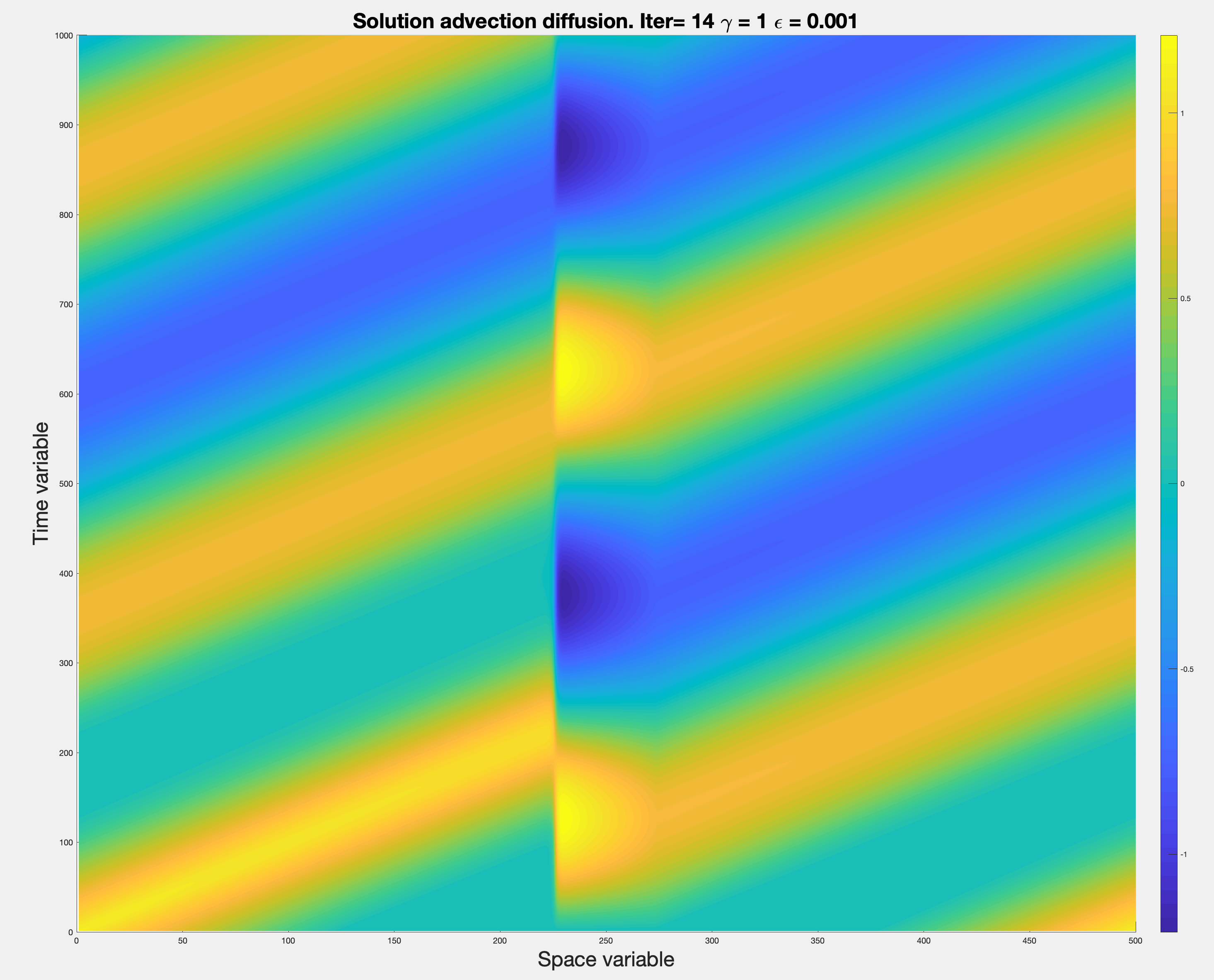} }
\subfigure[\label{fig16} Space derivative of the solution of the advection-convection equation]{\includegraphics[width=5.5cm,height=5.5cm]{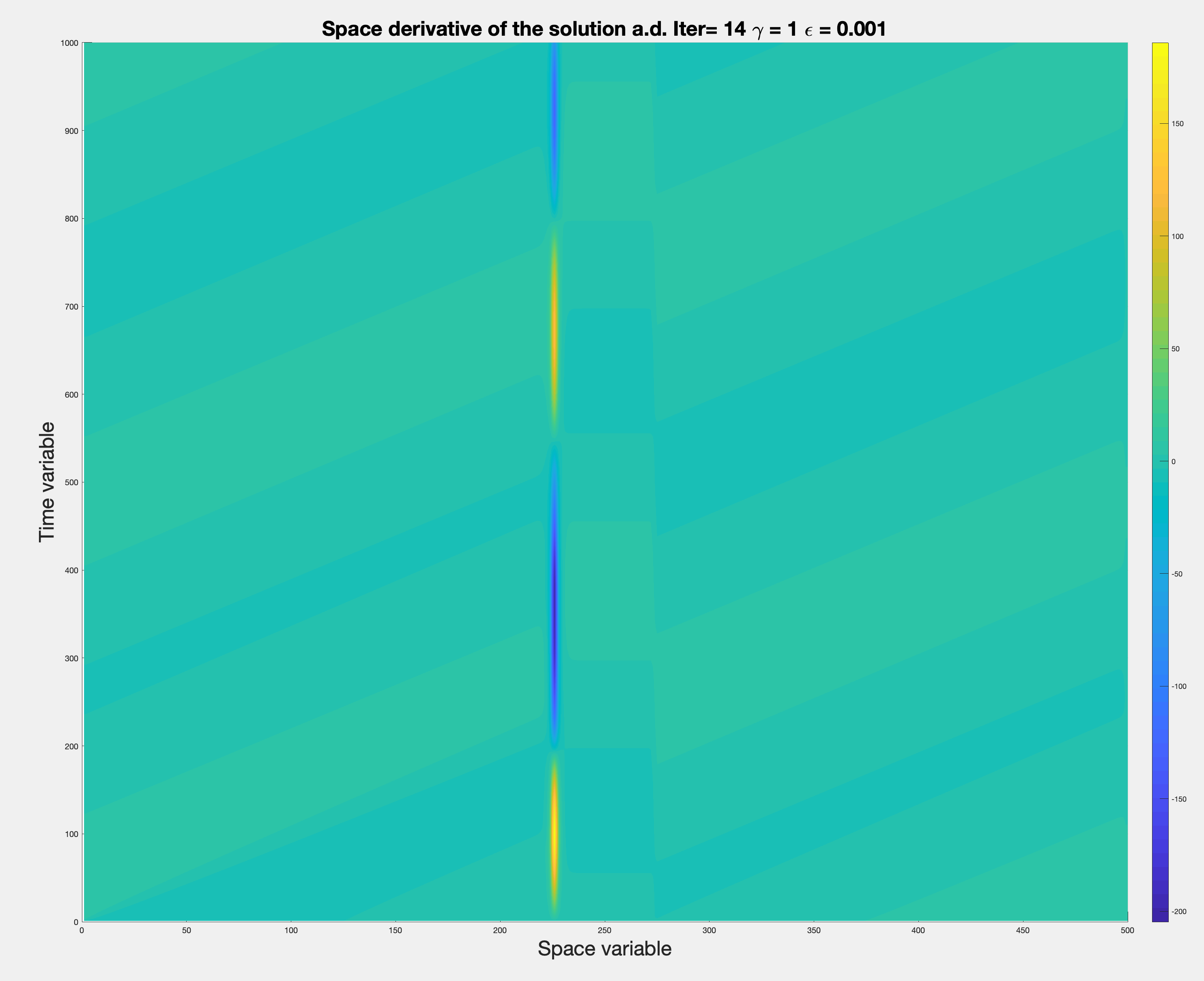} }
\caption{solution of the advection-convection equation with the two penalty terms and the duality algorithm.\label{fig1516} }
\end{figure}

\begin{figure}[!htbp]
\subfigure[\label{fig17} no duality  ]{\includegraphics[width=5.5cm,height=5.5cm]{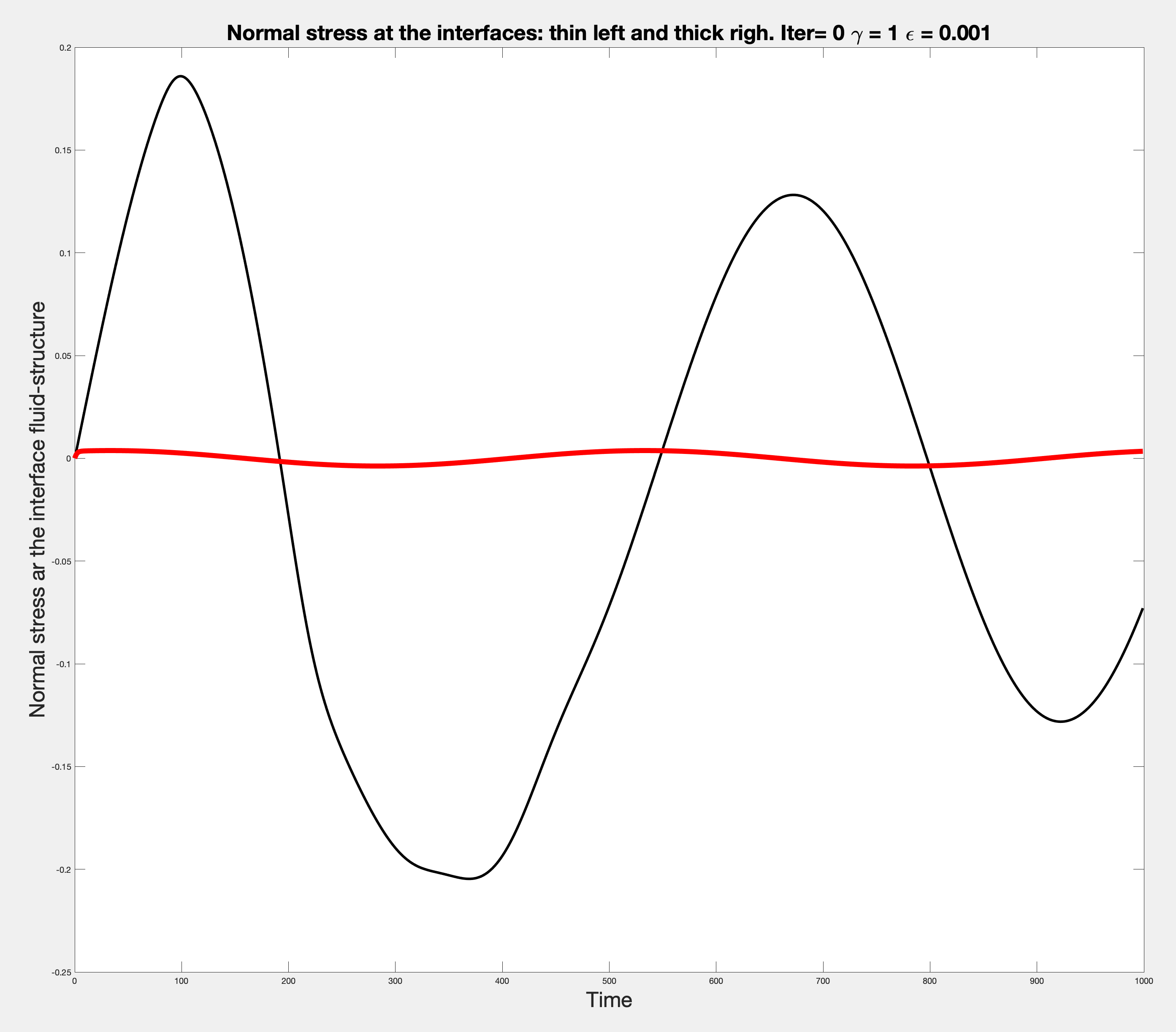} }
\subfigure[\label{fig18} duality ]{\includegraphics[width=5.5cm,height=5.5cm]{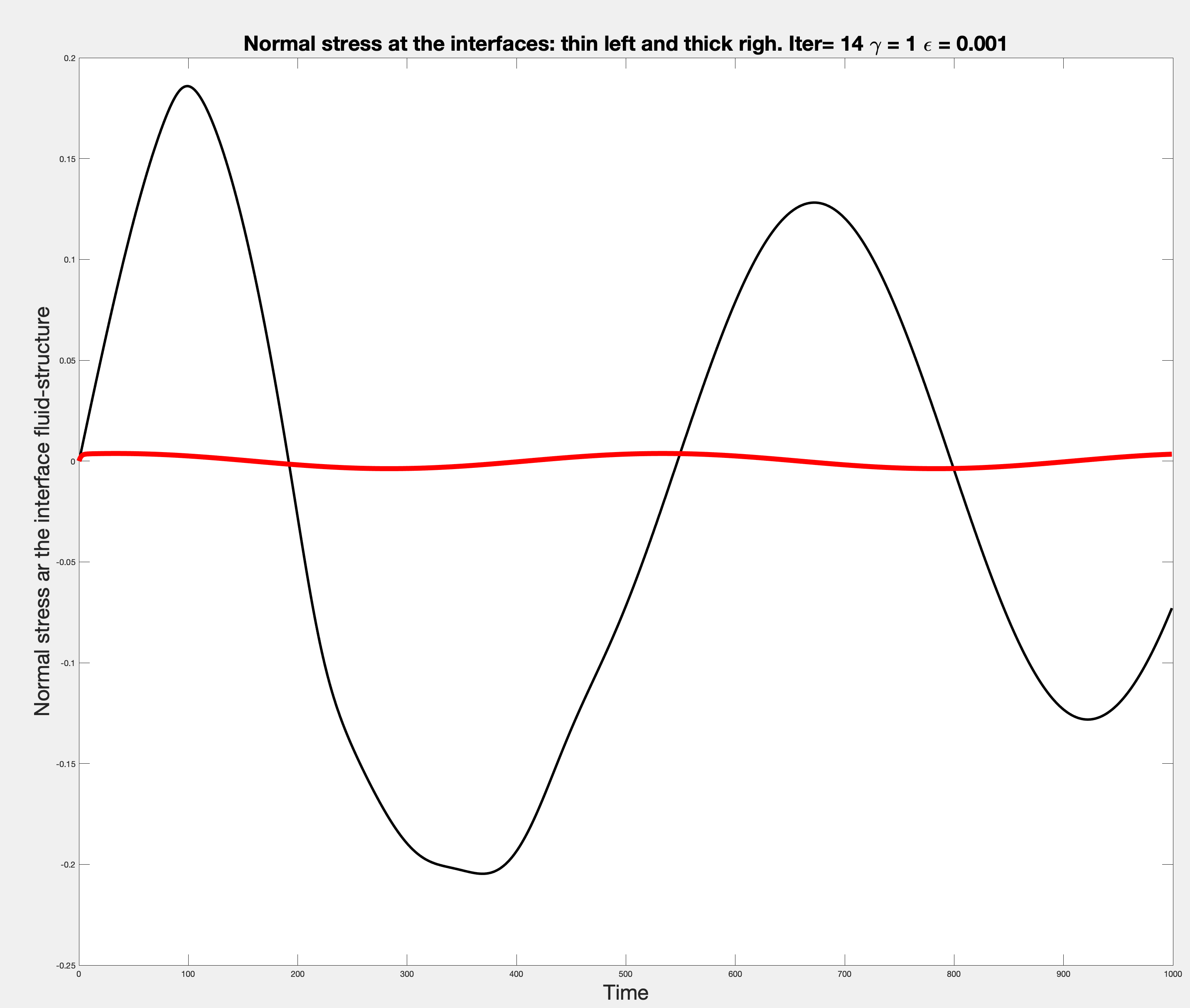} }
\caption{This Figure represents the normal stress at the f-s interfaces; the two penalty terms have been used. The difference with Figure \ref{fig1112bis} (no penalty inside the structure)  can be noticed. \label{fig1718} }
\end{figure}

\subsection{A non linear case (Burgers)}\label{test4} In this subsection we consider the following non linear model (Burgers equation) but the boundary of the structure aren't updated. The scheme used is similar to the one used  is the linear case  (fully implicit for the diffusion and Godounov scheme for the transport) and the data are also the same concerning the dimensions and the initial values. 
\beq\label{eq500bis} \left \{ \begin{array}{l}x\in]0,a[\cup]b,L[,\;t\geq 0:\;\fracj{\partial u}{\partial t}+\fracj{\partial}{\partial x}(\fracj{u^2}{2})-\nu\fracj{\partial^2 u}{\partial x^2}=0;\\\\
t\geq0\;u(0,t)=u(L,t),\;\;x\in ]0,a[\cup]b,L[:\;u(x,0)=u_0(x)\\\\t\geq 0,\;x\in]a,b[:\;u(x,t)=d_0(t) \;(\hbox{ sine function}).\end{array}\right.\eeq
\subsubsection{The numerical results} First of all let us point out that the results should be compared with the quasi-exact solution obtained with a very small value of the penalty parameter $\varepsilon$ but which requires some precautions in the adjustment of the numerical tests (ill conditioning). The  quasi-exact solution and its space derivatives are plotted on Figures \ref{fig215}-\ref{fig216}-\ref{fig217}-\ref{fig218}.

The first results have been plotted on Figures \ref{fig201}-\ref{fig202}-\ref{fig203} for the case where there is no penalty term inside the structure and no duality. But there is a moderate penalty term at the f-s interfaces. If we compare to the quasi-exact solution plotted on Figure \ref{fig215} one can see a big difference. The same is true (and even amplified for the space derivatives (see Figures \ref{fig202}-\ref{fig203} compared to Figures \ref{fig215}-\ref{fig217}).

On the Figures \ref{fig204}-\ref{fig205}-\ref{fig206}-\ref{fig207} we have plotted the solution of Burgers equation obtained with the penalty-duality algorithm but without penalty term ($L^2$-norm between the velocity obtained by the scheme and the one prescribed in the structure). The results compared with the quasi-exact solution are much better but not satisfying.

The penalty term inside the structure without duality has been added in the computations shown on Figures \ref{fig208}-\ref{fig209}-\ref{fig210}. The results are closer to the quasi-exact one of Figures \ref{fig215}-\ref{fig216}-\ref{fig217}. The improvement with adding the duality algorithm is small for small values of the penalty parameter as one can see on Figures \ref{fig211}-\ref{fig212}-\ref{fig213}-\ref{fig214} (compared to the quasi-exact solution), but exists.
\begin{figure}[!htbp]
\subfigure[Solution (Burgers) \label{fig201}]{\includegraphics[width=5.5cm,height=5.5cm]{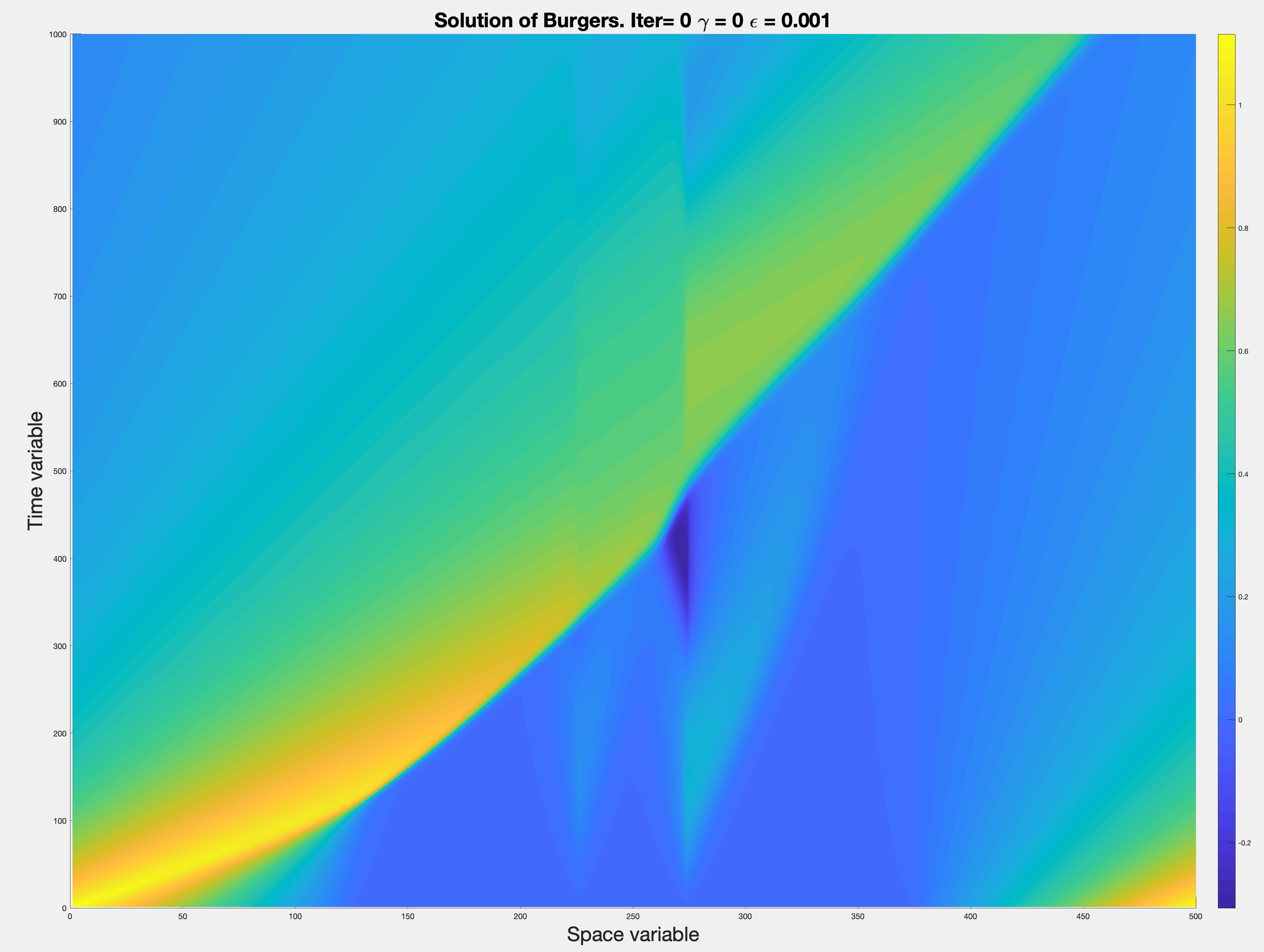}}
\subfigure[Derivative versus $x$ of the solution (Burgers)\label{fig202}]{\includegraphics[width=5.5cm,height=5.5cm]{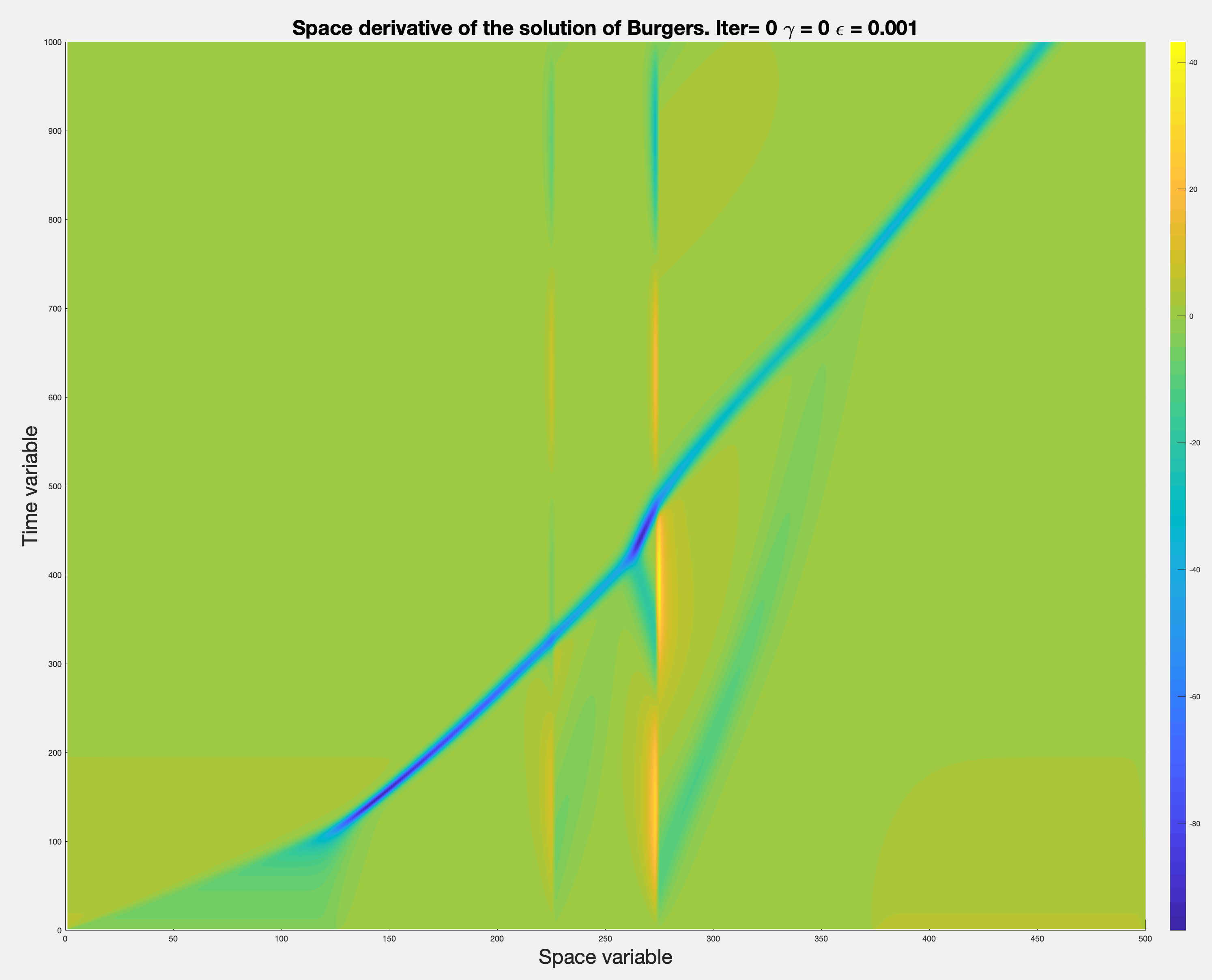}}
\caption{Solution (Burgers) without duality and no penalty term inside the structure (same data as for the linear case) but $r=.1$ \label{fig201202}}
\end{figure}

%\begin{figure}[htbp] %  figure placement: here, top, bottom, or page
%   \centering
%   \includegraphics[width=11cm,height=9cm]{burgers1.png} 
%   \caption{Solution (Burgers) without duality and no penalty term inside the structure (same data as for the linear case) but $r=.1$}
%   \label{fig201}
%\end{figure}
%
%\begin{figure}[htbp] %  figure placement: here, top, bottom, or page
%   \centering
%   \includegraphics[width=11cm,height=9cm]{burgers2.png} 
%   \caption{Derivative versus $x$ of the solution (Burgers) without duality and no penalty term inside the structure.}
%   \label{fig202}
%\end{figure}
%
\begin{figure}[htbp] %  figure placement: here, top, bottom, or page
   \centering
   \includegraphics[width=11cm,height=8.5cm]{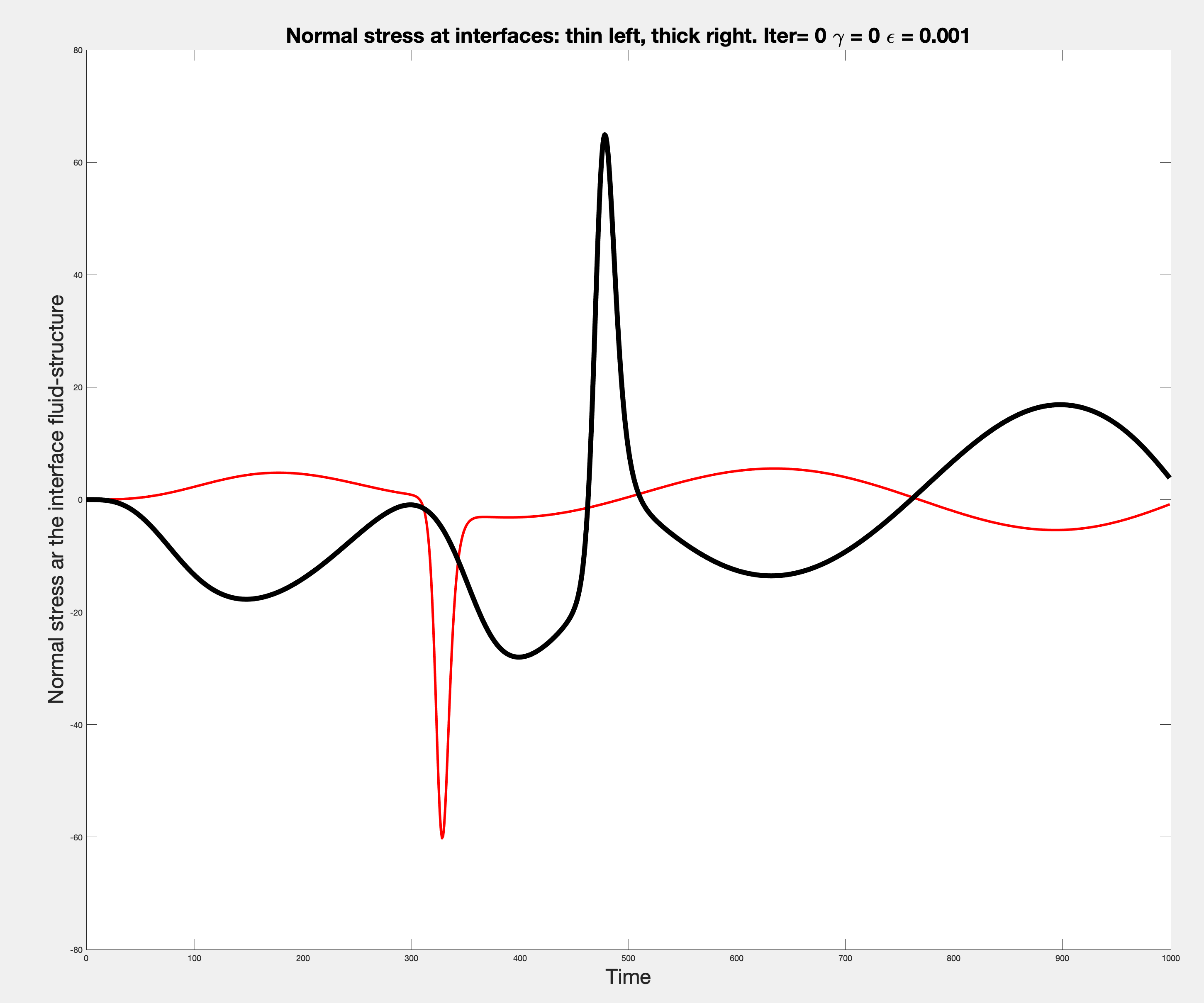} 
   \caption{Normal stress at the f-s interface for the Burgers solution without duality and no penalty term inside the structure (same data as for the linear case)}
   \label{fig203}
\end{figure}
\begin{figure}[!htbp]
\subfigure[Solution \label{fig204}]{\includegraphics[width=5.5cm,height=5.5cm]{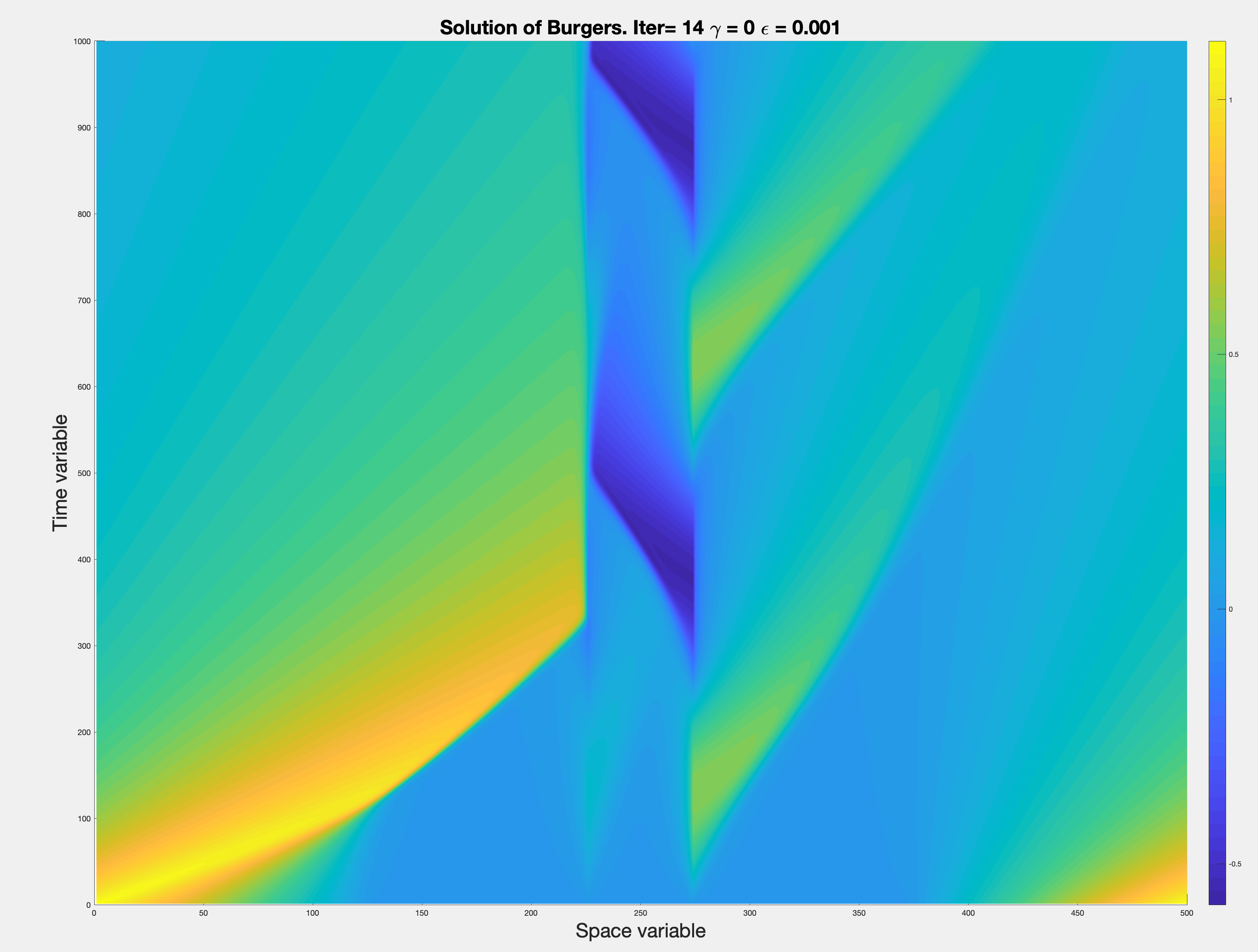}}
\subfigure[Derivative versus $x$ \label{fig205}]{\includegraphics[width=5.5cm,height=5.5cm]{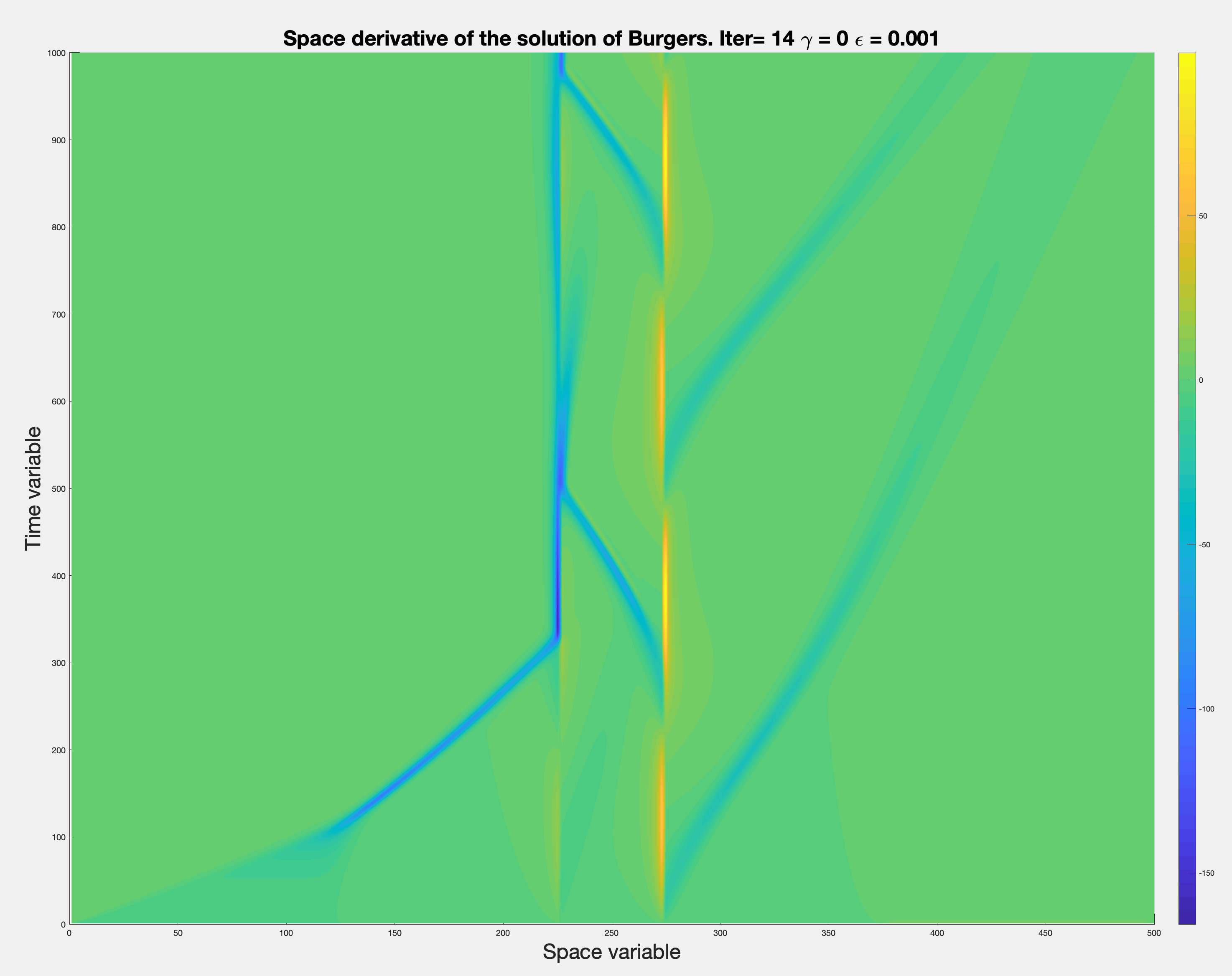}}
\caption{The solution with duality (Burgers) but no penalty term at the f-s interface. The data are still the same as in the linear case. One can see a meaningful difference with the case without duality. \label{fig204205}}
\end{figure}

%\begin{figure}[htbp] %  figure placement: here, top, bottom, or page
%   \centering
%   \includegraphics[width=11cm,height=8.5cm]{burgers4.png} 
%   \caption{The solution with duality (Burgers) but no penalty term at the f-s interface. The data are still the same as in the linear case. One can see a meaningful difference with the case without duality.}
%   \label{fig204}
%\end{figure}
%
%\begin{figure}[htbp] %  figure placement: here, top, bottom, or page
%   \centering
%   \includegraphics[width=11cm,height=9cm]{burgers5.png} 
%   \caption{Derivative versus $x$ for the Burgers solution with duality but no penalty term inside the structure.}
%   \label{fig205}
%\end{figure}
%
%
\begin{figure}[!htbp]
\subfigure[Normal stress at the f-s interface  \label{fig206}]{\includegraphics[width=5.5cm,height=5.5cm]{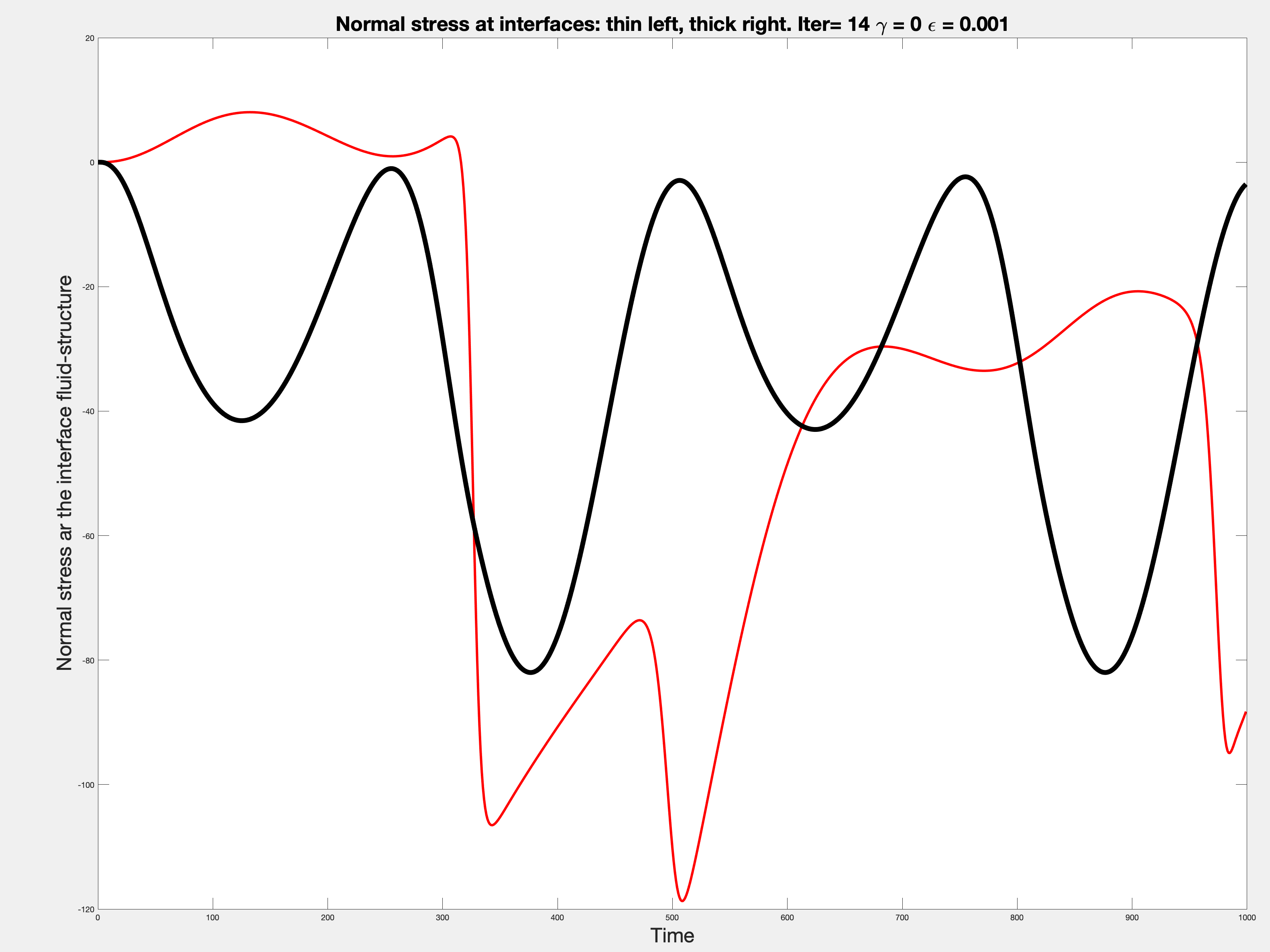}}
\subfigure[The Lagrange multipliers at the two f-s interfaces \label{fig207}]{\includegraphics[width=5.5cm,height=5.5cm]{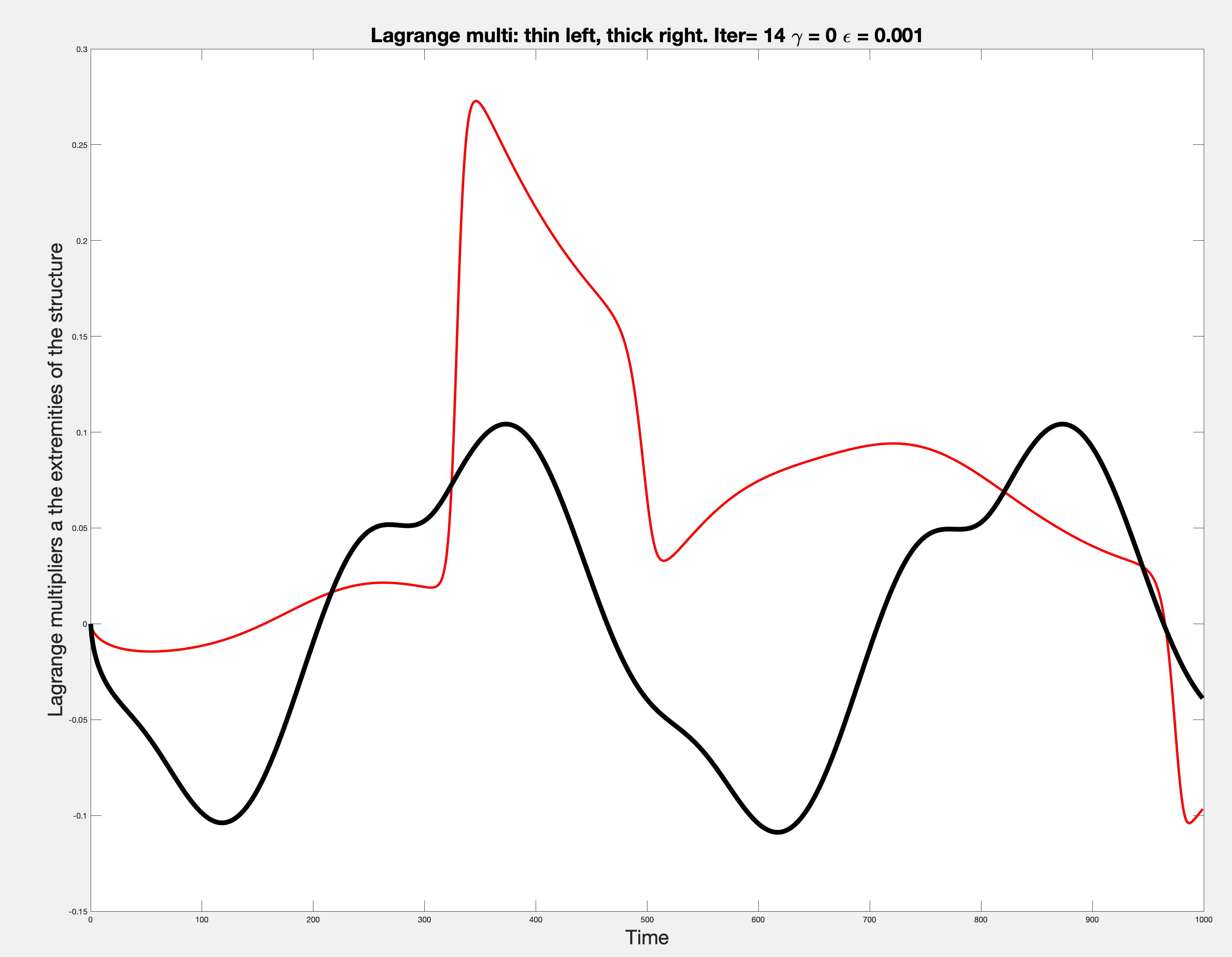}}
\caption{ Burgers solution obtained with duality but no penalty term inside the structure. \label{fig206207}}
\end{figure}
%
%\begin{figure}[htbp] %  figure placement: here, top, bottom, or page
%   \centering
%   \includegraphics[width=11cm,height=9cm]{burgers6.png} 
%   \caption{Normal stress at the f-s interface of Burgers solution obtained with duality but no penalty term inside the structure.}
%   \label{fig206}
%\end{figure}
%
%
%\begin{figure}[htbp] %  figure placement: here, top, bottom, or page
%   \centering
%   \includegraphics[width=11cm,height=9cm]{burgers7.png} 
%   \caption{The Lagrange multipliers at the two f-s interfaces (with duality but no penalty term inside the structure).}
%   \label{fig207}
%\end{figure}
%
%
\begin{figure}[!htbp]
\subfigure[solution \label{fig208}]{\includegraphics[width=5.5cm,height=5.5cm]{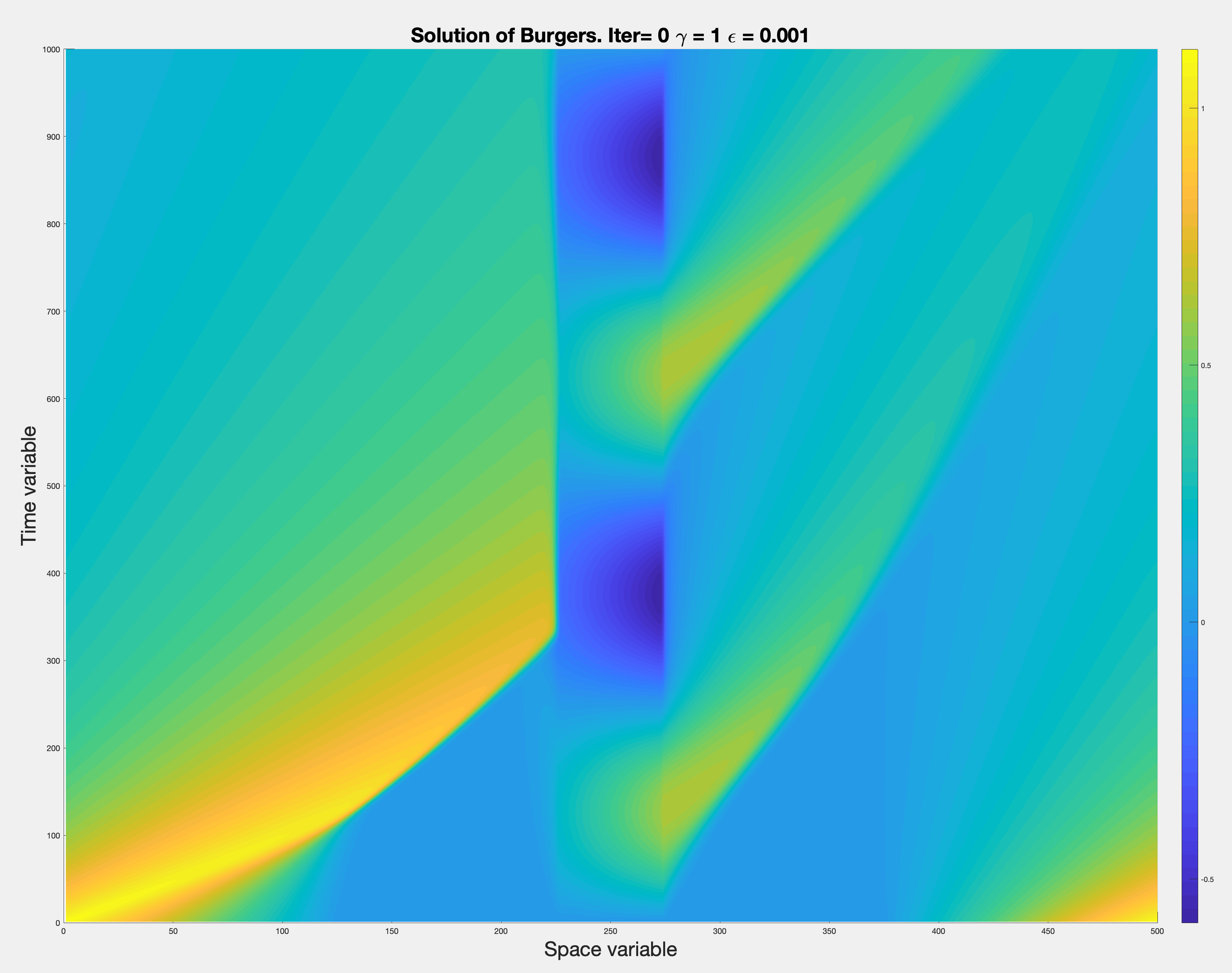}}
\subfigure[Derivative versus $x$ of Burgers solution \label{fig209}]{\includegraphics[width=5.5cm,height=5.5cm]{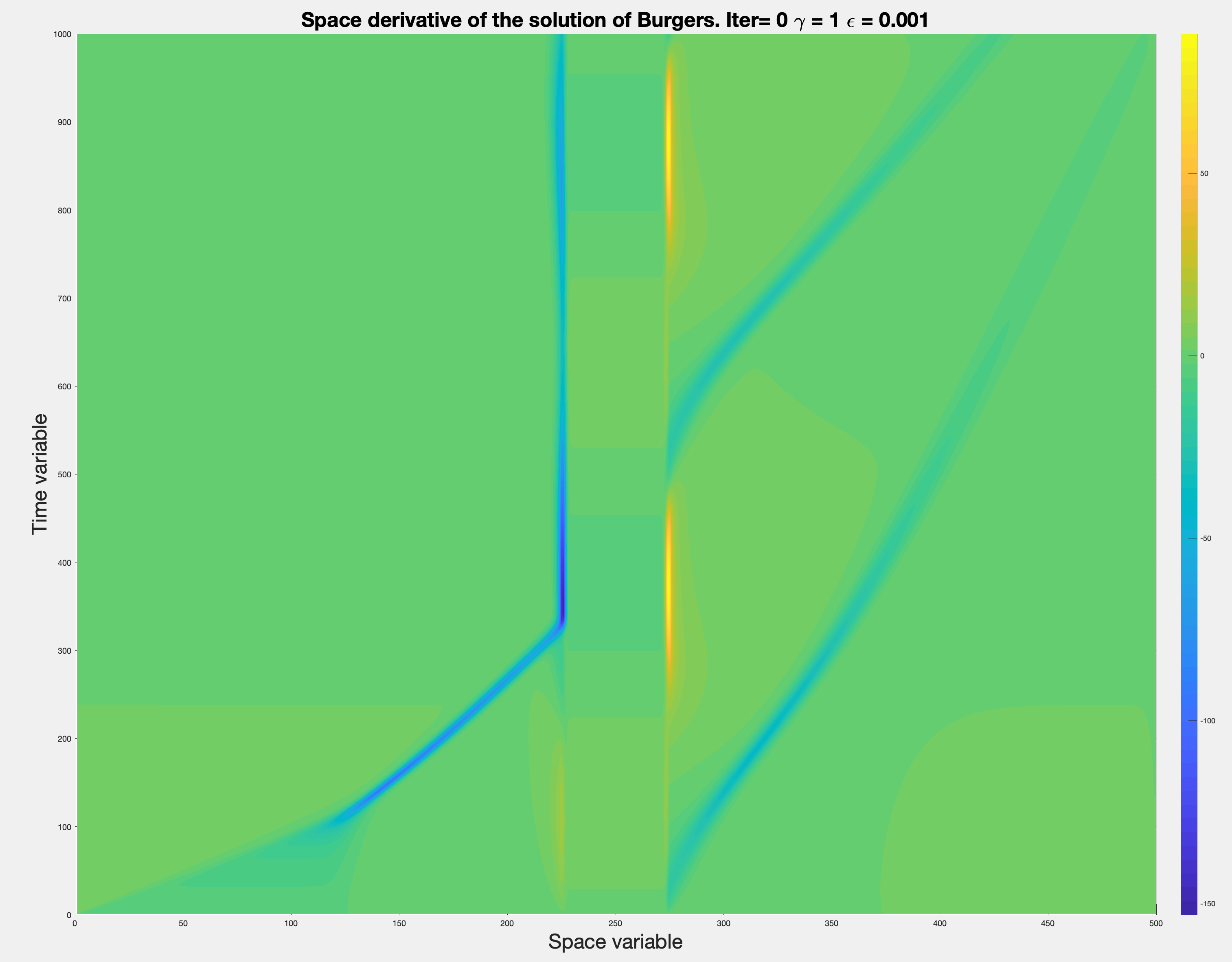}}
\caption{Burgers solution obtained without duality but a penalty term inside the structure ($L^2$). \label{fig208209}}
\end{figure}

%\begin{figure}[htbp] %  figure placement: here, top, bottom, or page
%   \centering
%   \includegraphics[width=11cm,height=9cm]{burgers8.png} 
%   \caption{Burgers solution obtained without duality but a penalty term inside the structure ($L^2$).}
%   \label{fig208}
%\end{figure}
%
%
%\begin{figure}[htbp] %  figure placement: here, top, bottom, or page
%   \centering
%   \includegraphics[width=11cm,height=9cm]{burgers9.png} 
%   \caption{Derivative versus $x$ of Burgers solution obtained without duality but a penalty term inside the structure.}
%   \label{fig209}
%\end{figure}
%
%
\begin{figure}[htbp] %  figure placement: here, top, bottom, or page
   \centering
   \includegraphics[width=11cm,height=9cm]{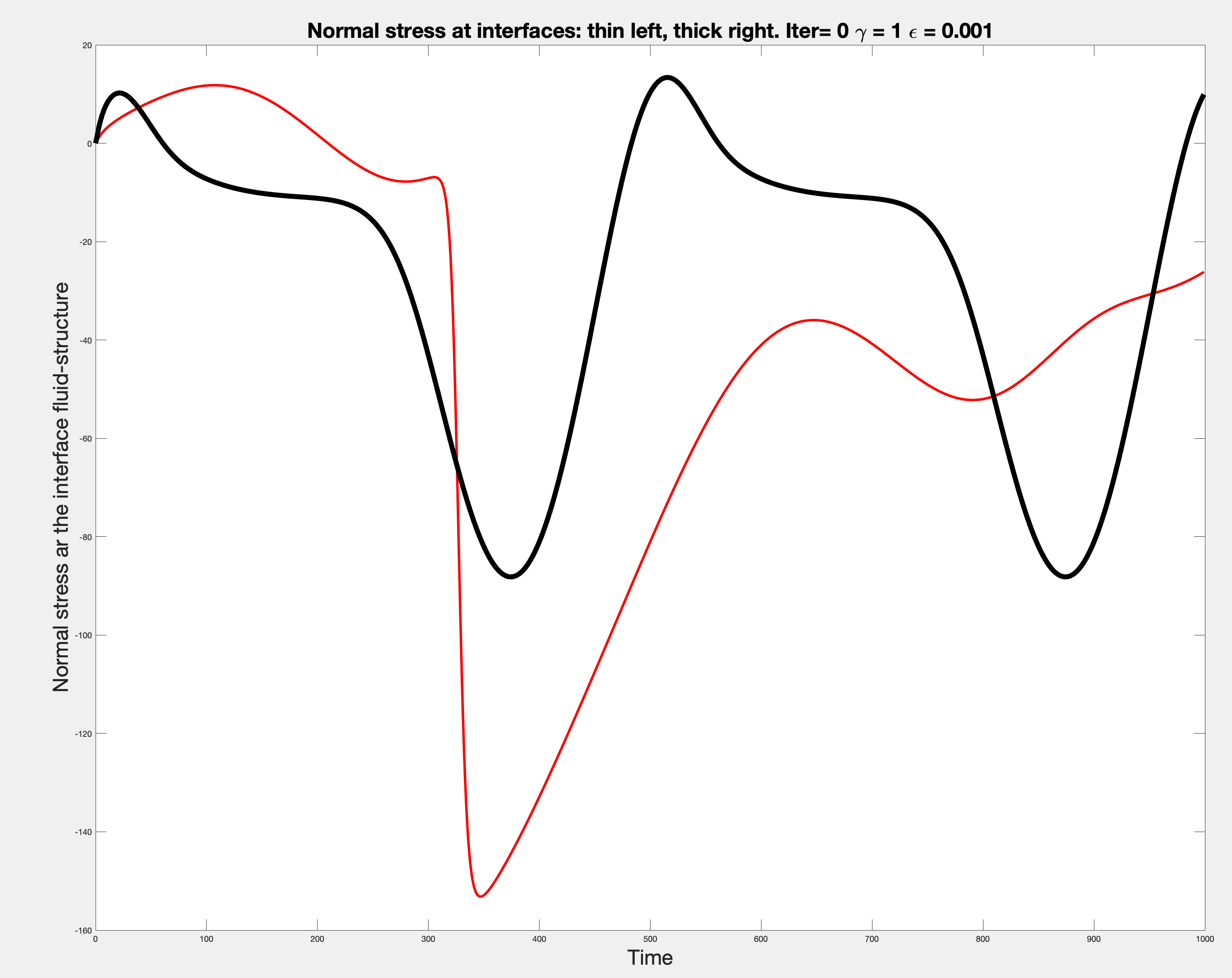} 
   \caption{Normal stress at the f-s interface (no duality but a penalty term inside the structure.}
   \label{fig210}
\end{figure}
\begin{figure}[!htbp]
\subfigure[solution \label{fig211}]{\includegraphics[width=5.5cm,height=5.5cm]{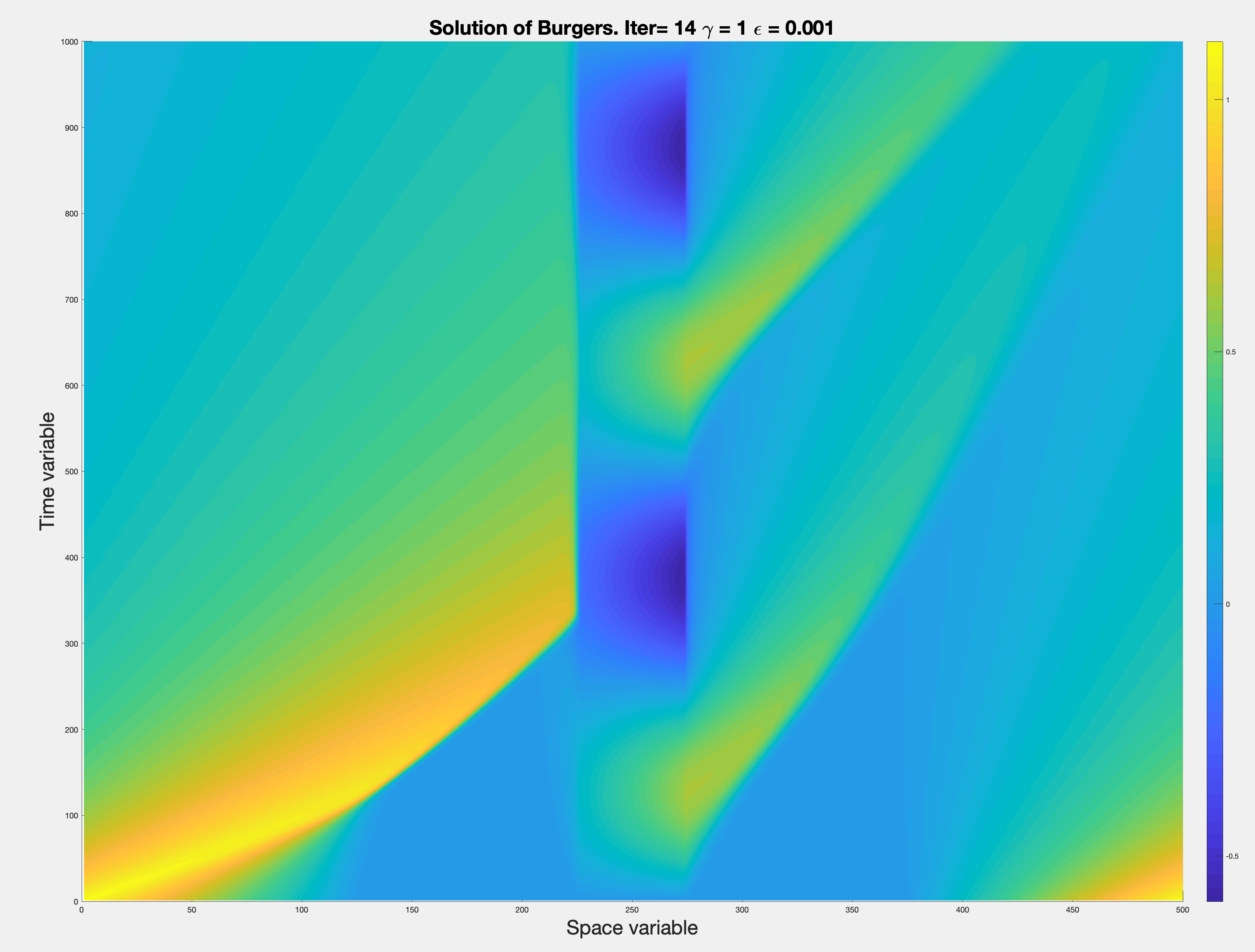}}
\subfigure[Derivative versus $x$ of Burgers solution \label{fig212}]{\includegraphics[width=5.5cm,height=5.5cm]{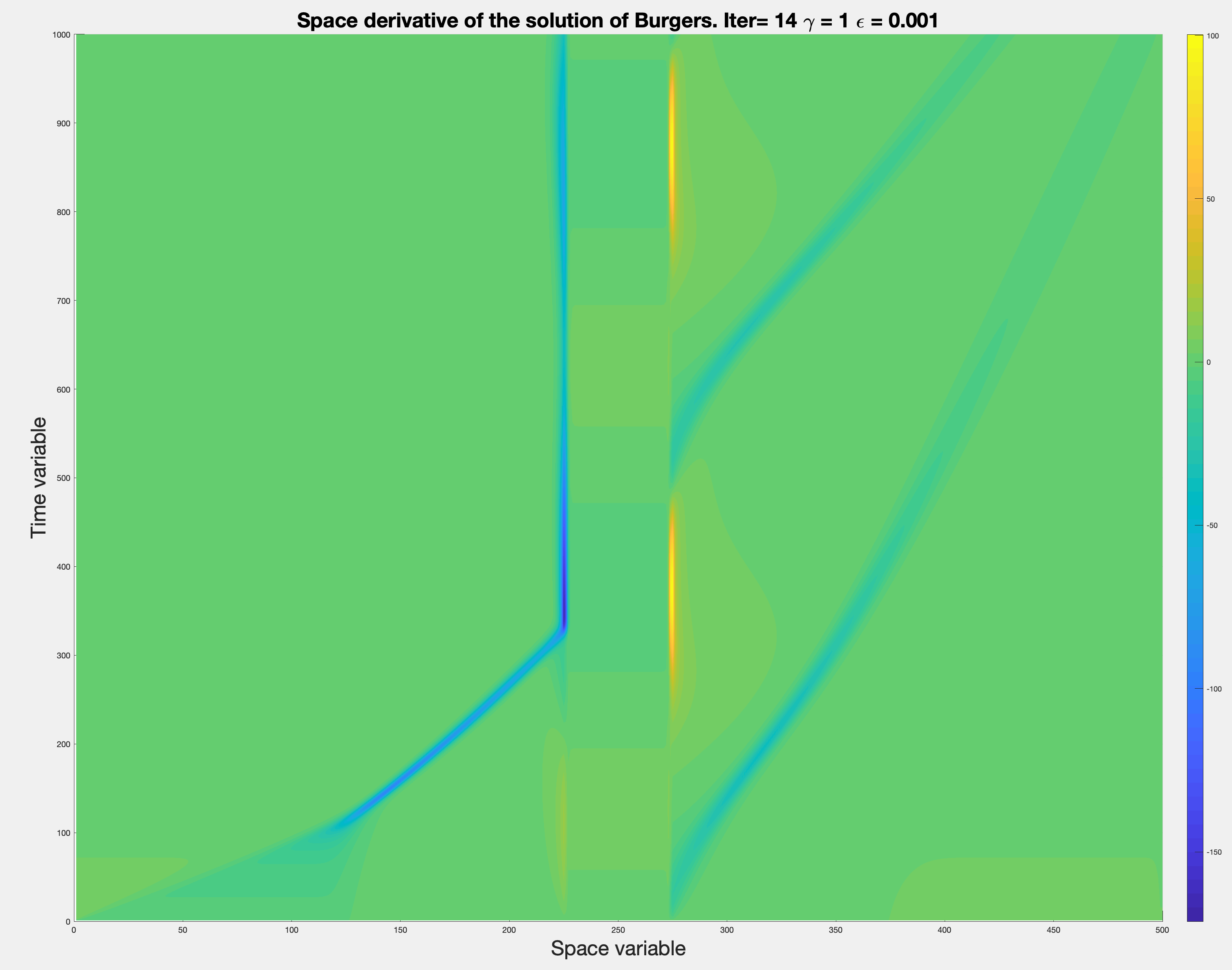}}
\caption{Burgers solution obtained with duality and a penalty term inside the structure. \label{fig211212}}
\end{figure}
%
%\begin{figure}[htbp] %  figure placement: here, top, bottom, or page
%   \centering
%   \includegraphics[width=11cm,height=9cm]{burgers11.png} 
%   \caption{Burgers solution obtained with duality and a penalty term inside the structure.}
%   \label{fig211}
%\end{figure}
%
%\begin{figure}[htbp] %  figure placement: here, top, bottom, or page
%   \centering
%   \includegraphics[width=11cm,height=9cm]{burgers12.png} 
%   \caption{Derivative versus $x$ of Burgers solution obtained with duality and a penalty term inside the structure.}
%   \label{fig212}
%\end{figure}
%
\begin{figure}[!htbp]
\subfigure[Normal stress at the f-s interface \label{fig213}]{\includegraphics[width=5.5cm,height=5.5cm]{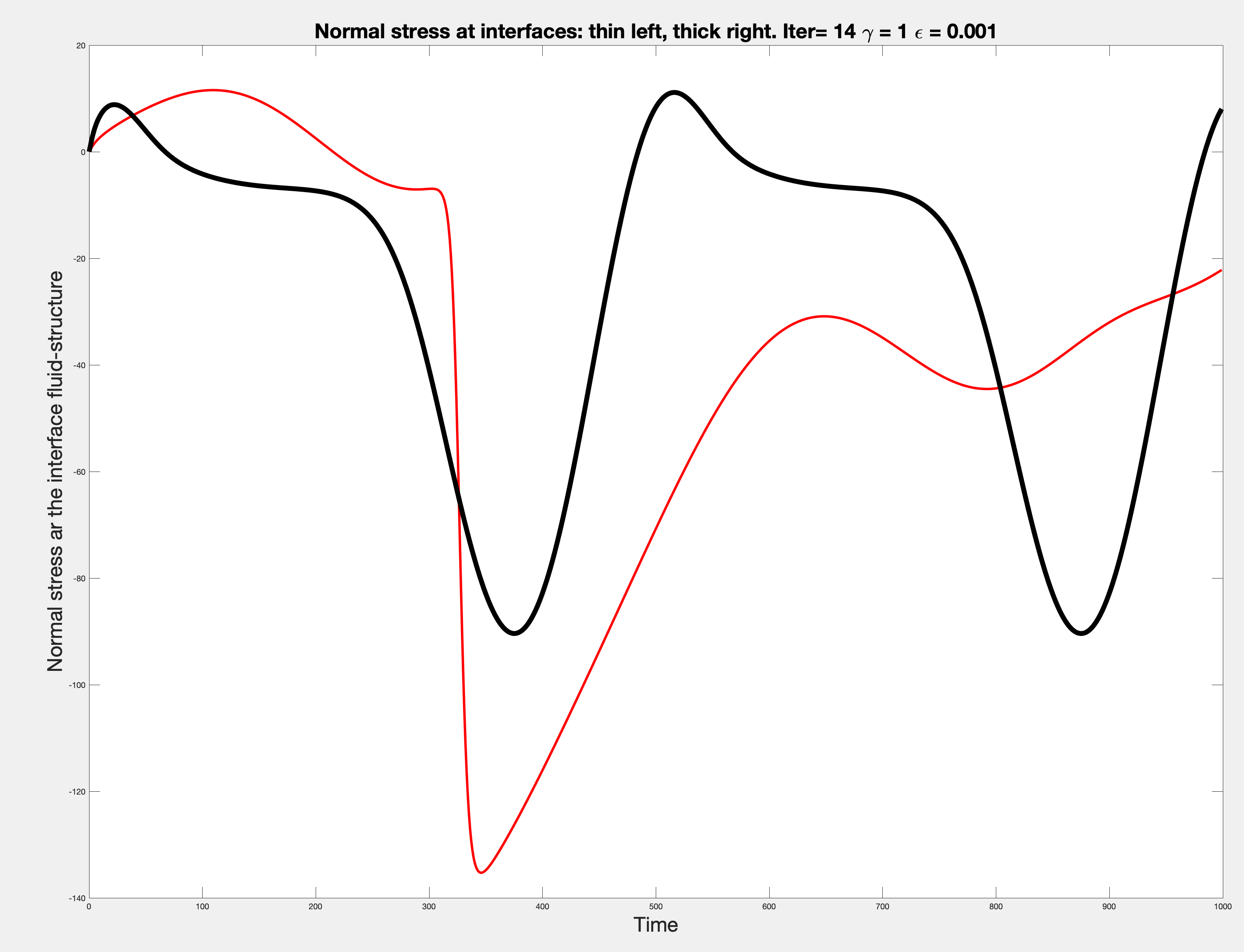}}
\subfigure[Lagrange multipliers at the f-s interfaces  \label{fig214}]{\includegraphics[width=5.5cm,height=5.5cm]{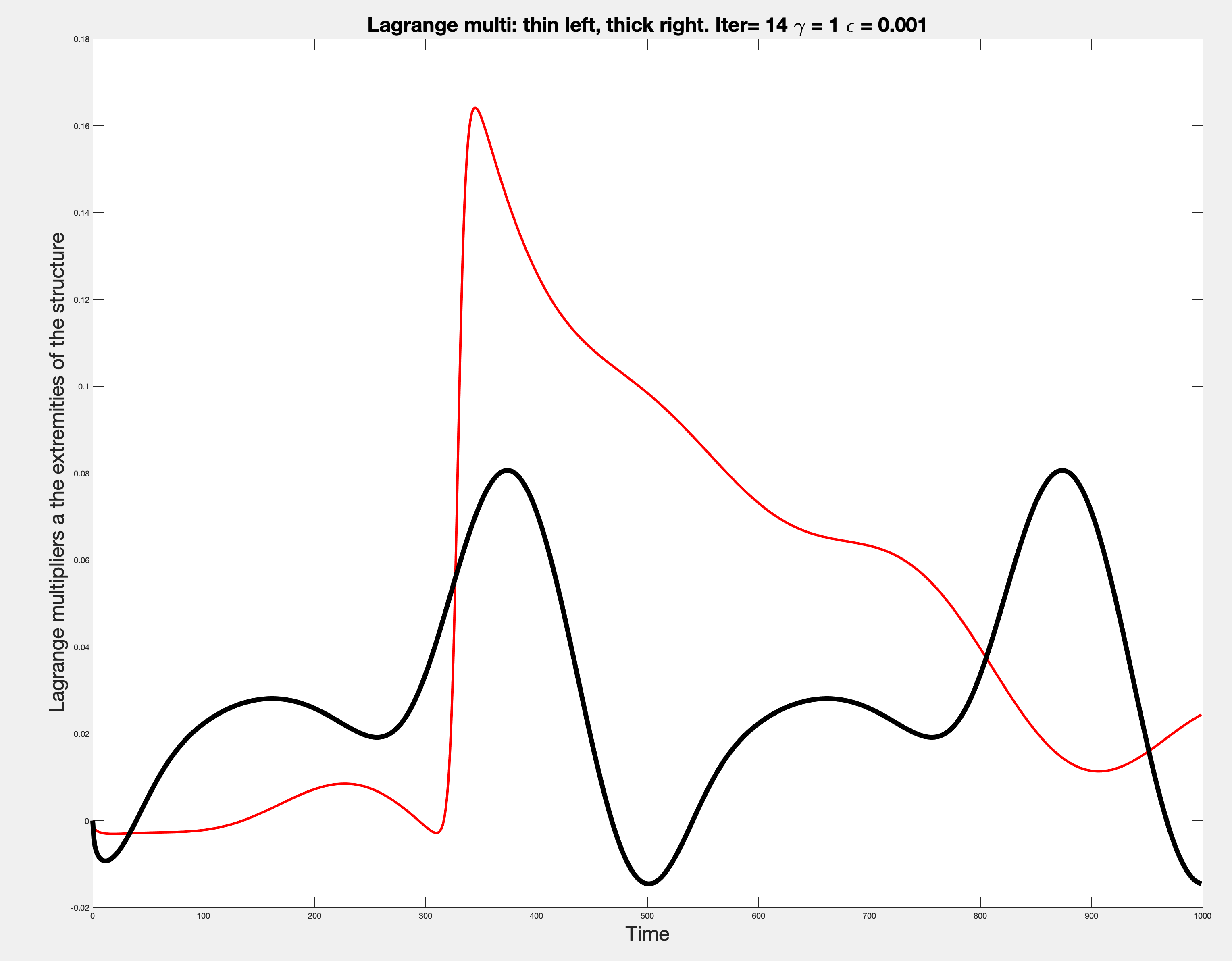}}
\caption{Normal stress and Lagrange multipliers at the f-s interface with duality and a penalty term inside the structure. \label{fig213214}}
\end{figure}
%
%\begin{figure}[htbp] %  figure placement: here, top, bottom, or page
%   \centering
%   \includegraphics[width=11cm,height=9cm]{burgers13.png} 
%   \caption{Normal stress at the f-s interface with duality and a penalty term inside the structure.}
%   \label{fig213}
%\end{figure}
%
%\begin{figure}[htbp] %  figure placement: here, top, bottom, or page
%   \centering
%   \includegraphics[width=11cm,height=9cm]{burgers14.png} 
%   \caption{Lagrange multipliers at the f-s interfaces  (duality and  penalty term inside the structure).}
%   \label{fig214}
%\end{figure}
%
%
\begin{figure}[!htbp]
\subfigure[Reference solution \label{fig215}]{\includegraphics[width=5.5cm,height=5.5cm]{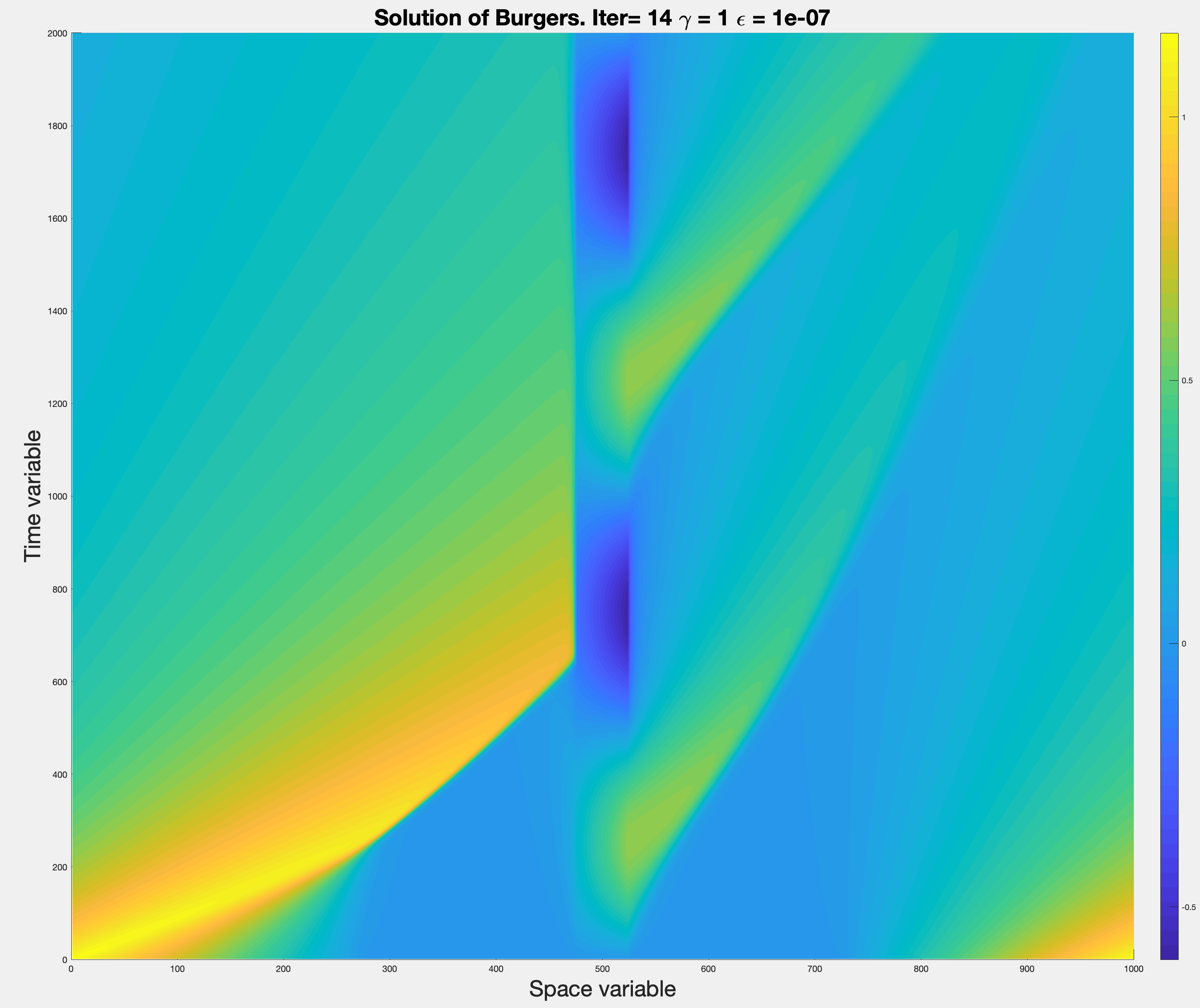} }
\subfigure[Derivative versus $x$ of the reference solution \label{fig216}]{\includegraphics[width=5.5cm,height=5.5cm]{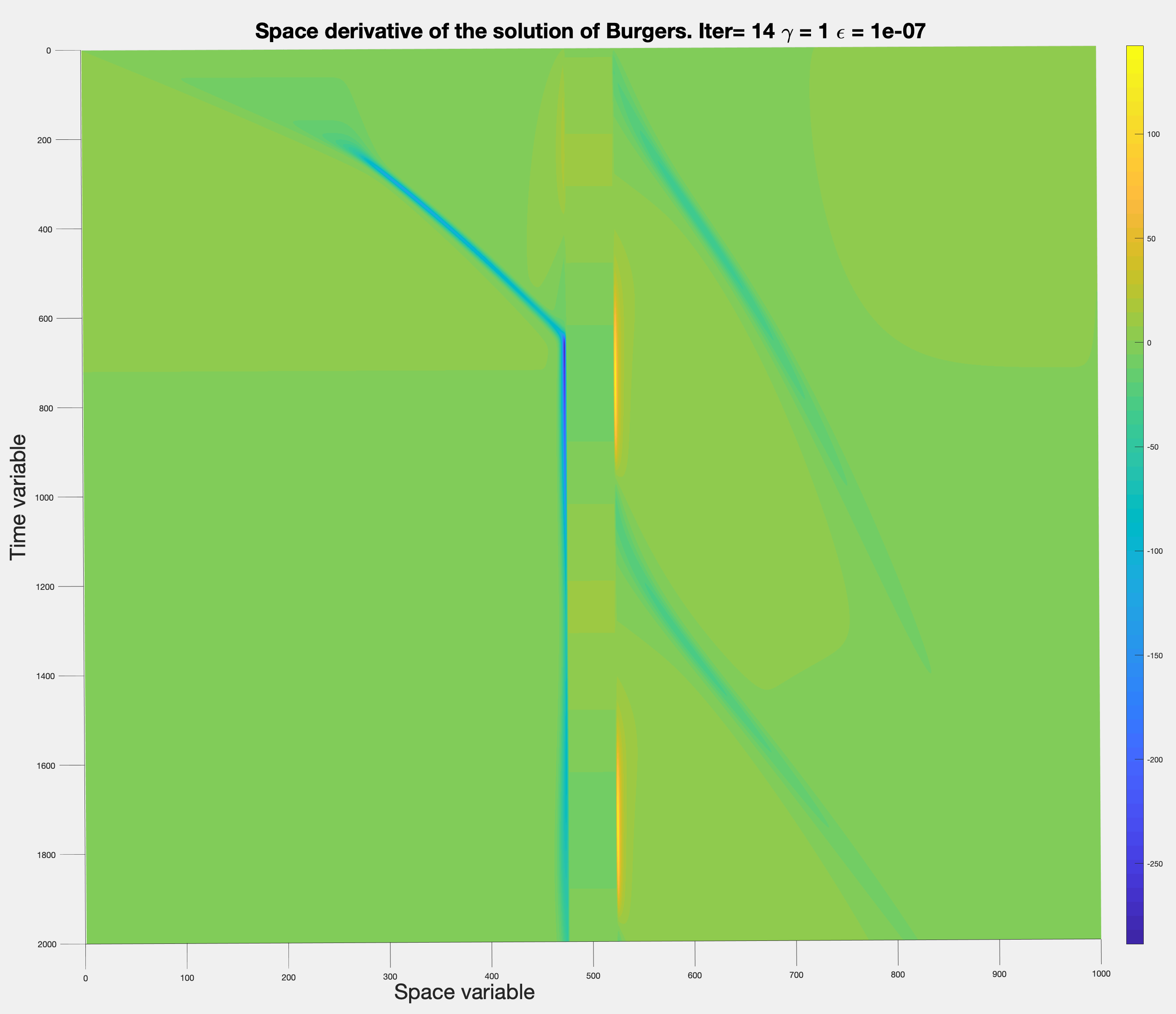} }
\caption{Reference solution of Burgers equation (duality + penalty term inside the structure + refined mesh + time step) \label{fig215216}}
\end{figure}
%
%\begin{figure}[htbp] %  figure placement: here, top, bottom, or page
%   \centering
%   \includegraphics[width=11cm,height=9cm]{burgers15.png} 
%   \caption{Reference solution of Burgers equation (duality + penalty term inside the structure + refined mesh + time step).}
%   \label{fig215}
%\end{figure}
%
%\begin{figure}[htbp] %  figure placement: here, top, bottom, or page
%   \centering
%   \includegraphics[width=11cm,height=8.5cm]{burgers16.png} 
%   \caption{Derivative versus $x$ of the reference solution (Figure \ref{fig215}) duality and a penalty term inside the structure.}
%   \label{fig216}
%\end{figure}
%
\begin{figure}[!htbp]
\subfigure[Normal stress at the f-s interfaces \label{fig217}]{\includegraphics[width=5.5cm,height=5.5cm]{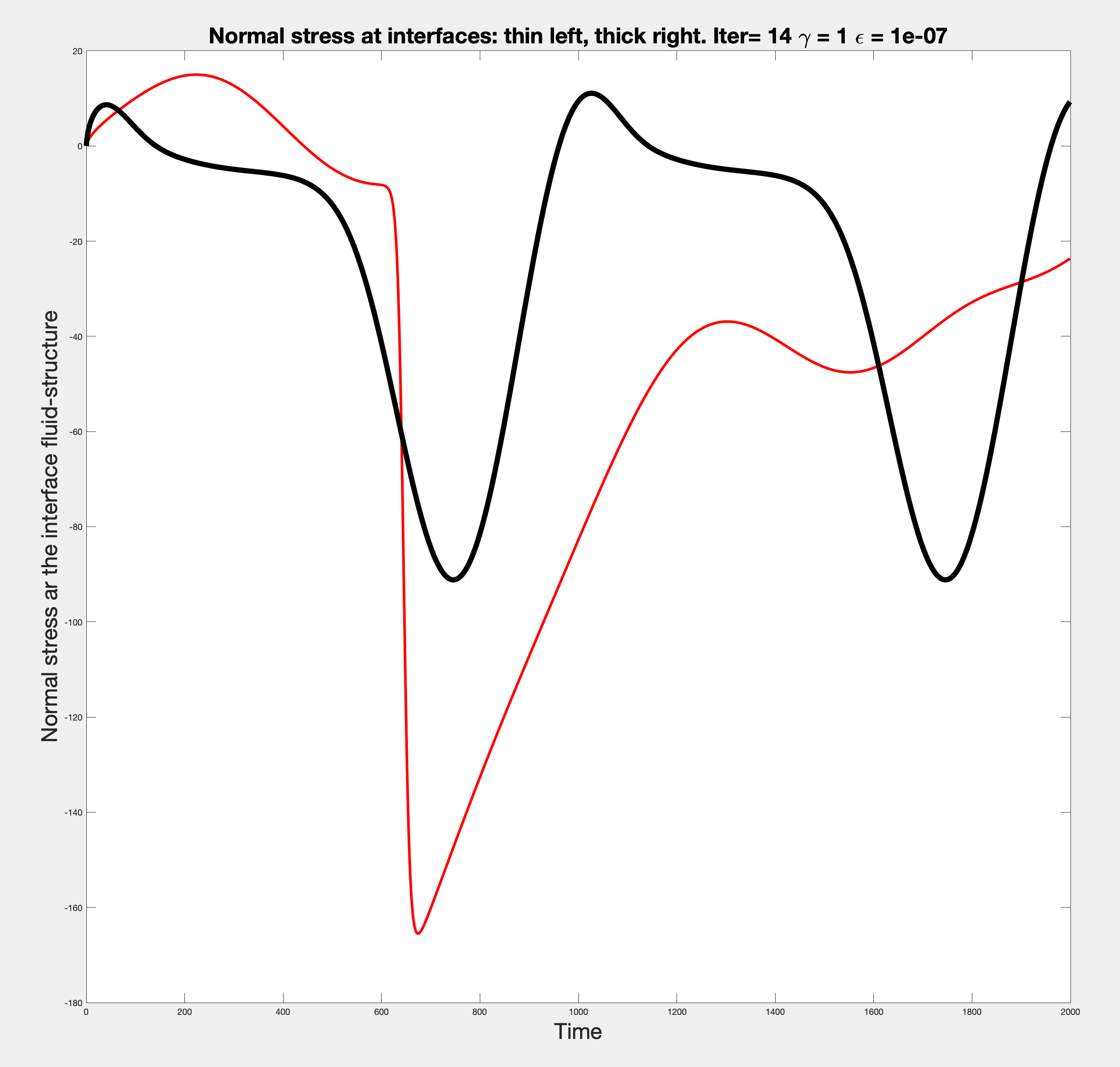} }
\subfigure[Lagrange multipliers at the f-s interfaces, (the sign of the Lagrange multiplier is opposite to the one of the normal stress in this case)\label{fig218}]{\includegraphics[width=5.5cm,height=5.5cm]{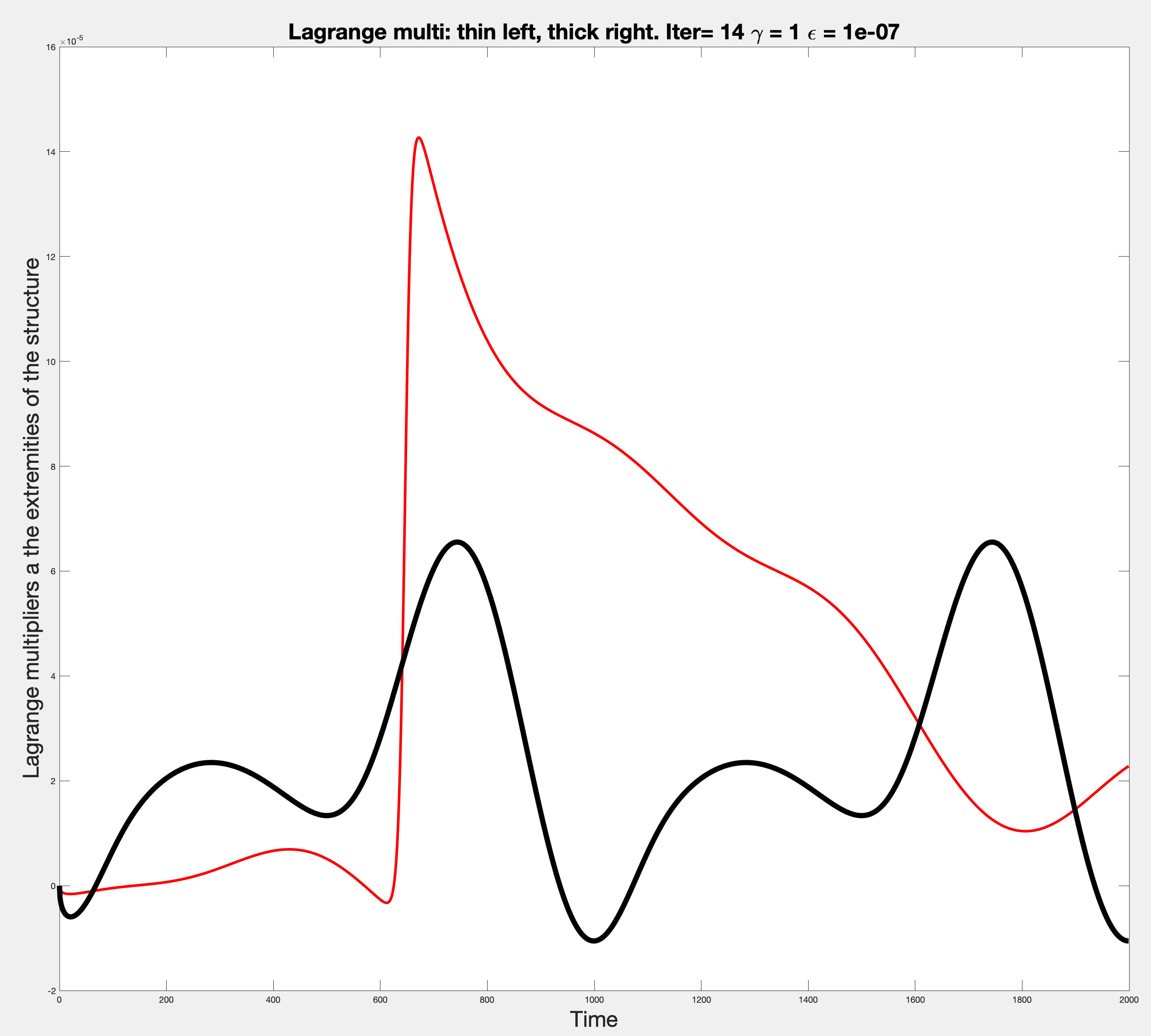} }
   \caption{Reference solution (Figure \ref{fig215}) obtained with duality and a penalty term inside the structure. \label{fig217218}}
 \end{figure}
%\begin{figure}[htbp] %  figure placement: here, top, bottom, or page
%   \centering
%   \includegraphics[width=11cm,height=8.5cm]{burgers17.png} 
%   \caption{Normal stress at the f-s interfaces  of the reference solution (Figure \ref{fig215}) obtained with duality and a penalty term inside the structure.}
%   \label{fig217}
%\end{figure}
%
%\begin{figure}[htbp] %  figure placement: here, top, bottom, or page
%   \centering
%   \includegraphics[width=11cm,height=8.5cm]{burgers17.png} 
%   \caption{Lagrange multipliers at the f-s interfaces  of the reference solution (Figure \ref{fig215}) obtained with duality and a penalty term inside the structure.}
%   \label{fig218}
%\end{figure}

%============
\section{Conclusion} In this paper we have first studied several penalty methods in a simple fluid-structure model assuming the case of a small reduce frequency in order to partially decouple the inertia terms between the fluid and the structure. We discussed the possibilities of several penalty terms inside the structure ($L^2$ or/and $H^1$ norms but also by penalizing the continuity of the velocities at the boundary between the fluid and the structure. Our goal has been to focus on the mathematical difficulties which can occur in the numerical schemes due to the ill conditioning of the penalty models.We compare the methods from the theoretical point of view and on a very simple 1D model. We also introduced in this framework the Bertsekas penalty-duality algorithm for the ensuring the velocity continuity at the fluid-structure interface. It appears that the penalty-duality algorithm applied to this quasi-static model without inertia terms is very efficient compared to other penalty methods. Furthermore it leads to a much better condition number of the numerical scheme as far as it not necessary to use a large value of the penalty coefficient. 

In a second step we compared the various possibilities discussed in the first step on two very simple 1D fluid-structure models. The first one is the advection-diffusion equation with a prescribed movement of an immersed structure and the second one is a similar model but with Burgers equation for the fluid. In this case there a convection term (linear for the advection and non linear for Burgers equation). The numerical tests show that the $L^2$ penalty term inside the structures is very efficient as far as the penalty parameter is very small. The coupling with the penalty-duality algorithm involving the velocity continuity at the interface enables to improve very slightly the numerical tests as far as the inner penalty parameter is very small. But it is much more efficient if this parameter is more moderate. This can be therefore an interesting improvement in cases where the condition numerber is an important point in the numerical scheme. One can forcast that this is mainly the case for for 2D and more for 3D models.

Because this paper is mainly theoretically oriented, we focused on the mathematical analysis of the penalty and penalty-duality methods for Stokes equations coupled with a rigid structure in movement. The tests are just an illustration in order to eillustrate our purpose and to point out the limits of our conclusions.  For a more physical study we refer to A. Falaise and E. Liberge \cite{elaf}.

Nevertheless, an embedding in a global Euler-Lagrange representation is a classical solution in order to be able to solve a fully time dependent model where the structure (assumed to be rigid for sake of simplicity) is moving inside the fluid with large displacements due to the forces applied on the structure by the fluid. This aspect will be discussed in a forthcoming study \cite {PHDEL2} where we use the so-called $\theta$-method \cite{PHDAER} for transferring informations between the fluid model and the one of the structure which can be flexible and moving with large displacements.

\end{document}